\documentclass{article}
\usepackage[utf8]{inputenc}
\usepackage[margin=1.0in]{geometry}
\usepackage[parfill]{parskip}
\usepackage{bm}
\usepackage{mathtools}
\usepackage{amssymb}
\usepackage{wrapfig}
\usepackage{xcolor}
\usepackage{graphicx} 
\usepackage[sorting=none]{biblatex}
\usepackage{verbatim}
\usepackage{xcolor}
\usepackage{booktabs,floatrow}
\usepackage{enumerate}
\usepackage{amsthm}
\usepackage{amsmath}
\usepackage[english]{babel}
\usepackage{comment}
\usepackage{pdflscape}
\usepackage{afterpage}
\usepackage[normalem]{ulem}
\usepackage[ruled,vlined]{algorithm2e}
\usepackage{csquotes}

\makeatletter
\def\thm@space@setup{%
  \thm@preskip=\parskip
  \thm@postskip=\thm@preskip 
}
\makeatother

\bibliography{bib.bib}

\theoremstyle{definition}
\newtheorem{theorem}{Theorem}[section]

\newtheorem{lemma}[theorem]{Lemma}
\newtheorem{proofpart}{Part}
\makeatletter
\@addtoreset{proofpart}{theorem}
\makeatother

\def\sqrtlimspace#1{%
  \begingroup
    \sbox0{$#1$}
    \def\underbrace##1_##2{##1}
    \sbox2{$#1$}
    \dimen0=\wd0 \advance\dimen0-\wd2
    \mathrlap{\sqrt{\phantom{\displaystyle#1}\kern\dimen0 }}
    \hphantom{\sqrt{\vphantom{\displaystyle#1}}}
  \endgroup
  #1}

\title{Fast increased fidelity approximate Gibbs samplers for Bayesian Gaussian process regression}
\author{Kelly R. Moran and Matthew W. Wheeler}
\date{} 

\begin{document}

\maketitle

\begin{abstract}
    Gaussian processes (GPs) are common components in Bayesian non-parametric models. Their use is supported by efficient sampling algorithms, a rich methodological literature, and strong theoretical grounding.  However, due to their prohibitive computation and storage demands, the use of exact GPs in Bayesian models is limited to problems containing at most several thousand observations. Computational and storage bottlenecks arise when sampling the GP. Sampling requires a matrix inversion and the Cholesky factorization of the conditional covariance matrix; these operations scale at $\mathcal{O}(n^3),$ where $n$ is the number of unique inputs. Storage of individual matrices scales at $\mathcal{O}(n^2),$ and can quickly overwhelm the resources of most modern computers for larger problems. To overcome these bottlenecks, we develop a sampling algorithm using $\mathcal{H}$ matrix approximation of the matrices comprising the GP posterior covariance. These matrices can approximate the true conditional covariance matrix within machine precision and allow for sampling algorithms that scale at $\mathcal{O}(n\hspace{1mm}\mbox{log}^2 n)$ time and storage demands scaling at $\mathcal{O}(n\hspace{1mm}\mbox{log}\hspace{1mm}n).$ We also describe how these algorithms can be used as building blocks to model higher dimensional surfaces at $\mathcal{O}(d n\hspace{1mm}\mbox{log}^2 n)$, where $d$ is the dimension of the surface under consideration, using tensor products of one-dimensional GPs. Though various scalable processes have been proposed for approximating Bayesian GP inference when $n$ is large, to our knowledge, none of these methods show that the approximation's Kullback-Leibler divergence to the true posterior can be made arbitrarily small and may be no worse than the approximation provided by finite computer arithmetic. In what follows, we describe $\mathcal{H}-$matrices, give an efficient Gibbs sampler using these matrices for one-dimensional GPs, offer a proposed extension to higher dimensional surfaces, and investigate the performance of this fast increased fidelity approximate GP, FIFA-GP, using both simulated and real data sets.
\end{abstract}

\section{Introduction}

Gaussian processes (GPs) provide flexible priors over smooth function spaces and are widely used in Bayesian non-parametric modeling. Though such models have strong theoretical underpinnings (e.g., see \cite{rasmussen2006} and references therein), they are limited in use due to the computational complexity required to sample from the posterior distribution, which grows at cubic time complexity and quadratic storage requirements in relation to the number of unique observations. For  $n$ observations, estimation requires computing the inverse, determinant, and square-root decomposition of the $n \times n$ covariance matrix. These operations scale at $\mathcal{O}(n^3)$ with $\mathcal{O}(n^2)$ memory requirements. When using Gibbs sampling, computational issues are compounded as reliable inference depends on multiple samples being drawn from the posterior distribution. In practice, the use of GPs is limited to data sets containing at most several thousand observations.

A large literature exists on scalable GPs. Common methods for approximating the likelihood from the full GP include (a) using a subset of the training data \cite{herbrich2003fast, chalupka2013framework}; (b) introducing a grid of inducing points \cite{williams2001using, quinonero2005unifying,  titsias2009variational, hensman2013gaussian} to approximate the covariance matrix using a low-rank matrix that is scaled further using structure exploiting algebra \cite{wilson2015kernel, wilson2015thoughts}; (c) relying on local data to make predictions in a given region \cite{gramacy2008bayesian, nguyen2009local, datta2016nearest}. Methods based on (a) or (b) struggle to capture local structure in the data unless the subset size or number of inducing points approach the number of training points. Methods based upon (c) tend to capture local structure well but may struggle with global patterns; additionally, they do not offer an approximation to the global covariance. \cite{saibaba2012efficient, saibaba2015fast} suggest methods for approximating posterior realizations of a function conditional on observed data, but do not address hyperparameter uncertainty or quantify the distributional similarity between the true and approximating posterior. The bulk of existing approximation approaches do not address the issue of unified parameter estimation and function interpolation. Further, the majority of the extant literature focuses on fast approximate methods for computing the inverse of a matrix, but does not address estimating the square-root of this inverse, which is required to sample from the process. For a thorough review of these approaches, see \cite{liu2018gaussian}. 

Though the majority of the above methodologies are used in a frequentist context for point estimation, explicitly Bayesian models have been proposed to efficiently approximate samples from a GP while accounting for hyperparameter uncertainty for large $n$. \cite{sang2012full} uses a combination of inducing point and sparse kernel methods to capture both local and global structure. Global structure is captured via a reduced rank approximation to the GP covariance matrix based on truncating the Karhunen–Lo{\`e}ve expansion of the process and solving the resulting integral equation using the Nystr{\"o}m method, and residual local structure is captured via a tapered kernel. Approximation fidelity is a function of taper length and number of inducing points. Computation using this method scales at $\mathcal{O}(nm^2 +nk^2)$, where $m$ is the number of inducing points and $k$ is the average number of non-zero entries per row in the approximated covariance matrix. For high fidelity approximations, or in higher dimensions, the size of $m$ and $k$ may need to approach $n$ to reach the desired accuracy. This leads to a sampling algorithm with quadratic complexity.

As an alternative to this approach, \cite{banerjee2012efficient} allows for increased computational efficiency and a probabilistic bound on the Frobenius norm between the approximated and true covariance matrix. This approach constructs a linear projection of the data onto a lower dimensional subspace using a stochastic basis construction \cite{sarlos2006improved, halko2009finding}. The $m-$dimensional subspace in this random projection method and that of the inducing point method in \cite{sang2012full} are equivalent when the projection matrix is chosen to be the first $m$ eivenvectors of the SVD of the GP covariance matrix. The construction of these approximate matrices may scale poorly as $n$ increases and increases the overall computation time even though the sampling algorithm is efficient. Finally, the algorithm still scales at $\mathcal{O}(m^3),$ where $m$ is the dimension of the subspace approximation, and in higher dimensions $m$ tends to approach $n,$ resulting in no computational savings.

A Bayesian approach that does offer significant computational savings is the hierarchical nearest-neighbor approach \cite{datta2016hierarchical} which has $\mathcal{O}(n\hspace{1mm}\mbox{log}\hspace{1mm}n)$ complexity.  This method serves as a sparsity-inducing prior allowing for scalability to previously unimplementable data sizes; however, it  sacrifices the ability to retain fidelity to a true (non-sparse) covariance matrix by forcing correlations between the majority of observations to be exactly $0,$ guaranteeing the approximating covariance matrix will diverge from the true covariance matrix.

As an alternative to the above mentioned matrix approximation methods, $\mathcal{H}$-matrices \cite{hackbusch1999sparse,grasedyck2003construction} provide efficient matrix compression techniques, that allow for quasi-linear  approximate solutions to linear systems of equations, determinant computation, matrix multiplication, and Cholesky-like square root matrix decomposition. The computational complexity of a given operation can vary depending on the type of $\mathcal{H}$-matrix and the particular algorithm, but those considered here, and the majority of the $\mathcal{H}-$matrix approaches defined in the literature,  have cost at most $\mathcal{O}(n\hspace{1mm} \mbox{log}^2 n)$. $\mathcal{H}$-matrix techniques combine methods similar to \cite{datta2016hierarchical} and \cite{banerjee2012efficient} to accurately approximate a covariance matrix to near machine precision by decomposing the matrix as a tree of partial and full rank matrices. 

$\mathcal{H}$-matrices have previously been used for likelihood approximation and kriging in large scale GPs \cite{ambikasaran2014fastdirect, litvinenko2019likelihood, geoga2019scalable}, Gaussian probability computations in high dimensions \cite{cao2019hierarchical, genton2018hierarchical}, and as a step in approximate matrix square root computation using Chebyshev matrix polynomials for conditional realizations \cite{saibaba2012efficient}. Despite these advances, from a Bayesian perspective, the benefits of the $\mathcal{H}-$matrix formulation do not directly translate into efficient Gibbs sampling algorithms. This is because the posterior conditional distribution has a covariance matrix that is comprised of multiplicative component pieces and not directly amenable to $\mathcal{H}-$matrix approximation (i.e., we can't approximate the posterior covariance using one single $\mathcal{H}-$matrix; see equation \ref{eq:gp_posterior}). To overcome this difficulty, we propose a sampling algorithm that does not require direct computation (approximate or otherwise) of the posterior covariance matrix, but instead uses the properties of $\mathcal{H}-$matrices to sample from the posterior conditional distribution. This algorithm is efficient, providing scalable near exact GP approximations for $d=1.$ 

Construction of these matrices bottlenecks when the surface under study has dimension $d > 1$. In cases where $d > 1$, we approximate the surface using a tensor product of $d$ 1-dimensional GPs similar to \cite{wheeler2019bayesian}. This results in a process that can approximate a surface with high fidelity by relying on $d$ 1-dimensional functions that, given the right covariance kernel, is not constrained by defining inducing points typical to spline tensor product methods \cite{deboor1978practical, eilers2010splines}. The method, which we call fast increased fidelity approximate GP (FIFA-GP), involves approximations of the GP covariance using $\mathcal{H}$-matrices, requiring $d$ $\mathcal{O}(n\hspace{1mm}\mbox{log}^2n)$ matrix factorizations for the tensor product formulation. We develop a novel sampling step for the bases of the function posterior with computational cost $\mathcal{O}(d\hspace{1mm} n\hspace{1mm}\mbox{log}^2\hspace{1mm}n)$.

When $d=1,$ we show the proposed method has a bound on the Kullback–Leibler divergence between the approximated and true conditional \textit{posterior} (competing methods only provide such a bound on the \textit{prior}) that, in many cases, can be increased to machine precision. Further when $d > 1$, we show how the method defines a stochastic process with an infinite functional basis based upon the individual GP covariance kernels. The methodology has the ability to run problems with $n$ on the order of $10^5$ on a local machine. Additionally, \textbf{\textsf{R}} code that can be used to reproduce all examples in this paper is included in the Supplementary Materials.

The manuscript is structured as follows: Section \ref{sec:background} gives an overview of Gaussian process modeling and describes existing methods for approximate Bayesian computation. Section \ref{sec:h_matrices} describes the matrix decomposition method used for the proposed GP approximation. Section \ref{sec:posterior} offers proofs for the approximation fidelity of the posterior and details the sampling algorithm. We illustrate the relative performance of all methods alongside a true GP in Section \ref{sec:simulation} with a simulation study having small $n$ and then compare performance of just approximate methods using a large $n$ simulation. Section \ref{sec:real_data} provides timing and performance results using real data. Finally, section \ref{sec:discussion} discusses possible applications and extensions of this approximation to other high-dimensional Bayesian settings.

\section{Gaussian process models}\label{sec:background}

We provide an overview of Gaussian process models and detail the computational bottlenecks associated with their use in Bayesian samplers. We also review existing approaches for approximating the GP posterior, pointing out drawbacks with each to motivate our approach.

\subsection{Gaussian process regression}

Suppose one observes $n$ noisy realizations of a function $f:\mathbb{R}^d \rightarrow \mathbb{R}$ at a set of inputs $\{\bm{x}_i \in \mathbb{R}^d\}_{i=1}^n$. Denote these observations  $\{y_i(\bm{x}_i)\}_{i=1}^n,$ and for simplicity of exposition, assume the inputs are unique. In the classic regression setting, error is assumed independent and normally distributed:
\begin{align}\label{eq:regression}
y_i=f(\bm{x}_i)+e_i, \quad e_i \sim \text{N}(0,\tau^{-1}), \quad i=1,\ldots,n.
\end{align}
When $f$ is an unknown function it is common to assume $f(\cdot) \sim GP(m(\cdot),k(\cdot,\cdot))$, a Gaussian process with mean function $m(\cdot)$ and covariance function $k(\cdot,\cdot).$ This process is completely specified by the mean and the covariance function. \textit{A priori} the mean function is frequently taken to be zero, with behavior of the process defined by the covariance function. This function, $k(\cdot,\cdot),$ along with its hyper-parameters, $\Theta$, specify properties of $f$ (e.g., for the exponential covariance kernel realizations of $f$ are smooth and infinitely differentiable);  the hyper-parameters $\Theta$ further specifies these properties (e.g. how quickly covariance between points decays or how far the function tends to stray from its mean).

Letting $f$ be a zero centered GP with $k(\cdot,\cdot)$ being any symmetric covariance kernel (e.g., squared exponential or Mat\'ern) having hyper-parameters $\Theta$, one specifies a nonparametric prior over $f$ for the regression problem given in  (\ref{eq:regression}). As an example, one may define $k(\cdot,\cdot)$ using the squared exponential covariance kernel, parameterized as $k(\bm{x}_i,\bm{x}_j)= \sigma_f^2 \  \text{exp}\big[ -\frac{1}{2}(\bm{x}_i-\bm{x}_j)'\Omega (\bm{x}_i-\bm{x}_j) \big]$
with $\Omega=\text{diag}(1/\rho_1^2,\ldots,1/\rho_d^2)$, where $\Theta=(\sigma_f^2, \rho_1^2,\ldots,\rho_d^2)$. 

In general for $n$ observations, define $K_{nn}(\Theta)$  to be the covariance matrix  for the $n$ observed inputs using $k(\cdot,\cdot).$ 
That is, $K_{nn}(\Theta)$ is such that $[K_{nn}(\Theta)]_{i,j} = k(\bm{x}_i,\bm{x}_j)$ for $i,j \in \{1, \ldots, n\}.$ Let $Y=[y_1,\ldots,y_n]'$ be the vector of observed noisy realizations of $f.$ The log-likelihood function is then 
\begin{align}\label{eq:log_likelihood}
    \ell_n(\Theta,\tau) = -\frac{n}{2}\ \text{log}(2\pi) - \text{log}\big[ \text{det}(K_{nn}(\Theta) + \tau^{-1}\mathbf{I}_n) \big] -\frac{1}{2}\ Y^T (K_{nn}(\Theta) + \tau^{-1}\mathbf{I}_n)^{-1} Y.
\end{align}
Computing this log-likelihood requires calculating the determinant and inverse of the $n\times n$ matrix $K + \tau^{-1}\mathbf{I}_n$, which are both $\mathcal{O}(n^3)$ operations. For Bayesian computation, the computational bottlenecks increase. Defining $\bm{f}=\{f(\bm{x}_i)\}_{i=1}^n$ and  $K = K_{nn}(\Theta)$, the posterior of $\bm{f}$ conditional on the hyper-parameters is known, i.e., 
\begin{equation}
\begin{split}
\bm{f}|\bm{y},X,\Theta,\tau &\sim \text{N}(\mu_{f}, \Sigma_{f}), \\
\Sigma_{f} &= K - K(K + \tau^{-1}\mathbf{I}_n)^{-1}K = K(\tau K + \mathbf{I}_n)^{-1}, \\
\mu_{f} &= K(K + \tau^{-1}\mathbf{I}_n)^{-1}\bm{y} = \tau \Sigma_{f} \bm{y}.
\label{eq:gp_posterior}
\end{split}
\end{equation}

Calculating the covariance in (\ref{eq:gp_posterior}) requires the inversion of $(\tau K + \mathbf{I}_n)$ followed by the matrix multiplication of $K$ and $(\tau K + \mathbf{I}_n)^{-1}$, which are both $\mathcal{O}(n^3)$ operations. Calculating just the mean in (\ref{eq:gp_posterior}) requires either the inversion of $(K + \tau^{-1} \mathbf{I}_n)$ followed by two matrix-vector products (an $\mathcal{O}(n^3)$ followed by two $\mathcal{O}(n^2)$ operations) or first calculating the variance in (\ref{eq:gp_posterior}) followed by a matrix-vector product (two $\mathcal{O}(n^3)$ operations followed by one $\mathcal{O}(n^2)$ operation).

For Markov Chain Monte Carlo algorithms,  assuming prior distributions are assigned to $\Theta$ and $\tau,$ inference proceeds by sampling from  $p(\Theta,\tau,\bm{f}|\bm{y},X),$ which requires evaluating (\ref{eq:log_likelihood}) and (\ref{eq:gp_posterior}) a large number of times. In addition to these evaluation, the Cholesky decomposition of the matrix product is required for sampling from the posterior of $\bm{f}$ at each iteration. This is also an $\mathcal{O}(n^3)$ operation. The computational cost of each sample limits the use of Bayesian estimation of the full GP to problems having at most $5000$ observations on most single processor computers, and $5000$ observations is generous. It is the authors' experience that  $1500$ is a more reasonable upper bound for most problems.

\subsection{Bayesian approximation methods}


Given the difficulties of evaluating (\ref{eq:log_likelihood}) and (\ref{eq:gp_posterior})  for large $n,$ a number of approximation methods have been developed.
The approach of \cite{sang2012full} combines the reduced rank process of \cite{banerjee2008gaussian} with the idea of covariance tapering in order to capture both global and local structure. Specifically, $f(\bm{x})$ is decomposed as $f(\bm{x}) = f_{global}(\bm{x}) + f_{local}(\bm{x}),$ where $f_{global}(\bm{x})$ is the reduced rank approximation from \cite{banerjee2008gaussian} and $f_{local}(\bm{x}) = f(\bm{x}) - f_{global}(\bm{x})$ is the residual of the process after this global structure has been accounted for. The covariance function of the residual $f_{local}(\bm{x})$ is approximated using a tapered covariance kernel, i.e. a kernel in which  the covariance is exactly 0 for data at any two locations whose distance is larger than the specified taper range. The full-scale method has improved performance relative to the reduced rank process or covariance tapering alone and has the desired quality of capturing global and local structure. However, the quality of this approximation to the original function is highly dependent on the choice of taper length and number of inducing points and there is no way to constructively bound the error between realizations of the approximate and true covariance matrices.

The compression approach of \cite{banerjee2012efficient} approximates the covariance matrix $K$ by $K^{\text{lp}}=(\Phi K)^T (\Phi K \Phi^T)^{-1} \Phi K$, where $\Phi$ is a projection matrix. The ``best'' rank-$m$ projection matrix in terms of $||\cdot||_F$ and $||\cdot||_2$ is the matrix of the first $m$ eigenvectors of $K.$ Because finding the spectral decomposition of $K$ is itself an $\mathcal{O}(n^3)$ operation, the two algorithms in \cite{banerjee2012efficient} focus on finding near-optimal projection approximations. The second of the proposed algorithms in the paper has the advantage of a probabilistic bound on the Frobenius norm between the projection approximation and true covariance matrix. That is, one can choose algorithm settings s.t. $\Pr(||K - \Phi^T\Phi K||_F < \varepsilon) = p$ for some desired probability $p.$ However, it is iterative; its use requires expensive pre-computation steps that scale poorly as $n$ increases with no defined order of computational complexity in this pre-computation phase. 

The hierarchical nearest-neighbor GP \cite{datta2016hierarchical} is introduced as a sparsity-inducing prior that allows for fully Bayesian sampling with scalability to previously unimplementable data sizes. This prior introduces a finite set of neighborhood sets having block sparsity structure, and defines a relation between these neighborhood sets via a directed acyclic graph. The distribution at any new points is then expressed via nearest neighborhood sets, yielding a valid spatial process over uncountable sets. While this prior is shown to yield similar inference to the full GP, the choice of a sparse prior means the ability to retain fidelity to a true (non-sparse) covariance matrix is sacrificed. Thus, although inference in simulation using the true and approximated posterior appear similar in numerical simulations, there is no bound on their divergence.

The $\mathcal{H}$-matrix approximations used in FIFA-GP for approximating $K$ generalize block-diagonal, sparse, and low-rank approximations to the underlying covariance matrix \cite{litvinenko2019likelihood}. The first layer of the FIFA-GP approximation to $K$ is analogous to the NNGP with neighbor sets defined to be those points located in the same dense diagonal block. These off diagonal blocks can be estimated using parallel random projection approaches to enable fast decomposition into low-rank approximations, which are similar to the projection algorithms proposed for the full matrix in the compressed GP model. Additionally, the unique structure of the $\mathcal{H}$ matrix composition allows for fast computation of the determinant, the Cholesky, and the inverse, which are required for an MCMC algorithm, but cannot be used directly to sample from (\ref{eq:gp_posterior}). Here, one of the key contributions of this manuscript is showing how to combine standard $H$-matrix operations to draw a random vector with distribution defined in (\ref{eq:gp_posterior}).

Using FIFA-GP, local structure is captured due to the dense block diagonal elements. At the same time, global structure is preserved via high-fidelity off-diagonal compression.  Furthermore, the approximation error between the true and approximate covariance matrices is bounded when constructing the $\mathcal{H}-$matrix. This allows for a bound on the KL-divergence of the posterior for the GP, and this bound has not been shown in previous methods.


\section{Covariance approximation using $\mathcal{H}$-matrices}\label{sec:h_matrices}

In this section we describe how low rank approximations and tree-based hierarchical matrix decompositions enable fast and accurate computations for Gaussian process.

\subsection{Low-rank matrices}

The rank of a matrix $A \in \mathbb{R}^{m \times n}$ is the dimension of the vector space spanned by its columns. Intuitively, the rank can be thought of as the amount of information contained in a matrix. The matrix $A$ is full rank if $\text{rank}(A) = \min(m,n);$ such a matrix contains maximal information for its dimension. For any rank $p$ matrix $A$ it is possible to write $A=UV^T$ where $U \in \mathbb{R}^{m \times p}$ and $V^T \in \mathbb{R}^{p \times n}$, then it is only necessary to store $\mathcal{O}(\max(m,n)p)$ values rather than $\mathcal{O}(mn)$ values. For GP covariance matrices having $n$ observations the savings gained by relaxing the $\mathcal{O}(n^2)$ memory requirements can be quite significant (see SI Section \ref{SI_cost_inputdim} for a concrete example).

It is also possible to approximate full-rank $A$ with some rank-$p$ matrix $A_p$ so that $A \approx A_p = UV^T.$ If $p \ll \text{rank}(A)$ significant computational gain can be achieved, but approximation fidelity depends on how much information is lost in the representation of $A$ as some lower-rank matrix. With most low-rank factorization methods the approximation error decreases as $p$ approaches $\text{rank}(A)$. 

\subsection{Hierarchical matrices}

A hierarchical matrix ($\mathcal{H}$-matrix) is a data-sparse approximation of a non-sparse matrix relying on recursive tree-based sub-division. The data-sparsity of this approximation depends on how well sub-blocks of the original matrix can be represented by low-rank matrices, and whether the assumed tree structure aligns with these potentially low-rank blocks. For dense matrices having suitably data-sparse $\mathcal{H}$-matrix approximations, many dense linear algebra operations (e.g. matrix-vector products, matrix factorizations, solving linear equations) can be performed stably and with near linear complexity since they take place on low-rank or small full-rank matrices. Storage gains are made when the decomposed low-rank blocks can be stored rather than the original full dense matrix blocks. The construction of and algorithms for $\mathcal{H}-$matrices are vast topics and outside of the scope of this manuscript; for further information on $\mathcal{H}$-matrices we refer the reader to \cite{hackbusch1999sparse,hackbusch2015hierarchical,grasedyck2003construction}.

A key advantage to using $\mathcal{H}$-matrices for matrix approximation is that the error between the true matrix and the approximate matrix can be bounded \textit{by construction}. Specifically, the max-norm (i.e. the maximum elementwise absolute difference) can be made arbitrarily small for the type of $\mathcal{H}-$matrix used in FIFA-GP.

The ability for a given matrix to be well approximated by an $\mathcal{H}$-matrix (and thus the potential for computational and storage gains) depends on the structure of that matrix. With a GP covariance matrix, the relationship between rows and columns of the matrix depends on the ordering of the individual data points. In one dimension, a simple direct sorting of the points will lead to the points closest in space being closest in the matrix. In multiple dimensions, clustering algorithms can be used to group similar data points. Alternatively, fast $\mathcal{O}(n\hspace{1mm}\mbox{log}\hspace{1mm} n)$ sorting via a $kd$-tree can be used (e.g. \cite{wald2006building}), in which the data are sorted recursively one dimension at a time.

There are many types of $\mathcal{H}$ matrices\cite{hackbusch1999sparse,borm2002data,ambikasaran2013mathcal,hackbusch2015hierarchical,grasedyck2003construction}. Each can facilitate posterior computation using the methods outlined below.
For our exposition, we focus on Hierarchical Off-Diagonal Low Rank (HODLR) matrices, a type of $\mathcal{H}$-matrix, due their fast construction and ease of use. HODLR matrices are defined recursively via $2 \times 2$ block partitions. Off-diagonal blocks are approximated via low rank representations and diagonal blocks are again assumed to be HODLR matrices. The recursion ends when the diagonal blocks reach some specified dimension at which the remaining diagonal block matrices are dense. The level of a HODLR matrix refers to the number of recursive partitions performed. Storage cost of off-diagonal blocks is reduced by storing the components of the low-rank decomposition rather than the original matrix elements. Compared to general $\mathcal{H}$-matrices, HODLR matrices may be faster in practice, but have the same time complexity for the same algorithm. As such, it can have faster decomposition and solve algorithms for our purposes; however, they sacrifice the flexibility of allowing high-rank off-diagonal blocks and adaptive matrix partitionings. HODLR matrices for Gaussian processes defined over dimensions greater than one may be difficult to construct in practice. 

\begin{figure}[htp]
    \centering
    \includegraphics[width=0.8\textwidth]{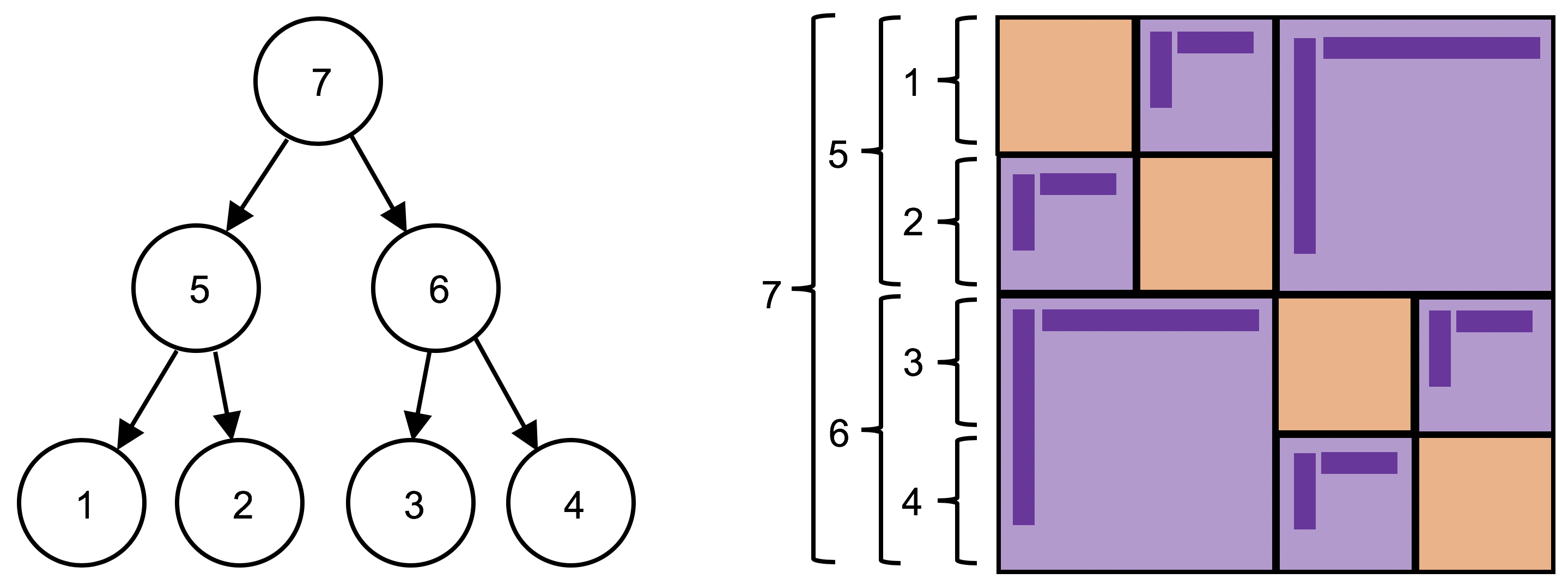}
    \caption{Example partition tree and associated partitioning of a 2-level HODLR matrix. The diagonal blocks are full rank, while the off-diagonal blocks can be represented via a reduced rank factorization.}
    \label{fig:hodlrmatrix}
\end{figure}

HODLR matrices, their factorization, and their use in GP likelihood estimation and kriging have been thoroughly discussed in \cite{ambikasaran2013mathcal, ambikasaran2014fastsymmetric}. Open source code for constructing, factorizing, and performing select linear algebra operations using HODLR matrices is available in the \texttt{HODLRlib} library \cite{ambikasaran2019hodlrlib}. Of relevance to our work are fast symmetric factorization of symmetric positive-definite HODLR matrices, determinant calculation, matrix-vector multiplication, and solver. Specifically, symmetric positive-definite HODLR matrix $A$ can be symmetrically factored into $A=WW^T$ in $\mathcal{O}(n\hspace{1mm}\mbox{log}^2 n)$ time \cite{ambikasaran2014fastsymmetric}. An example is included in SI Section \ref{sec:symm_fac_alg}.  Further details on these algorithms are available in \cite{ambikasaran2013mathcal, ambikasaran2014fastsymmetric}..

\begin{figure}[htp]
    \centering
    \includegraphics[width=0.8\textwidth]{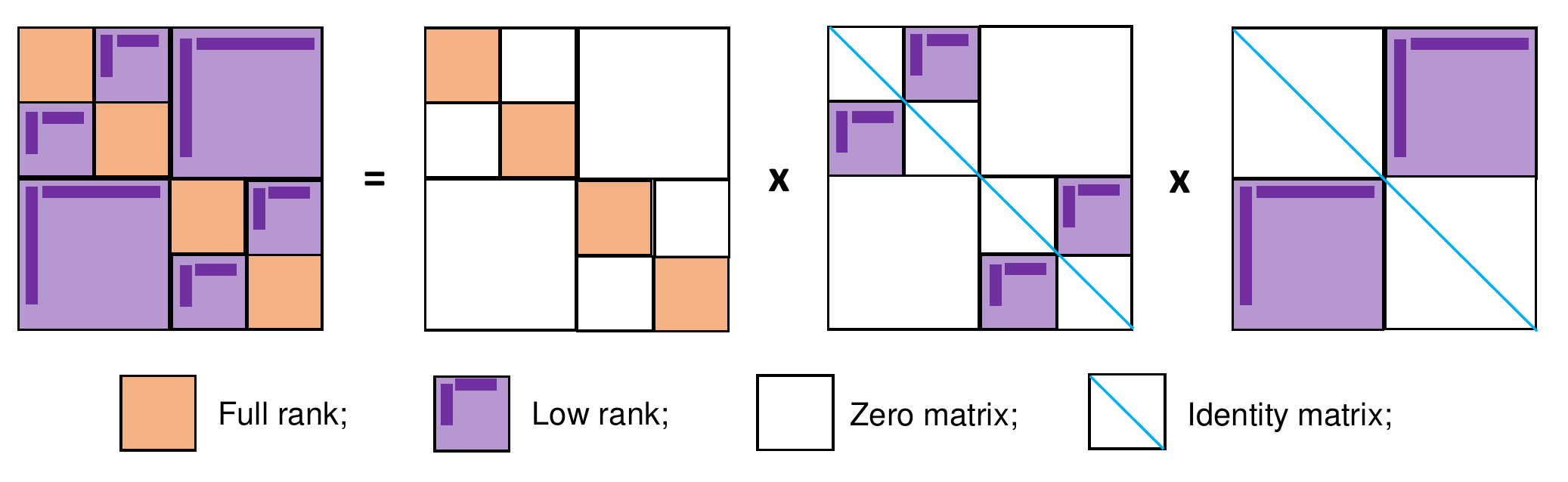}
    \caption{Factorization of a 2-level HODLR matrix.}
    \label{fig:hodlrmatrix_fac}
\end{figure}

Assuming points proximal in space are near each other in the covariance matrix, a major benefit of the HODLR decomposition is that local structure is captured with high fidelity due to the dense diagonal blocks in the approximate covariance matrix. By construction, these diagonal blocks are perfectly preserved. Furthermore, global structure is not ignored, as in covariance tapering methods, but rather approximated via low-rank representations of the off-diagonal blocks. The neighborhood sets of the hierarchical nearest-neighbor GP \cite{datta2016hierarchical} can be thought of similarly as the block diagonal entries (although these neighborhoods need not be all of the same size), but with relationships between neighborhood sets providing an update to this block-sparse structure rather than the subsequent matrix factors in Figure \ref{fig:hodlrmatrix_fac}. 


\section{Bayesian fast increased fidelity approximate GP algorithm}\label{sec:posterior}

\subsection{Gibbs sampler}

Consider the case where $d=1.$ Bayesian GP regression requires estimating the unknown hyperparameters of the covariance kernel. As these parameters vary based upon the kernel chosen, we only consider the squared exponential kernel, i.e., $k(x_i,x_j)= \sigma_f^2 \ \text{exp}( -\rho||x_i-x_j||^2)$, with extensions to other covariance kernels being straightforward. Under this assumption,  let $K_{\sigma_f,\rho}$ denote the $n \times n$ matrix computed from the squared-exponential kernel having hyperparameters $\sigma_f$ and $\rho$, and  evaluated at $\{x_i\in\mathbb{R}\}_{i=1}^n$. 

Let $y \sim N(f,\tau^{-1}),$ with conditionally conjugate priors specified for the parameters, i.e., take $\tau \sim \text{Ga}(a_1/2, b_1/2), $ $1/\sigma_f^2 \sim \text{Ga}(a_2/2, b_2/2),$ and $\rho \sim \sum_{\ell=1}^r r^{-1} \delta_{s_{\ell}},$ a discrete uniform distribution over possible values for $\rho$ in $\{s_1, \ldots, s_r\}.$ Such priors are standard in the literature \cite{banerjee2008gaussian, sang2012full, datta2016hierarchical, wheeler2019bayesian}. Then exact GP regression proceeds by sampling from the following conditional distributions:
\begin{align}
\bm{f} \ | \ - \ &\sim \ \text{N}(K_{\sigma_f,\rho}(K_{\sigma_f,\rho}+\tau^{-1}I)^{-1} \bm{y}, K_{\sigma_f,\rho}(\tau K_{\sigma_f,\rho}+I)^{-1}), \label{sec:posterior_gibs1}\\
\tau \ | \ - \ &\sim \ \text{Ga}\bigg(\frac{a_1+n}{2}, \frac{b_1 + (\bm{y}-\bm{f})^T(\bm{y}-\bm{f})}{2}\bigg), \label{sec:posterior_gibs2}\\
\sigma_f^{-2} \ | \ - \ &\sim \ \text{Ga}\bigg(\frac{a_2+n}{2}, \frac{b_2 + \sigma_f^2 \bm{f}^T K_{\sigma_f,\rho}^{-1} \bm{f}}{2}\bigg),\label{sec:posterior_gibs3} \\
\Pr(\rho=s_h \ | \ -) \ &= \ c \ \{\text{det}(K_{\sigma_f,s_h})\}^{-1/2} \ \text{exp}\Big\{-\frac{1}{2}\bm{f}^T K_{\sigma_f,s_h}^{-1} \bm{f}\Big\}, \label{sec:posterior_gibs4}
\end{align}
where $c^{-1} = \sum_{\ell=1}^r \{\text{det}(K_{\sigma_f,s_{\ell}})\}^{-1/2} \ \text{exp}\Big\{-\frac{1}{2}\bm{f}^T K_{\sigma_f,s_{\ell}}^{-1} \bm{f}\Big\}$.

This sampler requires the computation of the inverse, determinant, and Cholesky decomposition, which are $\mathcal{O}(n^3)$ operations. For what follows, we use the properties of $\mathcal{H}-$matrices to replace these operations with $\mathcal{O}(n\hspace{1mm}\mbox{log}^2\hspace{1mm} n)$ counterparts to develop a fast Gibbs sampler.  In the case where a given $\mathcal{H}-$matrix has already been factorized, e.g. in the case when the covariance kernels can be reused by fixing the length-scale components on a discrete grid and pre-computing the factorization, inverse and determinant computations can be done in $\mathcal{O}(n\hspace{1mm}\mbox{log}\hspace{1mm} n)$ time using the \texttt{HODLRlib} library \cite{ambikasaran2019hodlrlib}.
The discrete prior in step (\ref{sec:posterior_gibs4}) allows for the relevant matrix inverse and determinant to be computed at every grid point before running the Gibbs sampler so as not to have to recompute these values in each step. In practice, this step could be replaced with a Metropolis step for even faster computation when the number of possible values for $\rho$ is large. The Gaussian error sampler can be extended to non-Gaussian errors by defining an appropriate relation between a latent GP and the observed data, e.g., via a probit link for binary data \cite{choudhuri2007nonparametric} or a rounding operator for count data \cite{canale2013nonparametric}.

\subsubsection{Sampling $\bm{f}$}

Consider step (\ref{sec:posterior_gibs1}), which requires computation of the posterior mean
\begin{align*}
    K_{\sigma_f,\rho}(K_{\sigma_f,\rho}+\tau^{-1}I)^{-1} \bm{y}
\end{align*}
and posterior covariance,
\begin{align}
    K_{\sigma_f,\rho}(\tau K_{\sigma_f,\rho}+I)^{-1}. \label{samp:cov}
\end{align}  
Direct computation of these quantities has cubic time complexity. It is possible to use $\mathcal{H}-$matrices to quickly compute the mean \cite{ambikasaran2014fastdirect}, but it is difficult to construct (\ref{samp:cov}) in the $\mathcal{H}-$matrix format, which makes sampling from the standard algorithm difficult.


We propose a sampling algorithm for the GP function posterior that leverages the near linear HODLR operations (specifically matrix-vector products, solutions to linear systems, and applications of symmetric factor to a vector) rather than direct matrix multiplication or inversion. Let $K = K_{\sigma_f,\rho}$ for notational convenience. We approximate the matrices $K$ and $M=\tau K + I$ in (\ref{sec:posterior_gibs1}) by $\mathcal{H}$-matrices $\tilde{K}$ and $\tilde{M}$ respectively. In what follows, we use the relation $K(\tau K + \mathbf{I})^{-1} = (\tau K + \mathbf{I})^{-1}(\tau K^2 + K) (\tau K + \mathbf{I})^{-1}$ to develop our sampling algorithm.


\begin{algorithm}[H]
\SetAlgoLined
\KwResult{Produce an approximate draw from $p(\bm{f}|\bm{y}, X, \Theta)$ using factorizations of cost $\mathcal{O}(n\log^2{}n)$, and matrix-vector product and solve operations
of cost $\mathcal{O}(n\log{}n)$}. 
\KwIn{\\ \qquad \quad $\bm{y}$: Noisy observation of $\bm{f}$.\\
         \qquad \quad $\epsilon$: specified tolerance.\\
         \qquad \quad $B$: Maximum block size.\\}
         \qquad \quad $\{\Theta,\tau\}$: Hyper-parameters.\\
 \eIf{$\tau < 1$}{
    $\epsilon^\ast = \tau \epsilon$\;
   }{
   $\epsilon^\ast = \epsilon$ \;\label{alg1_tol}
  }
 Construct: $\tilde{K} \approx  K$ with  tolerance $\epsilon^\ast/\tau$ using  $\{B,\Theta,\tau\}$ \tcp*{$\mathcal{H}$-matrix construction.}
 Construct: $\tilde{M} \approx  \tau K + I$ with tolerance $\epsilon^\ast$ using $\{B,\Theta,\tau\}$ \tcp*{$\mathcal{H}$-matrix construction.}
 Construct: $W$, such that $\tilde{K}=WW^T$    \tcp*{$\mathcal{H}$-matrix symmetric factorization.}
 Factorize: $\tilde{M}$ for later $\mathcal{H}$-matrix operations \tcp*{$\mathcal{H}$-matrix factorization.}
 Sample:     $\bm{a}, \bm{b} \sim \text{N}(0,I)$ \;
 Let:  $Z = \sqrt{\tau} \tilde{K} \bm{a} + \tilde{W} \bm{b}$ \tcp*{$Z \sim N(0, \tau \tilde{K}^2 + \tilde{K})$}\label{alg1_1}
 Solve: $\tilde{M} w = Z$ \tcp*{$w \sim N(0,\tilde{K}[\tau\tilde{K}+I]^{-1}])$}\label{alg1_2}
 Solve: $\tilde{M} r = \tau \bm{y}$ \tcp*{r = ($\tau\tilde{K}+I)^{-1}\tau \bm{y}$}
 \KwRet{$Z^{\ast\ast} = w + \tilde{K} r$}\;
 \caption{GP sampler using $\mathcal{H}$-matrix}
\end{algorithm}


\addvspace{\baselineskip}
\begin{lemma}\label{thm:alg1}
From Algorithm 1, $Z^{**} \sim \text{N}(\tilde{\mu}_{f}, \tilde{\Sigma}_{f}),$ where $\tilde{\Sigma}_f = \tilde{K}\tilde{M}^{-1}$ and $\tilde{\mu}_f = \tau \tilde{\Sigma}_f \bm{y}$ are defined to be the approximations of the posterior function variance $\Sigma_f$ and mean $\mu_f,$ respectively.
\end{lemma}

\begin{proof}
See Appendix.
\end{proof} 

The most expensive steps in Algorithm 1 are the $\mathcal{H}$-matrix factorizations of the matrices $\tilde{K}$ and $\tilde{M}$. When using a fixed grid of length-scale parameters, the factorization of $\tilde{K}$ can be pre-computed. If the precision $\tau$ is also fixed (i.e., if the factorization of $\tilde{M}$ can also be pre-computed), the cost of the entire algorithm will be $\mathcal{O}(n\log{}n)$ rather than $\mathcal{O}(n\log^2{}n).$

\textbf{Remark}: There are a number of $\mathcal{H}-$matrix constructions that can be used
in this algorithm, and they may be preferable in certain situations. For example, the $\mathcal{H}^2$-matrix \cite{borm2002data} provides a faster construction for 2 and 3-dimensional inputs, and may be preferable in these cases. 


When discussing computation times above we have assumed the $\mathcal{H}$-matrix decomposition method used is the \texttt{HODLRlib} implementation of the HODLR decomposition. However, any hierarchical decomposition allowing for symmetric matrix factorization and having an $\epsilon$ bound by construction on the max-norm between the approximated and true matrices could be used instead with the associated computation cost being that of the method used.

\subsubsection{Hyperparameter posterior approximation}

Sampling (\ref{sec:posterior_gibs3}) and (\ref{sec:posterior_gibs4}) involves solving the linear system $K\bm{b}=\bm{f}$ for $\bm{b}$ (i.e., finding $K^{-1}\bm{f}$) and finding the determinant of $K$. If matrix $K$ is approximated by a $\mathcal{H}$-matrix $\tilde{K},$ significant computational savings result. The HODLR $\mathcal{H}$-matrix decomposition performs these operations at a cost of $\mathcal{O}(n\log{}n).$ Here $\tilde{K}$ is a factorized HODLR matrix, which has been computed prior to this operation (the factorization itself is an $\mathcal{O}(n\log{}^2n)$ operation).

\subsubsection{Approximation fidelity}

The following theorem concerns the approximation fidelity of $\bm{f}|\bm{y}$ when the component pieces of the GP posterior covariance $\Sigma_f=K(\tau K + \mathbf{I})^{-1}$, $K$ and $M=\tau K + \mathbf{I}$, are calculated using approximations $\tilde{K}$ and $\tilde{M},$ respectively, with maximum absolute difference between the true and approximated matrices being bounded by $\varepsilon$. The proof is general and gives an upper bound on convergence for large $n.$ The implication of this result is that there’s an upper limit to the approximation fidelity when $n$ grows past a certain point, even if the true matrices are used, because of the limits of finite computer arithmetic.

\begin{theorem}\label{thm:approx_fidel}
Let $\bm{p} \sim N(\mu_{f}, \Sigma_{f})$ where $\mu_{f} = \tau \Sigma_{f} \bm{y}$ and $\Sigma_{f}$ is an $n \times n$ positive definite matrix, with $\Sigma_{f} = K (\tau K + I)^{-1}$ for $K$ the $n \times n$ realization of some symmetric covariance kernel, $\bm{y}$ is a length-$n$ vector, and $\tau$ is a constant. Define $M = (\tau K + I)$ such that $\Sigma_{f}=KM^{-1}$. Then there exists matrices $\tilde{K}, \tilde{M} \in \mathcal{H}$ with $||K-\tilde{K}||_{\max} \leq \varepsilon$ and $||M-\tilde{M}||_{\max} \leq \varepsilon$ such that for $\tilde{\Sigma}_{f} = \tilde{K} \tilde{M}^{-1},$ $\tilde{\mu}_{f} = \tau \tilde{\Sigma}_{f} \bm{y}$, and $\bm{q} \sim N(\tilde{\mu}_{f}, \tilde{\Sigma}_{f})$ 
\begin{align}
    \mathcal{D}_{KL}(\mathcal{P}||\mathcal{Q}) = E_{\mathcal{P}} \Bigg[ \mbox{log}\Bigg( \frac{\mathcal{P}}{\mathcal{Q}} \Bigg) \Bigg] \leq  c_1 n^2 \varepsilon + c_2 n^{5/2} \varepsilon + c_3 n^3 \varepsilon^2,
\end{align}
with $\lim_{\varepsilon \to 0} \mathcal{D}_{KL}(\mathcal{P}||\mathcal{Q}) = 0,$ where the density functions of $\bm{p}$ and $\bm{q}$ are denoted by $\mathcal{P}$ and $\mathcal{Q},$ respectively. The constants $c_1,$ $c_2,$ and $c_3$ are dependent on the conditioning of $K$ and $M$. Note that using the \texttt{HODLRlib} library, $\tilde{K}$ and $\tilde{M}$ can be created and factorized in $\mathcal{O}(n\hspace{1mm}\mbox{log}^2(n))$ time.
\end{theorem}

\begin{proof}
See Appendix.
\end{proof} 

The proof relies on the assumption that $||\tilde{K}-K||_F$ and $||\tilde{M}-M||_F$ can be bounded to be arbitrarily small. The advantage to using a hierarchical matrix decomposition to approximate $K$ and $M$ is that the norm is often bounded \textit{by construction}. Specifically, for the HODLR decomposition used in this paper the max-norm is bounded on construction \cite{ambikasaran2019hodlrlib}. Therefore, the resulting $F$-norm of the difference between the HODLR approximation and the true matrix is bounded by $n^2$ times the max-norm for each matrix (i.e., for $n \times n$ matrix A, $||\tilde{A}-A||_F \leq n^2 ||\tilde{A}-A||_max$), satisfying the assumption of the proof.



\subsubsection{Additional considerations}
Smooth GPs measured at a dense set of locations tend to have severely ill-conditioned covariance matrices (Section 3.2 of \cite{banerjee2012efficient} provides an excellent discussion of this issue). An $\mathcal{H}$-matrix approximation does not necessarily improve the conditioning relative to the original dense matrix. Typically, practitioners mitigate this ill-conditioning by adding a small nugget term to the diagonal of the covariance matrix. We include this nugget term prior to $\mathcal{H}$-matrix construction.

Algorithm 1 does not require a linear solve involving the $\mathcal{H}$-matrix approximation of $K$, but it does require one involving that of $\tau K + I$. A practical tweak that improves conditioning and doesn't alter the fundamental algorithm is to scale$\bm{y}$ so as to make $\tau$ smaller, and remove this scaling factor in post processing. A smaller $\tau$ makes inverting $\tau K + I$ (and its approximation) more stable with little to no impact on the posterior surface estimates.  It is also possible to combine the ideas in \cite{banerjee2012efficient} with a given $\mathcal{H}$-matrix algorithm so that the dense block diagonals in the factorization of both $\tilde{K}$ and $\tilde{M}$ could be compressed to the desired level of fidelity in order to improve conditioning within each block.

Details on how the function posterior at new inputs can be sampled, an efficient way to handle non-unique input points, and adjustments for heteroskedastic noise are provided in Section \ref{sec:additional_considerations} of the Appendix.

\subsection{Extending to higher dimensional inputs}

Though Lemma \ref{thm:alg1} and Theorem \ref{thm:approx_fidel} are valid for any input dimension, the HODLR factorization doesn’t generally scale well with the number of input dimensions. For an example of this phenomenon, see SI Section \ref{sec:symm_fac_alg}. To take advantage of the fast factorization and linear algebra operations afforded by HODLR when $d=1$, we propose an extension of these algorithms that scales to higher dimensions at $\mathcal{O}(d n\hspace{1mm}\mbox{log}^2 n).$ We model a $d$-dimensional surface as a scaling factor times a tensor product of $1$-dimensional GPs having unit variance:
\begin{gather}
y_i = \beta \ f_1(x_{1,i}) \otimes f_2(x_{2,i}) \otimes \ldots \otimes f_d(x_{d,i}) +e_i, \nonumber \\ 
e_i \sim \text{N}(0,\tau^{-1}), \quad f_h(\cdot) \sim GP(m_h(\cdot),k_h(\cdot,\cdot)), \label{eq:tensor_prod} \\
i=1,\ldots,n, \quad h=1,\ldots,d, \nonumber
\end{gather}
where $\beta$ is the scaling factor and $k_h(\cdot,\cdot)$ is some covariance kernel with unit function variance.


Approximating surfaces using a tensor product of bases is common \cite{deboor1978practical, dierckx1984algorithms, unther1996interpolating, juttler1997surface}. Often, the tensor product takes the form of a finite basis expansion using splines. The flexibility of more traditional tensor product spline approaches carries over to this approach, but one does not have to be concerned with the choice of the knot set to achieve this flexibility (e.g. see de Jong and van Zanten \cite{de2012adaptive} for an example of this using tensor product B-splines and knot set selection). By utilizing lemma 2.1 of de Jong and van Zanten,  it is trivial to show that there exists GP tensor product specifications in the sample path of (\ref{eq:tensor_prod}) that are arbitrarily close to any $d-$dimensional H\"{o}lder space of functions having up to $r$ continuous partial derivatives $D^\alpha,$ $\alpha \leq r$.

\begin{lemma}
Let $\mathcal{C}^r$ be the H\"{o}lder space of functions having $r$ continuous partial derivatives 
and $\epsilon > 0,$ then for any $h \in \mathcal{C}^r$ there exists a tensor product GP  $f = f_1 \otimes f_2 \otimes \ldots \otimes f_d$ with sample paths such that 
such that 
\begin{align*}
   Pr( \mid\mid h - f \mid \mid_\infty \leq \epsilon ) > 0
\end{align*}
\end{lemma}

This follows from the fact that B-splines are continuous functions and that by lemma $2.1$ of \cite{de2012adaptive} there exists B-splines such that their tensor product is within an $\epsilon$ ball of any continuous surface. The result then follows because there there are a large number of covariance kernels $k(\cdot,\cdot)$ that are dense in the space of continuous functions (see for example \cite{tokdar2007posterior}). This lemma shows that (\ref{eq:tensor_prod}) provides a flexible solution to 
modeling higher dimensional surfaces. 

In practice, the actual covariance kernel used may not be adequate to model a given data-set (e.g. there may be a high prior-probability placed on sample paths that are too smooth and/or the prior over the hyperparameters too restrictive). In these cases, one tensor product may over smooth the data and miss local features of the surface; here, multiple additive tensor products may be considered to alleviate this problem. This was the strategy of \cite{wheeler2019bayesian} who used an additive sum of GP tensor products like those defined in (\ref{eq:tensor_prod}). It is also the strategy in most standard spline based approaches, and it is often used in hierarchical grid refinements for tensor product B-splines (e.g. see  \cite{juttler1997surface}); however,  the choice of an appropriate knot set for each additive component is not necessary when using FIFA GPs.  

$H-$matrices can be used to approximate the GP in equation (\ref{eq:tensor_prod}) and provide the previously discussed benefits of speed and near-machine-precision fidelity. Sampling from the function posterior for each $f_h(\cdot)$ involves heteroskedastic noise; and one must modify the KL divergence bound as well as the sampling algorithm. To see why, note that $f_h(x_{h,i}) = \frac{y_i - e_i}{\prod_{k \neq h} f_k(x_{k,i})}$. Then $\frac{y_i}{\prod_{k \neq h} f_k(x_{k,i})}$ is a noisy observation of $f_h(x_{h,i})$ with variance $\tau^{-1}/(\prod_{k \neq h} f_k(x_{k,i}))^2$.

This approach also provides additional computational advantages if the number of unique input values is smaller than $n$ in a given dimension. For example, consider observations made on a $500 \times 500$ grid of inputs. Then computation proceeds rapidly even though the total number of observations is 250,000. Such computation gains are not unique to the tensor product approach, but are also available using additive kernels \cite{duvenaud2011additive, durrande2012additive} or Kronecker based inference \cite{flaxman2015fast}; however, these approaches do not have the computational advantages of $H-$matrix arithmetic and thus still scale with cubic complexity.  


 The tensor product of GP realizations is different than a GP with a separable covariance kernel comprised of the tensor product of univariate kernels, as in \cite{karol2014small}. The latter also offers computational advantages, but has the same issues for large covariance matrices.  Finally, the tensor product approach in (\ref{eq:tensor_prod}) is one way in which univariate GPs can be used to model higher dimensional surfaces. In some cases, an additive model may suffice \cite{duvenaud2011additive, durrande2012additive}.

\section{Simulation study}\label{sec:simulation}

In this section, we report performance and timing of each method for a small-$n$ simulation using data generated from a Gaussian process and a large-$n$ simulation in which the true generating function is known but not a GP.  All calculations in this and the subsequent section are performed on a 2016 MacBook Pro with a 2.9 GHz Intel Core i7 processor. 

\subsection{Small-$n$ simulation}\label{sec:simulation_smalln}

We compare the performance of the approximate methods to that of an exact GP sampler. Of interest is both similarity of point estimates to the exact GP and fidelity of uncertainty about those estimates. The experiment with synthetic data proceeds as follows for $n \in \{100,500,1000\}$ and $n^*=50$:
\begin{enumerate}
    \item Points $x_1,\ldots,x_n$ are simulated from a $\text{N}_{[-2,2]}(0,1)$ distribution, so observed data are more concentrated about the origin. 
    \item A ``true'' function $\bm{f}^{\text{true}}$ is simulated from a Gaussian process having an exponential covariance function $k(x_i,x_j) = \sigma_f^2 \ \text{exp}(-\rho||x_i-x_j||^2)+\tau^{-1}\delta_{ij}.$ Note that the true values of $\sigma_f,$ $\rho,$ and $\tau$ are varied in each configuration.
    \item A sampler based on the exact GP covariance matrix is run to get exact Bayesian estimates of the posterior for the hyperparameters $\sigma_f,\rho,\tau,$ the function $\bm{f}$ at training points $\{x_i\}_{i=1}^n,$ and the function $\bm{f}^*$ at test points $\{x_{i^*}\}_{i^*=1}^{n^*}.$ These estimates are defined as the ``exact posterior.'' 
    \item Each of the approximate methods are used to obtain samples from the posterior for the hyperparameters $\sigma_f,\rho,\tau,\bm{f}$, and $\bm{f}^*$, referred to as ``approximate posterior'' samples.
    \item For each sampler, the first 2,000 samples are discarded as burn-in and every tenth draw of the next 25,000 samples is retained.
\end{enumerate}


\afterpage{%

\begin{landscape}

\begin{table}[!phtb]
    \begin{tabular}{l c c c c c c c c }\toprule
        & & Truth & Exact GP & \multicolumn{2}{c}{FIFA-GP}  & \multicolumn{2}{c}{Compressed GP} \\
        \cmidrule(r){3-3} \cmidrule(r){4-4}\cmidrule(l){5-6}\cmidrule(l){7-8}
        &&& & $\epsilon_{\text{max}}=10^{-14}$ & $\epsilon_{\text{max}}=10^{-10}$ & $\epsilon_{\text{fro}}=0.01$ & $\epsilon_{\text{fro}}=0.1$ \\\midrule
        \rule{0pt}{4ex}Smooth and & $\text{MSPE}_{f^*}$ & - & 1e-04 & 1e-04 & 1e-04 & 2e-04 & 2e-04 \\
            low noise & 95\% CI ($\bm{f}^*$) Area & - & 0.27 & 0.27 & 0.27 & 0.29 & 0.30 \\
            &$\hat{\tau} \ [\tau_{\text{ll}},\tau_{\text{ul}}]$ & $\tau = 30$ & 28.3 [25.9, 30.9] & 28.3 [25.9, 30.8] & 28.3 [25.9, 30.8] & 28.4 [26.0, 31] & 28.4 [25.9, 30.9] \\
             &$\hat{\sigma}_f \ [\sigma_{f,\text{ll}},\sigma_{f,\text{ul}}]$ & $\sigma_f=1$ & 1.04 [0.52, 2.27] & 0.92 [0.51, 1.85] & 0.98 [0.52, 2.06] & 0.72 [0.44, 1.22] & 0.70 [0.43, 1.17] \\
             &$\hat{\rho} \ [\rho_{\text{ll}},\rho_{\text{ul}}]$ & $\rho=0.25$ & 0.33 [0.15, 0.69] & 0.35 [0.17, 0.70] & 0.33 [0.16, 0.65] & 0.68 [0.56, 0.77] & 0.76 [0.70, 0.80] \\
             & Time (min) & - & 96.7 & 8.7 & 8.0 & 6.2 & 6.0 \\
        \rule{0pt}{4ex}Smooth and & $\text{MSPE}_{f^*}$ & - & 0.008 & 0.008 & 0.008 & 0.007 & 0.007 \\
            high noise & 95\% CI ($\bm{f}^*$) Area & - & 0.83 & 0.85 & 0.85 & 0.98 & 0.95 \\
            &$\hat{\tau} \ [\tau_{\text{ll}},\tau_{\text{ul}}]$ & $\tau = 2$ & 2.25 [2.06, 2.45] & 2.26 [2.07, 2.46] & 2.26 [2.06, 2.46] & 2.26 [2.07, 2.47] & 2.26 [2.06, 2.46] \\
             &$\hat{\sigma}_f \ [\sigma_{f,\text{ll}},\sigma_{f,\text{ul}}]$ & $\sigma_f=1$ & 1.01 [0.53, 2.01] & 0.98 [0.52, 1.88] & 0.97 [0.50, 1.98] & 0.89 [0.51, 1.73] & 0.87 [0.50, 1.49] \\
             &$\hat{\rho} \ [\rho_{\text{ll}},\rho_{\text{ul}}]$ & $\rho=0.25$ & 0.23 [0.13, 0.53] & 0.27 [0.13, 0.85] & 0.26 [0.13, 0.68] & 0.68 [0.55, 0.82] & 0.60 [0.55, 0.71] \\
             & Time (min) & - & 94.5 & 8.7 & 8.0 & 6.1 & 6.2  \\
        \rule{0pt}{4ex}Wiggly and & $\text{MSPE}_{f^*}$ & - & 0.001 & 0.001 & 0.001 & 0.001 & 0.001 \\
            low noise & 95\% CI ($\bm{f}^*$) Area & - & 0.37 & 0.37 & 0.37 & 0.40 & 0.39 \\
            &$\hat{\tau} \ [\tau_{\text{ll}},\tau_{\text{ul}}]$ & $\tau = 30$ & 30.6 [28.0, 33.4] & 30.6 [28.0, 33.4] & 30.6 [28.0, 33.3] & 30.7 [28.0, 33.5] & 30.7 [28.1, 33.4] \\
             &$\hat{\sigma}_f \ [\sigma_{f,\text{ll}},\sigma_{f,\text{ul}}]$ & $\sigma_f=1$ & 1.45 [0.90, 2.50] & 1.41 [0.87, 2.39] & 1.50 [0.90, 2.91] & 1.14 [0.80, 1.69] & 1.17 [0.79, 1.74] \\
             &$\hat{\rho} \ [\rho_{\text{ll}},\rho_{\text{ul}}]$ & $\rho=2$ & 1.70 [1.20, 2.50] & 1.74 [1.22, 2.60] & 1.69 [1.14, 2.55] & 2.54 [2.54, 2.54] & 2.44 [2.35, 2.53] \\
             & Time (min) & - & 96.0 & 9.6 & 8.6 & 7.9 & 7.3 \\
         \rule{0pt}{4ex}Wiggly and & $\text{MSPE}_{f^*}$ & - & 0.013 & 0.013 & 0.013 & 0.013 & 0.013 \\
            high noise & 95\% CI ($\bm{f}^*$) Area & - & 1.32 & 1.33 & 1.33 & 1.30 & 1.30 \\
            &$\hat{\tau} \ [\tau_{\text{ll}},\tau_{\text{ul}}]$ & $\tau = 2$ & 2.09 [1.91, 2.28] & 2.09 [1.91, 2.28] & 2.09 [1.91, 2.29] & 2.09 [1.91, 2.28] & 2.09 [1.91, 2.28] \\
             &$\hat{\sigma}_f \ [\sigma_{f,\text{ll}},\sigma_{f,\text{ul}}]$ & $\sigma_f=1$ & 0.89 [0.57, 1.48] & 0.93 [0.56, 1.64] & 0.91 [0.58, 1.59] & 0.92 [0.58, 1.51] & 0.92 [0.60, 1.52] \\
             &$\hat{\rho} \ [\rho_{\text{ll}},\rho_{\text{ul}}]$ & $\rho=2$ & 2.61 [1.67, 3.73] & 2.48 [1.56, 3.78] & 2.56 [1.62, 3.79] & 2.43 [2.41, 2.50] & 2.42 [2.36, 2.49] \\
             & Time (min) & - & 95.3 & 10.8 & 9.7 & 7.6 & 8.0 \\
        \\ \bottomrule
    \end{tabular}
    \caption{Performance results, parameter estimates, and computing time for Bayesian GP regression with $n=1000$ data points simulated using a squared exponential covariance kernel. Time shown is total time for setup and 27,000 iterations through Gibbs sampler, with 2,500 samples retained.}\label{tab:sim_smalln}
\end{table} 

\end{landscape}

}

Table \ref{tab:sim_smalln} shows simulation results for $n=1000$. In the table, $\text{MSPE}_{f^*} = \frac{1}{n^*}\sum_{i=1}^{n^*}(f^*_i-\hat{f}^*_i)^2$ summarizes predictive performance for determining the true function mean at test points and 95\% CI ($\bm{f}^*$) Area summarizes the geometric area covered by the 95\% credible interval about $\bm{f}^*$. Hyperparameter summary  $\hat{\theta}=\frac{1}{T}\sum_{t=1}^T \theta^{(t)}$ provides the mean of the $T$ samples, where $\theta^{(t)}$ is the $t$-th draw of $\theta,$ and $[\theta_{\text{ll}},\theta_{\text{ul}}]$ gives the associated 95\% pointwise posterior credible interval. The key take away from Table \ref{tab:sim_smalln} is that inference made using FIFA-GP is extremely similar to that made using the exact GP for GP hyperparameters, function posterior, and the noise precision. The compressed GP has similar predictive performance and noise precision estimates, but inference on the GP hyperparameters and area of uncertainty about the function posterior differs from those quantities as measured by the exact GP. Analogous tables for $n=100$ and $n=500,$ and replicate samplers run with the same data for $n=100,$ shown in SI Section \ref{sec:sim_replicates}, provide further empirical support that the exact GP and FIFA-GP behave almost identically in terms of both predictive performance and inference. 


\subsection{Large-$n$ simulation}\label{sec:simulation_largen}

We compare the predictive performance and the computing time of each approximation method using a larger simulated data set, which was not based on a true underlying Gaussian process model. In this simulation, the true function is $f(x) = \text{sin}(2x) + \frac{1}{8}e^x$ and input data are sampled from a $\text{N}_{[-2,2]}(0,1)$ distribution. Values of $y$ are sampled having mean $f$ and precision $\tau=1.$

Parameter estimates and performance results analogous to those provided in the small-$n$ simulation are provided in SI Sections \ref{SIsec:large_n_compare} and \ref{SIsec:large_n_fifa}. As when data were simulated using a true GP for the mean, parameter estimates for FIFA-GP are similar to that of the exact GP sampler. As $n$ increases, both the precision and accuracy of FIFA-GP increase. Figure \ref{fig:simLargeNtiming} shows the time taken (in minutes) to iterate through 100 steps of the sampler for the exact GP and that of the approximated methods with a grid of 100 possible length scale values comprising the options for the discrete sampling step. The FIFA-GP method remains computationally tractable even when $n=$ 200,000. The cost of the pre-computation steps (i.e., projection construction) scales poorly for the compressed GP method.

\begin{figure}[htp]
    \centering
    \includegraphics[width=.8\linewidth]{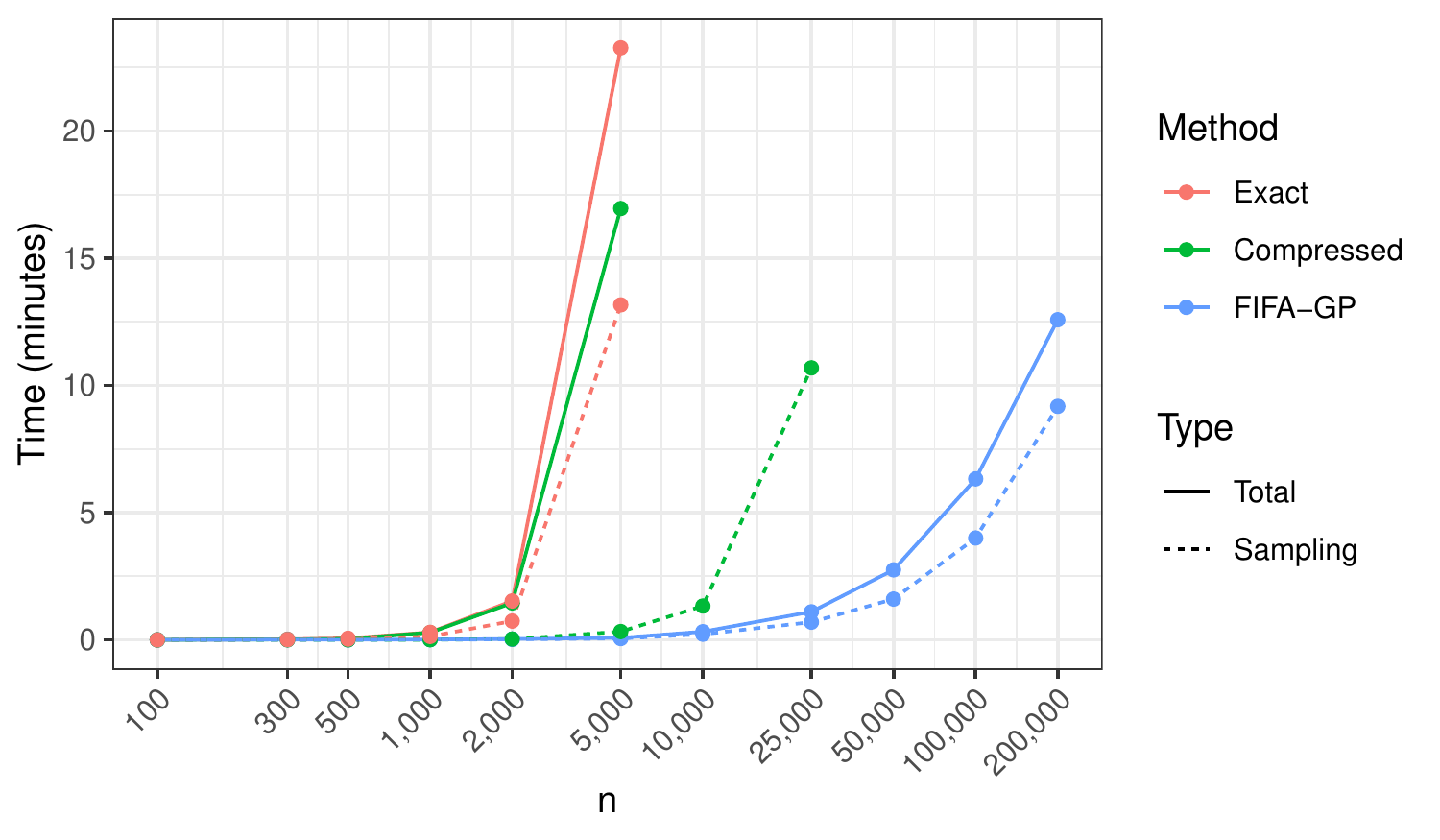}
    \caption{Time taken for 100 samples from full Bayesian posterior. Type ``Total'' includes the time taken to pre-create the matrices for the discrete uniform grid of length-scale values; ``Sampling'' only includes the time taken during the sampling phase after these setup computations have occurred.}
    \label{fig:simLargeNtiming}
\end{figure}

\subsection{Tensor product simulation}\label{sec:simulation_largen}

In order to illustrate the capability of the tensor product approach defined in equation (\ref{eq:tensor_prod}), we simulate from two functions having $d=2.$ The first function, $g_1(x_1,x_2) = \sin(x_1) \sin(x_2) \sqrt{x_1 x_2},$ is separable into a function of $x_1$ multiplied by a function of $x_2$ and thus should be easily approximated by the tensor product formulation. The second function, $g_1(x_1,x_2) = x^2 - 2xy+3y+2,$ is not separable and thus may not be well approximated by a single tensor product term. For each function, draws of $x_1$ and $x_2$ are sampled from $\mathcal{U}(0,4)$ and noisy observations of the functions are made with precision $\tau=0.5$. The number of observations is varied, with $n\in\{100, 500, 1000, 2000, 5000, 10000, 25000\}.$

Figure \ref{fig:tp_func1} shows the noisy observations of $g_1$ and model estimated surface for varying $n.$ Figure \ref{fig:tp_mspe1} shows the MSPE for surface estimates at test points. The model is able to learn the surface reasonably well, even with small $n$. Coverage of the 95\% credible interval is near nominal (0.962, 0.949, and 0.946 for $n=500$, $n=2000$, and $n=25000$, respectively). Furthermore, the model is salable to large $n$ because it relies on one-dimensional HODLR-approximated GPs.

\begin{figure}[!htb]
  \centering
  \includegraphics[width=.3\linewidth]{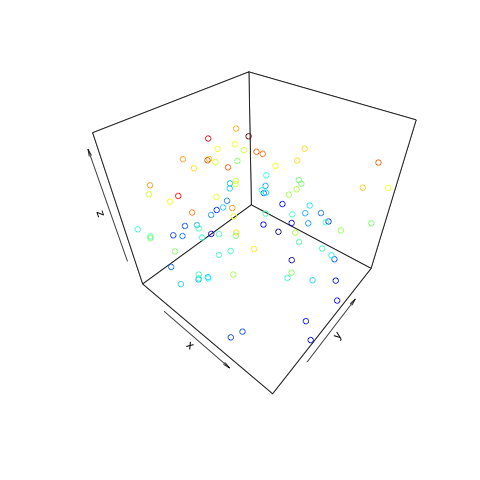}
  \includegraphics[width=.3\linewidth]{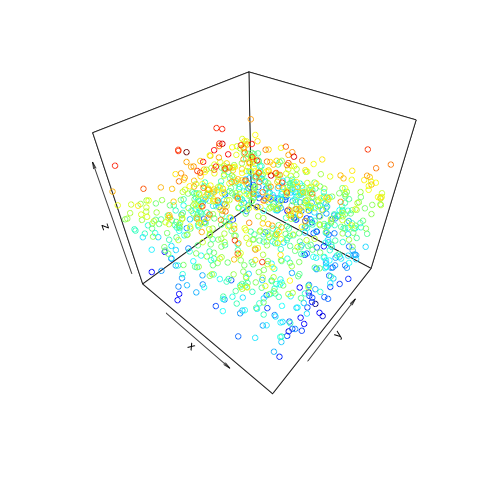}
  \includegraphics[width=.3\linewidth]{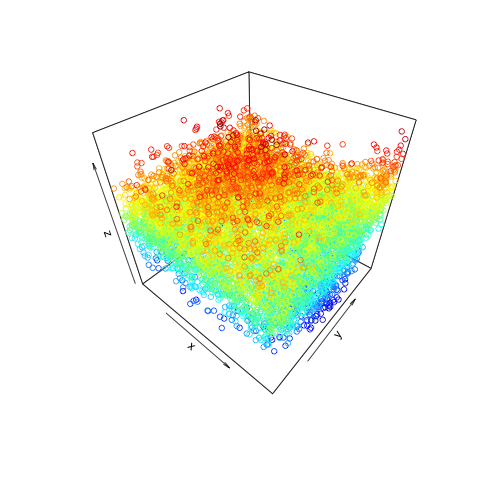}
  \includegraphics[width=.3\linewidth]{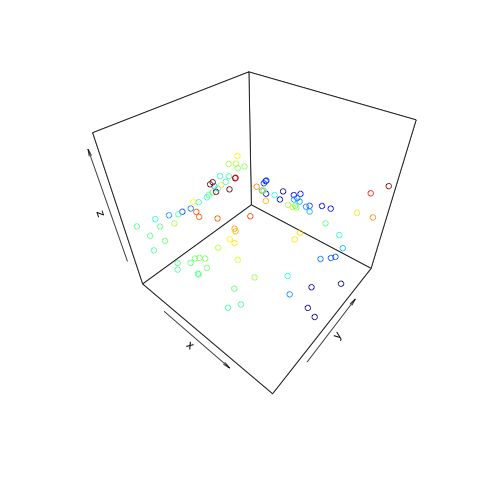}
  \includegraphics[width=.3\linewidth]{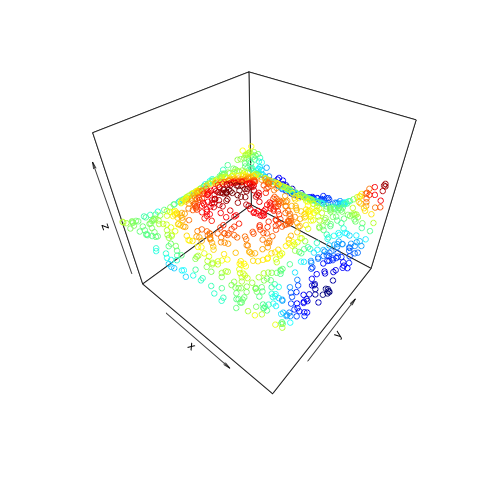}
  \includegraphics[width=.3\linewidth]{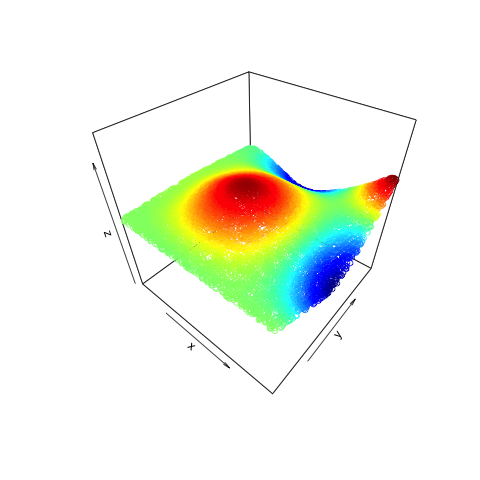}
  \caption{Top: Noisy observations. Bottom: Estimated surface from tensor product model. Left to right: $n=100$, $n=1000$, and $n=10000$.}
  \label{fig:tp_func1}
\end{figure}

\begin{figure}[!htb]
  \centering
  \includegraphics[width=.7\linewidth]{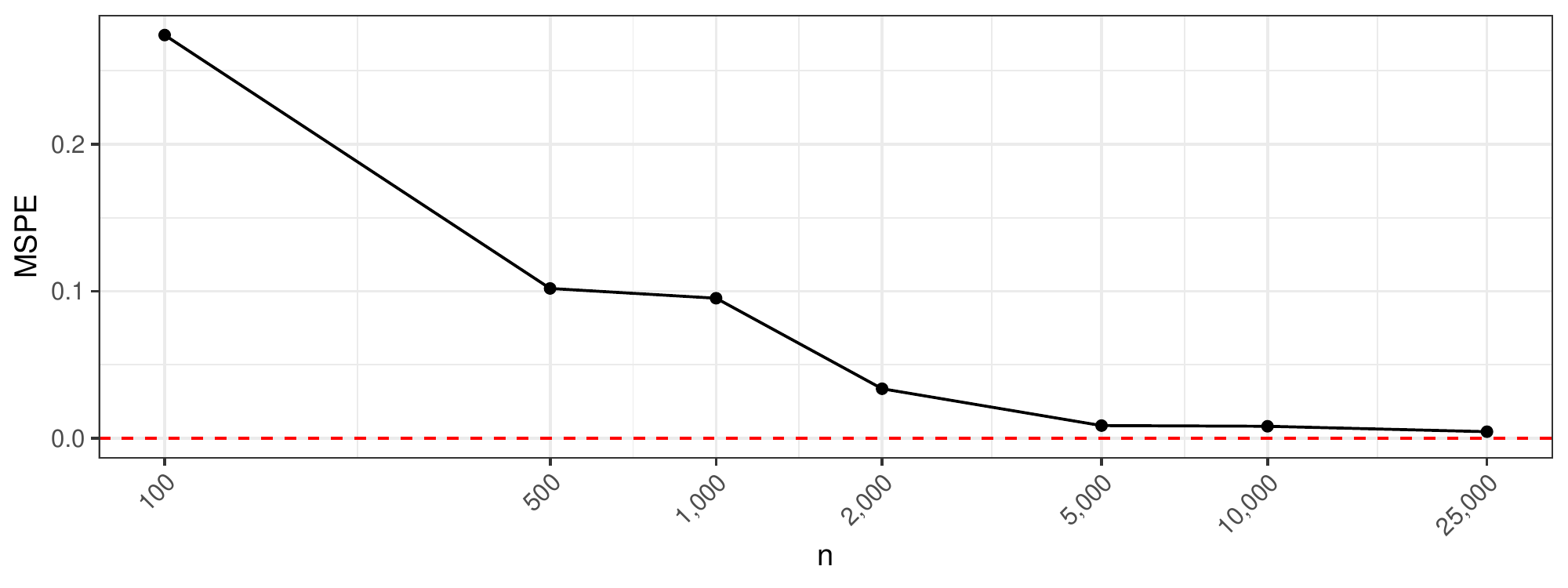}
  \caption{MSPE of the tensor product surface for predicting $g_1$ at new inputs.}
  \label{fig:tp_mspe1}
\end{figure}


Figure \ref{fig:tp_func2} shows the noisy observations of $g_2$ and model estimated surface for varying $n,$ with either 1 or 2 additive basis components used. Figure \ref{fig:tp_mspe2} shows the MSPE for surface estimates at test points. In this example, $g_2$ is poorly approximated by a single tensor product component, i.e. by the model $y_i = f_1(x_{1,i}) \otimes f_2(x_{2,i}) + e_i$, which may be due to choice of co-variance kernel or prior.  When a second tensor product basis as added, i.e. when we model the data as $y_i = f_1(x_{1,i}) \otimes f_2(x_{2,i}) + h_1(x_{1,i}) \otimes h_2(x_{2,i}) + e_i$, the model is flexible enough to capture the true surface. Coverage of the 95\% credible interval is near nominal (0.956, 0.949, 0.947 for $n=500$, $n=2000$, and $n=25000$, respectively).

\begin{figure}[!htb]
  \centering
  \includegraphics[width=.3\linewidth]{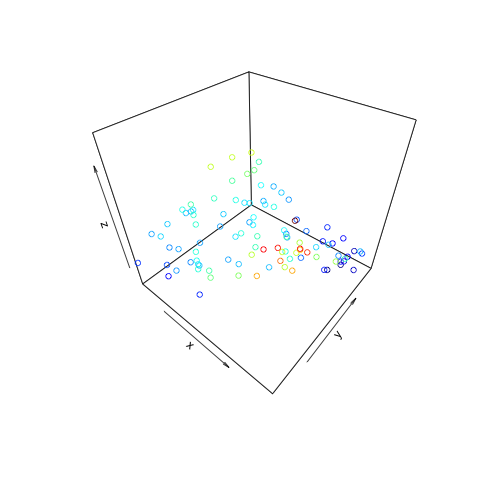}
  \includegraphics[width=.3\linewidth]{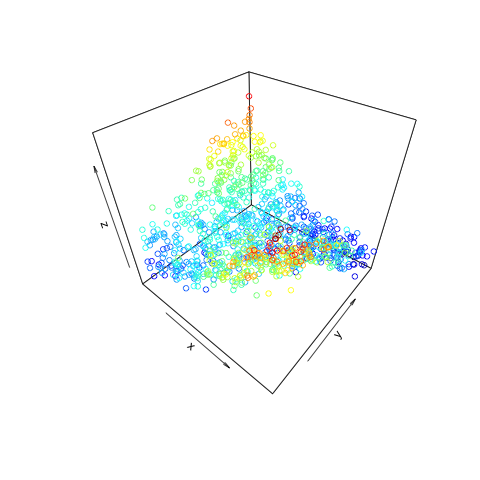}
  \includegraphics[width=.3\linewidth]{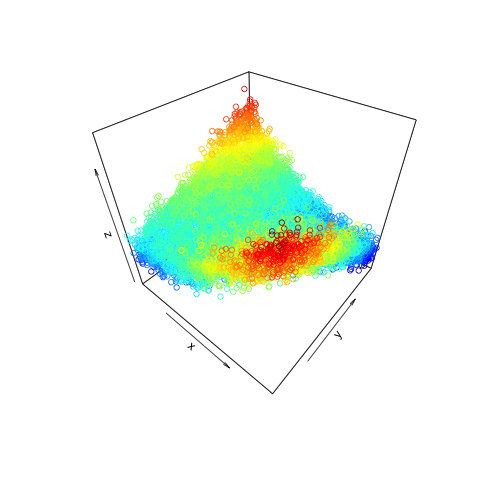}
    \includegraphics[width=.3\linewidth]{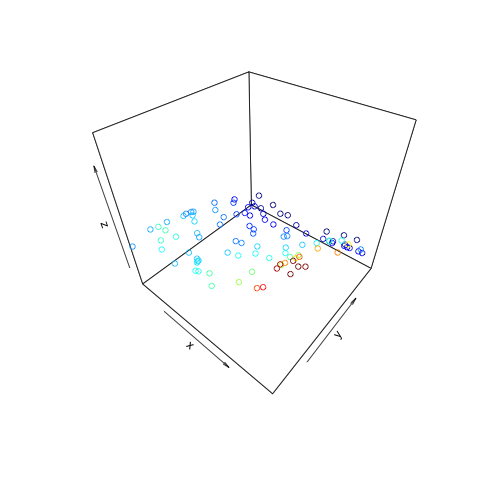}
  \includegraphics[width=.3\linewidth]{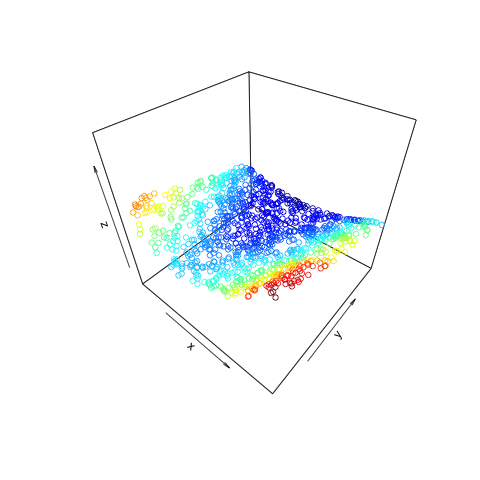}
  \includegraphics[width=.3\linewidth]{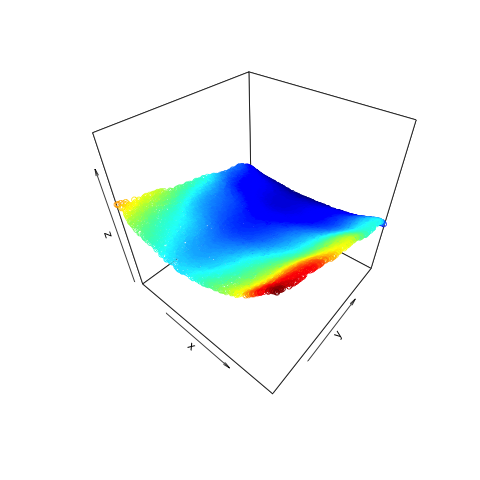}
  \includegraphics[width=.3\linewidth]{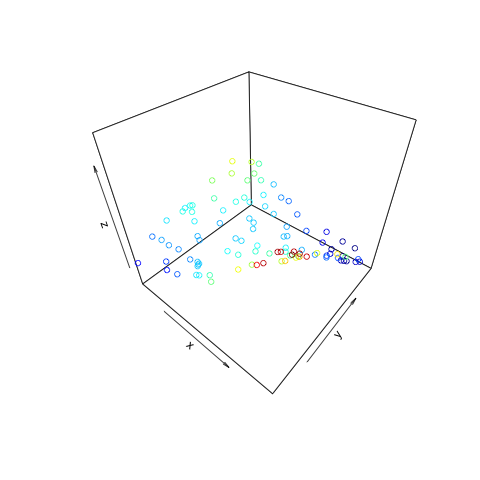}
  \includegraphics[width=.3\linewidth]{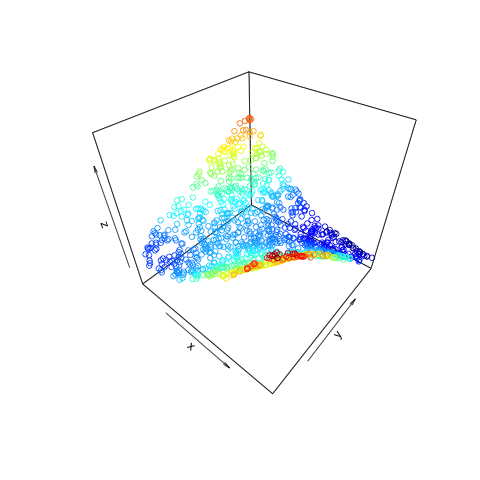}
  \includegraphics[width=.3\linewidth]{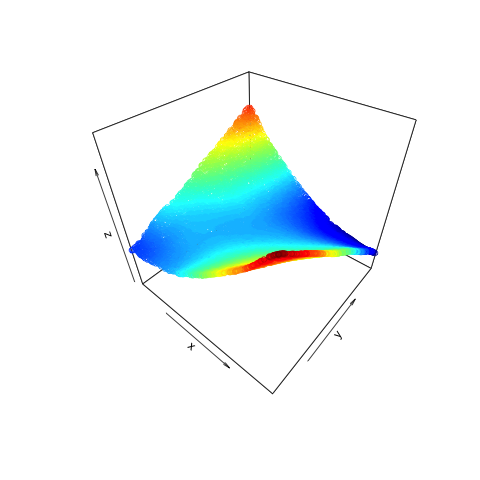}
  \caption{Top: Noisy observations. Middle (bottom): Estimated surface from tensor product model with one (two) additive basis term(s). Left to right: $n=100$, $n=1000$, and $n=10000$.}
  \label{fig:tp_func2}
\end{figure}

\begin{figure}[!htb]
  \centering
  \includegraphics[width=.8\linewidth]{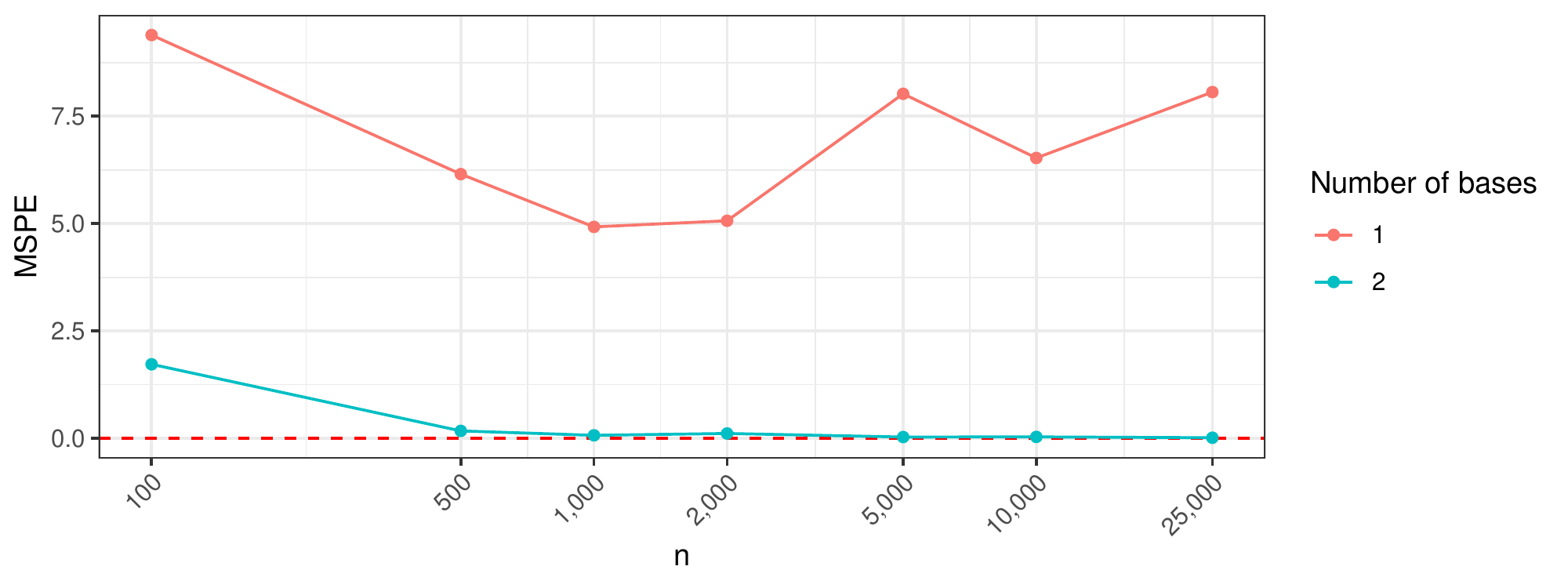}
  \caption{MSPE of the tensor product surface for predicting $g_2$ at new inputs with 1 or 2 additive bases.}
  \label{fig:tp_mspe2}
\end{figure}


\section{Illustrations using real data}\label{sec:real_data}

We consider two data examples to show how multidimensional surfaces can be modeled using FIFA-GP. As we only include a spacial process and do not include relevant covariates,  we do not claim the estimates to be optimal;  rather they are illustrative of how the methodology can be used to model flexible functional forms for extremely large data sets
on relatively modest computer hardware. 



\subsection{Atmospheric carbon dioxide}

NOAA started recording carbon dioxide (CO2) measurements at Mauna Loa Observatory, Hawaii in May of 1974. The CO2 data are measured as the mole fraction in dry air in units of parts per million (ppm). Hourly data are available for download from 1974 to the present at \url{https://www.esrl.noaa.gov/gmd/dv/data/index.php?parameter_name=Carbon%2BDioxide&frequency=Hourly%2BAverages&site=MLO}. 
To illustrate the utility of the proposed method, we use the full data for which the quality control flag indicated no obvious problems during collection or analysis (358,253 observations). We model the year-season surface as the sum of a GP for the annual effect, a GP for the seasonal effect, and a tensor product of GPs for year and season. For the tensor product, the one dimensional GP's had just under $9,000$ observations in each dimension.  For this example, the full MCMC algorithm took under thirty minutes. 

Figure \ref{fig:tp_mlo_surface} shows the observed and model-predicted year-season CO2 surfaces. There is an evident increase in CO2 across years along with a seasonal pattern, with a peak in early summer and a trough in the fall. The tensor product spline basis allows for variation in the seasonal shape that varies smoothly with year.

\begin{figure}[hptb]
	\centering
	\includegraphics[width=.45\linewidth]{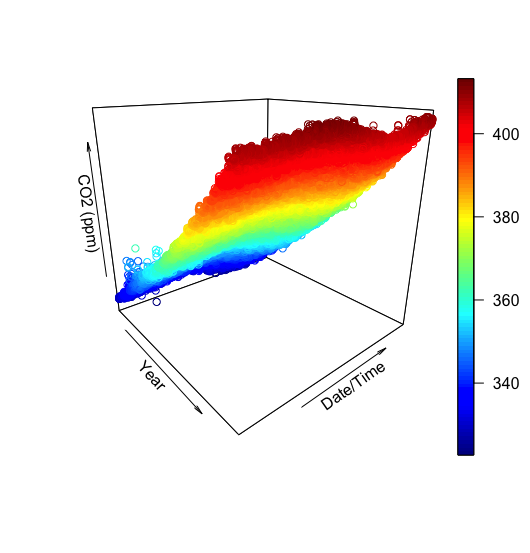}
	\includegraphics[width=.45\linewidth]{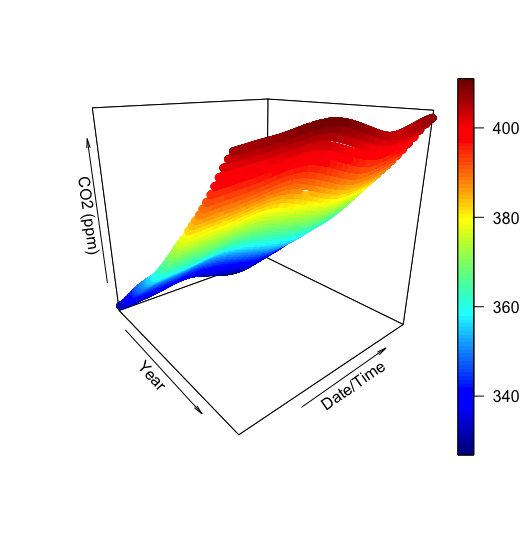}
	\caption[Observed and model-predicted surfaces for hourly Mauna Loa CO2 data.]{Left: Observed surface. Right: Model-predicted surface.}
	\label{fig:tp_mlo_surface}
\end{figure}

Figure \ref{fig:tp_mlo_slices} shows slices of the model-predicted year-season CO2 surfaces along each dimension. The tensor product of Gaussian processes fits the model well, with residuals having no apparent seasonal or annual patterns. Earlier years seem to exhibit some heteroskedastic and positive-skewed noise, which could be explicitly accounted for using an alternative noise model.

\begin{figure}[!htb]
	\centering
	\includegraphics[width=.65\linewidth]{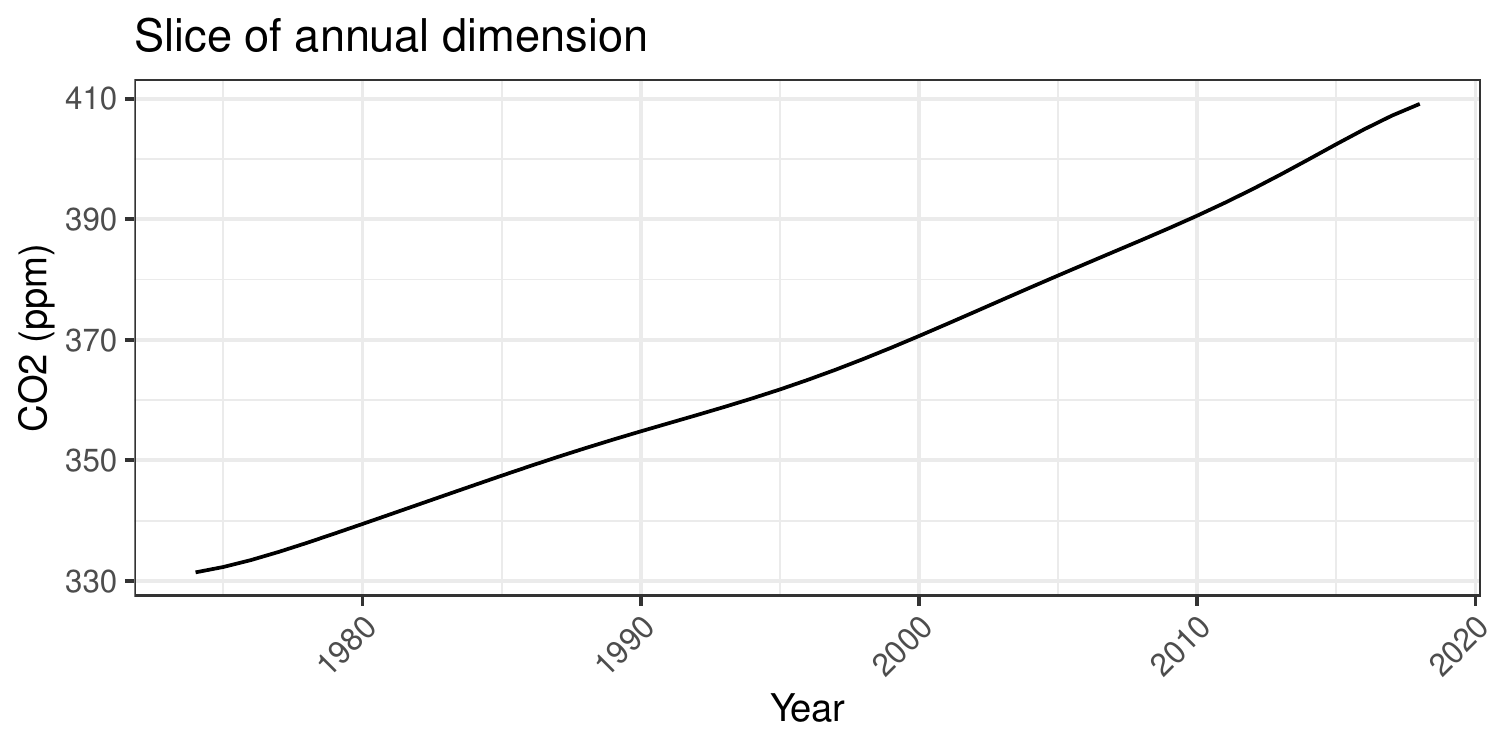}
	\includegraphics[width=.65\linewidth]{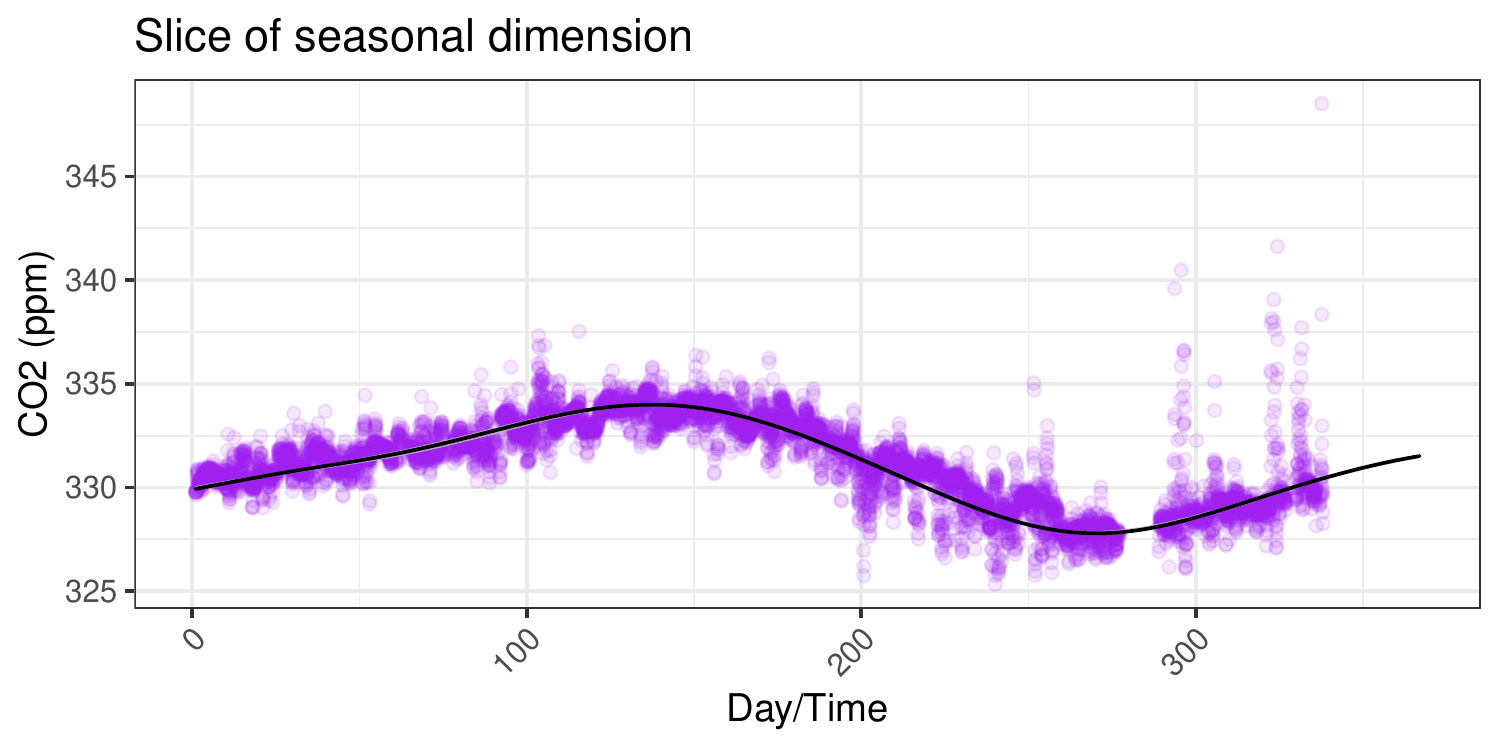}
	\caption[Slices of model-predicted surface for hourly Mauna Loa CO2 data.]{Top: Slice of CO2 surface viewed along the year dimension (hour 14 of day 185 shown). Evident is the annual increase in CO2. Bottom: Slice of CO2 surface viewed along the season dimension (year 1975 shown). Evident is the seasonal pattern of CO2 levels, with a peak in early summer and a trough in the fall.}
	\label{fig:tp_mlo_slices}
\end{figure}

\subsection{Particulate matter}

PM$_{2.5}$ describes fine particulate matter, inhalable and with diameters that are generally under 2.5 micrometers. The EPA monitors and reviews national PM$_{2.5}$ concentrations, and sets national air quality standards under the Clean Air Act. Fine particles are the particulate matter having greatest threat to human health. These particles are also the main cause of reduced visibility in many parts of the United States.

We use the sample-level PM$_{2.5}$ data for the contiguous states since 1999 (4,977,391 observations).  We model the latitude-longitude-time surface as the sum of a GP for the geographic effect and a tensor product of GPs for latitude, longitude, and time. This 
three dimensional surface fit again was fast taking about an hour. 

Figure \ref{fig:res_pm25} shows the model predicted mean PM$_{2.5}$ value for the contiguous United States at three time points. There is an evident decrease in PM$_{2.5}$ across years along with regional hot spots, with the southeast and parts of California being particularly problematic. The total mean PM$_{2.5}$ decreased over time, which reflects the increased air quality standards over this time period. 

\begin{figure}[!htbp]
	\centering
	\includegraphics[width=.75\linewidth]{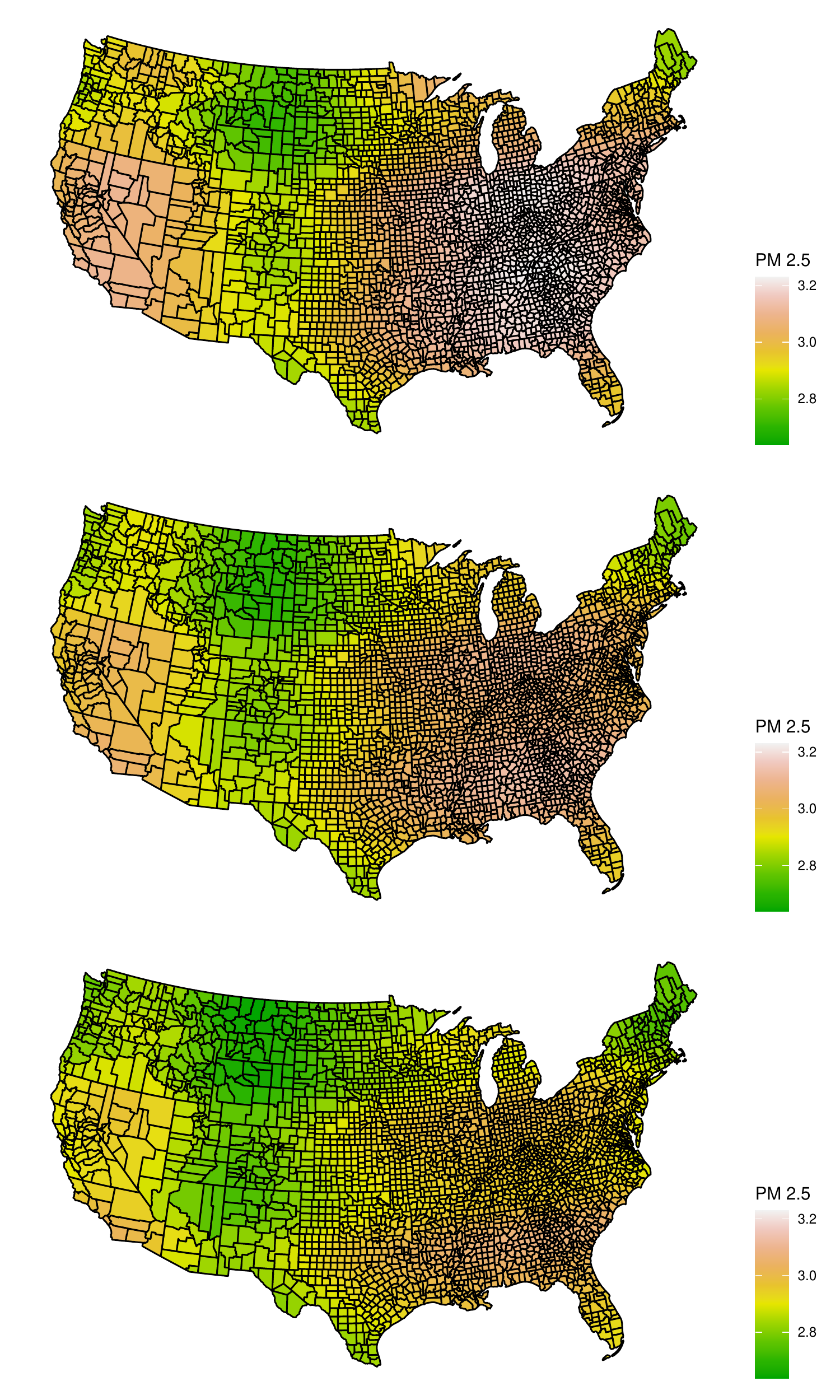}
	\caption{Model predicted mean PM$_{2.5}$ value for the contiguous United States. From top to bottom: February 19, 1999; July 16, 2009; and December 11, 2019.}
	\label{fig:res_pm25}
\end{figure}


\section{Discussion}\label{sec:discussion}

We have described an algorithm that allows near-exact Gaussian process sampling to be performed in a computationally tractable way up to $n$ on the order of $10^5$ for univariate inputs using $\mathcal{H}$ matrices. As observed by \cite{litvinenko2006application, ambikasaran2014fastdirect}, the HODLR matrix format can quickly grow in computation cost in higher dimensional input spaces. We have addressed this issue by proposing a tensor product approach that leverages the speed of HODLR operations in univariate input domains and scales to higher dimensions at $\mathcal{O}(d n\hspace{1mm}\mbox{log}^2 n).$ It would also be possible to use the univariate GPs as building blocks in other ways to construct multidimensional surfaces. The downside to this approach is that it is unsuitable for approximating an elaborate GP with a carefully chosen covariance for $d>1$.

While the use of HODLR matrices is convenient due to the well supported code base, the ideas in the paper are extensible to other classes of $\mathcal{H}$-matrices as well. Lemma \ref{thm:alg1} and Theorem \ref{thm:approx_fidel} are valid for any matrix approximation having an error bound. Other $\mathcal{H}$-matrix approximations could ameliorate some of the computational costs associated with HODLR matrices in higher-dimensional input settings, without requiring the tensor product GP approach, and 
we have begun to implement these ideas using $\mathcal{H}^2-$compression matrices using the H2LIB library \cite{hackbusch2002data}.  This library efficiently compresses matrices up to three dimensions, and we have had success modeling 3-dimensional GP problems up to $n=20,000.$  For the H2LIB library, the inversion uses a single processor, but one sample still takes less than a second for many problems where $n=20,000.$  For higher dimensional inputs, the tensor product approach still is applicable, and we imagine one could approximate a surface using different $\mathcal{H}/\mathcal{H^2}$ matrix compression algorithms with the same theoretical guarantees on approximation as well as computational efficiency.

Although we considered the intensive Bayesian regression problem of the case of large $n$ when $f()$ is modeled as a Gaussian process, the framework introduced here is general and can be applied to any problem in which some large structured matrix inversion is required within a Bayesian sampler. Still within the GP realm, there are many cases when univariate GPs are used as model components rather than for standalone regression problems. Having the ability to plug-in a fast near-exact sampler could open the use of these component GPs up to larger problems.

Outside of the GP regression, similar set of computational bottlenecks arise in the case of large $p$ when $f()$ is given a parametric sparsity inducing prior (e.g., the horseshoe prior). Specifically, the matrix inversion required in \cite{makalic2015simple} is also of the form $(\Sigma + \tau I)^{-1},$ where $\Sigma$ is some $p \times p$ positive definite matrix, $\tau$ is a constant, and $I$ is a $p \times p$ identity matrix. Exploring the use of $\mathcal{H}$-matrix compression and arithmetic in the sparse Bayesian regression setting provides a promising avenue of future research.

\section{Acknowledgements}

This work was partially supported by the Department of Energy Computational Science Graduate Fellowship (grant DE-FG02-97ER25308). This research was supported [in part] by the Intramural Research Program of the NIH, National Institute of Environmental Health Sciences. The funders had no role in study design, data collection and analysis, decision to publish, or preparation of the manuscript. The authors would like to thank David Dunson and Amy Herring for helpful comments.  

\printbibliography

\newpage

\section{Appendix}\label{Appendix}

\subsection{Additional sampling considerations}\label{sec:additional_considerations}

Here we provide details on how the function posterior at new inputs can be sampled, an efficient way to handle non-unique input points, and handling heteroskedastic noise.

\subsubsection{Function posterior at new inputs}



In the Gibbs sampling framework, i.e. when we can condition on $\bm{f}$, the predictive distribution of the realization of the function $f^*$ at a new point $\bm{x}^*$ is given by
\begin{equation}
\begin{split}
f^*|\bm{x}^*,\bm{f},X,\Theta,\tau &\sim \text{N}(\mu_{*}, \sigma^2_{*}), \\
\mu_{*} &= \bm{k}_{*n} K^{-1}\bm{f}, \\
\sigma^2_{*} &= k_{**} - \bm{k}_{*n} K^{-1}\bm{k}_{n*},
\label{eq:gp_posterior_newx}
\end{split}
\end{equation}
where $k_{**} = k(\bm{x}^*,\bm{x}^*)$, and $\bm{k}_{*n} = k(\bm{x}^*,X)$ and $\bm{k}_{n*}=k(X,\bm{x}^*)=(\bm{k}_{*n})^T$ are realizations of the covariance kernel between new input $\bm{x}^*$ and each of the observed inputs.

Posterior realizations of $f^*$ can be sampled in $\mathcal{O}(n\log{}n)$ per point assuming the $\mathcal{H}$-matrix approximation for $K$ is pre-assembled and pre-factorized; in this case the most costly calculation of $\sigma_*^2$ from equation (\ref{eq:gp_posterior_newx}) is a linear solve. Such pre-computation is possible when the length-scale is sampled on a fixed grid. In this case we can also pre-solve $K^{-1}\bm{k}_{n*}$ for each possible length-scale value. 

Unfortunately, we cannot pre-solve $K^{-1}\bm{f}$ because $\bm{f}$ is sampled and changes at each iteration of the Gibbs sampler. Therefore, for $N$ such new points the cost is $\mathcal{O}(Nn\log{}n)$. Thus, this method is not preferred for applications in which function realizations are desired for a very large number of new points (i.e., when $N\approx n$). However, it is quite powerful in applications for which $n \gg N.$

\subsubsection{Non-unique input points}

Denote the set of all $n$ (possibly non-unique) inputs as $X=\{\bm{x}_h \in \mathbb{R}^d\}_{h=1}^n.$ Let $X^u=\{\bm{x}_i^u \in \mathbb{R}^d\}_{i=1}^{U}$ denote the set of $U$ \textit{unique} input values. Then $U\leq n$ and for any $i\neq j$ $\bm{x}_i^u \neq \bm{x}_j^u.$ Computational savings can be achieved when $U<n.$ Let $Q_i$ denote the set of original indices for which the input values are equal to $\bm{x}_i^u$, i.e. $Q_i = \{h \text{ s.t. } \bm{x}_h = \bm{x}_i^u \}.$ Then define $\bm{y}^u=\{y_i^u = \frac{1}{|Q_i|}\sum_{h \in Q_i} y_h \in \mathbb{R} \}_{i=1}^U$ to be the set of $U$ average observations at each unique input value. Sampling and inference can proceed using a GP prior on the $U \times U$ covariance matrix associated with observations $\bm{y}^u.$ The precision term $\tau$ added to the diagonal of the original covariance matrix is then replaced with $|Q_i|^2\cdot\tau$ for entry $(i,i)$ of the new covariance matrix for all $i\in 1, \ldots, U.$

\subsubsection{Heteroskedastic noise}

It was assumed that $e_i \sim \text{N}(0,\tau^{-1}).$ However, heteroskedastic noise can easily be accomodated in the above samplers. 
Suppose that $e_i \sim \text{N}(0,\tau_i^{-1}).$ 
Define $D^{-1}=\text{diag}(\tau_1^{-1},\ldots,\tau_n^{-1})$ to be the diagonal matrix of variance parameters so that $D=\text{diag}(\tau_1,\ldots,\tau_n)$ is then the diagonal matrix of precision parameters.
Then the posterior variance from equation (\ref{eq:gp_posterior}) becomes
\begin{align}
    \text{Cov}(\bm{f}|\bm{y}, X, \Theta) &= K - K(K + D^{-1})^{-1}K \tag*{} \\
    &= (D + K^{-1})^{-1} \tag*{(Sherman-Morrison-Woodbury formula)} \\
    &= K(DK + I)^{-1} \tag*{(($D + K^{-1})^{-1} = K K^{-1} (D + K^{-1})^{-1}$)} \\
    &= K(K + D^{-1})^{-1}D^{-1} \tag*{(Right-multiply by $DD^{-1}$)} \\
    &= D^{-1}(K + D^{-1})^{-1}K \tag*{(Transpose symmetric matrix)}.
\end{align}

The posterior mean from equation (\ref{eq:gp_posterior}) becomes
$$\mathbb{E}(\bm{f}|\bm{y}, X, \Theta) = K(K + D^{-1})^{-1}\bm{y} = \text{Cov}(\bm{f}|\bm{y}, X, \Theta)D\bm{y}.$$

The format of the above posterior covariance is chosen so as to enable easier HODLR operations. See Appendix \ref{sec:gp_posterior_hetero} for further algebraic exposition and a discussion of how the KL divergence bound and Algorithm 1 are adjusted for heteroskedastic variance.

\subsubsection{Practical considerations}

The HODLR approximation relies on random compression of the off-diagonal block matrices. This process doesn't guarantee that the result is always a positive semidefinite matrix, although in practice it most often is (and when it is not, simply decreasing $\varepsilon$ will often lead to the conditions being met). Similarly, the product $\tilde{K}\tilde{M}^{-1}$ in Theorem \ref{thm:approx_fidel} is not guaranteed to be positive semidefinite for every perturbation from $K$ and $M.$ In fact, this issue is practically present for large $n$ even when using the exact GP due to finite precision arithmetic.

\subsection{Proofs}\label{Appendix:proofs}

\begin{proof}[Proof of Lemma \ref{thm:alg1}]
After \textit{Step} \ref{alg1_1}, $\text{Var}(Z)=\tau \tilde{K}\tilde{K} + \tilde{K} =\tilde{K}(\tau \tilde{K}+I).$ Note that $\tau \tilde{K}+I$ is equivalent to $\tilde{M}$ because of the tolerances specified in \textit{Step} \ref{alg1_tol}. To see why, note that $||K-\tilde{K}||_{\max} < \epsilon^*/\tau \implies ||\tau K- \tau \tilde{K}||_{\max} < \epsilon^*$, so the HODLR approximation $\tau \tilde{K}$ satisfies the $\epsilon^*$ bound for approximating $\tau K$. Since the off-diagonal block matrices are factorized via partial-pivoted LU decomposition, multiplication by a constant does not impact the solution other than by rescaling. That is, finding an approximation for $K$ with tolerance $\epsilon^*/\tau$ is equivalent to finding an approximation for $\tau K$ with tolerance $\epsilon^*$ and dividing the result by $\tau.$ Next, recall the HODLR decomposition preserves diagonal elements of the original matrix, so a decomposition of matrix $A$ plus the identity matrix is equivalently expressed as $\widetilde{A + I}$ and $\tilde{A} + I.$ Thus $\tau \tilde{K}+I = \widetilde{\tau K + I}$ and $Z \sim \text{N}(0, \tilde{K}\tilde{M})$ at the start of \textit{Step} \ref{alg1_2}. Now at the end of \textit{Step} \ref{alg1_2} $\text{Var}(Z^*)=\tilde{M}^{-1}\tilde{K}\tilde{M}\tilde{M}^{-1} = \tilde{M}^{-1}\tilde{K}.$ Because $\tilde{M}^{-1},$ $\tilde{K},$ and their product are symmetric, the terms commute and $\text{Var}(Z^*)=\tilde{K}\tilde{M}^{-1}$. Thus $Z^*\sim \text{N}(0, \tilde{\Sigma}_{f})$. The mean $\tilde{\mu}_{f}$ is created in \textit{Step} \ref{alg1_3}, leading to the sum in the final step to have the desired distribution.
\end{proof}

\begin{proof}[Proof of Theorem \ref{thm:approx_fidel}]

From equation (\ref{eq:gp_posterior}), let $M = \tau K + I.$ Define $\tilde{K}$ and $\tilde{M}$ to be the $\mathcal{H}$-matrix approximations of $K$ and $M,$ respectively. Denote $\bm{p} \sim N(\tau K \bm{y}, K M^{-1})$ and $\bm{q} \sim N(\tau \tilde{K} \bm{y}, \tilde{K} \tilde{M}^{-1}),$ with density functions $\mathcal{P}$ and $\mathcal{Q},$ respectively. In other words, $\mathcal{P}$ is the true GP posterior and $\mathcal{Q}$ is the $\mathcal{H}-$approximated GP posterior.

\begin{align}\label{eq:big_KL_pts1to3}
\begin{split}
    E_{\mathcal{P}} \Bigg[ \mbox{log}\Bigg( \frac{\mathcal{P}}{\mathcal{Q}} \Bigg) \Bigg] = \frac{1}{2} \bigg( &
    \underbrace{\text{log}\bigg[\frac{\text{det}(\tilde{K}\tilde{M}^{-1})}{\text{det}(KM^{-1})}\bigg]}_\text{(i.)} + \\
    & \underbrace{\text{tr}\big[ (\tilde{K}\tilde{M}^{-1})^{-1} (KM^{-1}) \big] - n}_\text{(ii.)} + \\ 
    & \underbrace{ \big[ \tau \tilde{K}\tilde{M}^{-1}\bm{y} - \tau KM^{-1}\bm{y} \big]' \big(\tilde{K}\tilde{M}^{-1}\big)^{-1} \big[ \tau \tilde{K}\tilde{M}^{-1}\bm{y} - \tau KM^{-1}\bm{y} \big] }_\text{(iii.)} \bigg)
\end{split}
\end{align}

Let $\Omega^K = \tilde{K} - K,$ and $\Omega^M = \tilde{M} - M$ be the matrices of differences between the approximated and true matrices. Assume by construction $\underset{1 \leq i,j \leq n}{\max} |\Omega^K_{ij}|\leq \epsilon$ and $\underset{1 \leq i,j \leq n}{\max} |\Omega^M_{ij}|\leq \epsilon.$ That is, assume the absolute entry-wise difference between each approximated and true matrix is at most some $\epsilon.$

Note that using the HODLR decomposition one can create $\tilde{M}$ using $\tilde{K}$ rather than having to create it separately, if desired. For $\tau>1,$ simply using $\tilde{M}=\tau \tilde{K} + I$ satisfies the $\epsilon$-bound on both matrices. For $\tau<1,$ finding $\tilde{K}$ satisfying the bound for $\epsilon/\tau$ and then $\tilde{M}=\tau \tilde{K} + I$ satisfies the $\epsilon$-bound on both matrices. In the latter case $\tilde{K}$ is calculated to a higher level of fidelity in order to guarantee the fidelity of $\tilde{M}$ meets the requisite bound.

\begin{proofpart} 

For the part (i.) bound we will rely on the Hoffman-Wielandt (HW) inequality \cite{rajendra1997matrix}. This inequality states that for $n \times n$ $\Sigma^0, \Sigma^1$ symmetric with $\lambda_i^0, \lambda_i^1$ the eigenvalues of $\Sigma^0, \Sigma^1$ respectively, there exists a permutation $\pi$ such that $$\sum_{i=1}^n (\lambda_{\pi(i)}^0 - \lambda_i^1)^2 \leq ||\Sigma^0 - \Sigma^1||_F^2.$$

Suppose $||\Sigma^0 - \Sigma^1||_F^2 \leq c^2,$ then each individual $(\lambda_{\pi(i)}^0 - \lambda_i^1)^2$ must be $\leq c^2$ because the sum for all $n$ terms is $\leq c^2.$ In other words, $|\lambda_{\pi(i)}^0 - \lambda_i^1| \ \leq \ |c|$ for each $i \in 1,\ldots,n.$ Suppose $c \geq 0,$ then 
$$\frac{\lambda_{\pi(i)}^0}{\lambda_i^1} \in [1-c, 1+c].$$ 
In the case that the max norm of $(\Sigma^0 - \Sigma^1)$ is bounded by $\epsilon>0,$ i.e. $\underset{1\leq i,j \leq n}{\max}|\Sigma^0_{ij} - \Sigma^1_{ij}| < \epsilon,$ we have 
\begin{align}\label{eq:hw_ineq_1}
\sum_{i=1}^n (\lambda_{\pi(i)}^0 - \lambda_i^1)^2 \leq ||\Sigma^0 - \Sigma^1||_F^2 \leq n^2 \epsilon^2 \quad \implies \quad \frac{\lambda_{\pi(i)}^0}{\lambda_i^1} \in [1-n \epsilon, 1+n \epsilon].
\end{align}

Now consider the log of the product of the above terms. Supposing all $\lambda_i^0,\lambda_i^1>0$ (which will be the case in our application because we are considering covariance matrices), 
\begin{align}\label{eq:hw_ineq}
\begin{split}
\frac{\lambda_{\pi(i)}^0}{\lambda_i^1} \in [1-n \epsilon, 1+n \epsilon] \quad &\implies \quad \prod_{i=1}^n \frac{\lambda_{\pi(i)}^0}{\lambda_i^1} \in [ (1-n \epsilon)^n, (1+n \epsilon)^n ] \\
&\implies \quad \text{log}\bigg[\prod_{i=1}^n \frac{\lambda_{\pi(i)}^0}{\lambda_i^1}\bigg] \in [n \text{log}(1-n \epsilon), n \text{log}(1+n \epsilon)].
\end{split}
\end{align}
Note for the lower bound in equation (\ref{eq:hw_ineq}) to be sensible we need $1>n\epsilon \implies \epsilon < \frac{1}{n}.$ See Sections \ref{sec:more_algebra_evalDifRat} and \ref{sec:more_algebra_maxnormFnorm} for additional algebraic exposition.

Additionally, we will rely on the following determinant and log properties for any matrices $A$ and $B$ of compatible dimensions and $\lambda^A_i$ the eigenvalues of $A$ for $n \times n$ matrix $A$:
\begin{align}
\text{det}(AB) &= \text{det}(A) \cdot \text{det}(B), \label{eq:d1} \\
\text{det}(A^{-1}) &= \text{det}(A)^{-1}, \label{eq:d2} \\
\text{det}(A) &= \prod_{i=1}^n \lambda^A_i, \label{eq:d3} \\
\text{log}(AB) &= \text{log}(A) + \text{log}(B), \label{eq:l1} 
\end{align}

Now armed with the HW inequality and the determinant and log properties, we return to the issue of bounding part (i.):
\begin{align*}
\text{log}\bigg[\frac{\text{det}(\tilde{K}\tilde{M}^{-1})}{\text{det}(KM^{-1})}\bigg] &= \text{log}\bigg[\frac{\text{det}(\tilde{K}) \cdot \text{det}(\tilde{M}^{-1})}{\text{det}(K) \cdot \text{det}(M^{-1})}\bigg] \tag*{(Equation (\ref{eq:d1}))} \\
&= \text{log}\bigg[\frac{\text{det}(\tilde{K}) \cdot \text{det}(M)}{\text{det}(K) \cdot \text{det}(\tilde{M})}\bigg] \tag*{(Equation (\ref{eq:d2}))} \\
&= \underbrace{\text{log}\bigg[\frac{\text{det}(\tilde{K})}{\text{det}(K)}\bigg]}_\text{(d1.)} + \underbrace{\text{log}\bigg[\frac{\text{det}(M)}{\text{det}(\tilde{M})}\bigg]}_\text{(d2.)} \tag*{(Equation (\ref{eq:l1}))}
\end{align*}

Consider part (d1.). Let $\lambda_i^{K}$ and $\lambda_i^{\tilde{K}}$ be the eigenvalues of $K$ and $\tilde{K},$ respectively, and $\pi$ be some permutation of the indices $i=1,\ldots,n,$ then we have
\begin{align*}
\text{log}\bigg[\frac{\text{det}(\tilde{K})}{\text{det}(K)}\bigg] &= \text{log}\bigg[\frac{\prod_{i=1}^n \lambda_i^{\tilde{K}}}{\prod_{i=1}^n \lambda_i^{K}}\bigg] \tag*{(Equation (\ref{eq:d3}))} \\
&= \text{log}\bigg[\frac{\prod_{i=1}^n \lambda_{\pi(i)}^{\tilde{K}}}{\prod_{i=1}^n \lambda_i^{K}}\bigg] \tag*{} \\
&= \text{log}\bigg[\prod_{i=1}^n\frac{\lambda_{\pi(i)}^{\tilde{K}}}{\lambda_i^{K}}\bigg] \tag*{} \\
&\leq n \ \text{log}(1 + n \epsilon) \tag*{(Equation (\ref{eq:hw_ineq}))} \\ 
&\leq n^2 \epsilon \tag*{(For $x>-1,$ $\text{log}(1+x)\leq x$)}.
\end{align*}

The bound on part (d2.) follows analogously, so the overall bound on part (i.) is
\begin{align}\label{eq:bnd_parti}
\text{log}\bigg[\frac{\text{det}(\tilde{K}\tilde{M}^{-1})}{\text{det}(KM^{-1})}\bigg] \leq 2 n^2 \epsilon.
\end{align}

\end{proofpart} 

\begin{proofpart} 

For the part (ii.) bound we have
\begin{align}
\text{tr}\big[  (\tilde{K} & \tilde{M}^{-1})^{-1} (KM^{-1}) \big] - n \tag*{} \\ 
=& \ \text{tr}\big[ \tilde{M}\tilde{K}^{-1} KM^{-1} \big] - n \tag*{} \\
=& \ \text{tr}\big[ \tilde{K}^{-1} KM^{-1} \tilde{M} \big] - n \tag*{(Trace circular)} \\
=& \ \text{tr}\big[ \tilde{K}^{-1} (\tilde{K}-\Omega^K)M^{-1} \tilde{M} \big] - n \tag*{($K=\tilde{K}-\Omega^K$ by definition)} \\
=& \ \text{tr}\big[ \tilde{K}^{-1} (\tilde{K}-\Omega^K)M^{-1} (M+\Omega^M) \big] - n \tag*{($\tilde{M}=M+\Omega^M$ by definition)} \\
=& \ \text{tr}\big[ \tilde{K}^{-1} (\tilde{K}-\Omega^K)M^{-1} (M+\Omega^M) \big] - n \tag*{($\tilde{M}=M+\Omega^M$ by definition)} \\
=& \ \text{tr}\big[ (\tilde{K}^{-1}\tilde{K}-\tilde{K}^{-1}\Omega^K) (M^{-1}M+M^{-1}\Omega^M) \big] - n \tag*{} \\
=& \ \text{tr}\big[ (I^n-\tilde{K}^{-1}\Omega^K) (I^n+M^{-1}\Omega^M) \big] - n \tag*{($I^n$ is $n\times n$ identity matrix)} \\
=& \ \text{tr}\big[ I^n + M^{-1}\Omega^M -\tilde{K}^{-1}\Omega^K - \tilde{K}^{-1}\Omega^K M^{-1}\Omega^M \big] - n \tag*{} \\
=& \ \text{tr}\big[ I^n \big] + \text{tr}\big[ M^{-1}\Omega^M \big] - \text{tr}\big[ \tilde{K}^{-1}\Omega^K \big] - \text{tr}\big[ \tilde{K}^{-1}\Omega^K M^{-1}\Omega^M \big] - n \tag*{(Trace additive)} \\
=& \ \text{tr}\big[ M^{-1}\Omega^M \big] - \text{tr}\big[ \tilde{K}^{-1}\Omega^K \big] - \text{tr}\big[ \tilde{K}^{-1}\Omega^K M^{-1}\Omega^M \big] \tag*{($\text{tr}\big[ I^n \big]=n$)} 
\end{align}

Consider the bound on the first term in the above expression:
\begin{align}
\text{tr}\big[ M^{-1}\Omega^M \big] \ &= \ \sum_{i=1}^n \sum_{j=1}^n M^{-1}_{ij}\Omega^M_{ji} \tag*{(Definition of trace)} \\
&\leq \ ||M^{-1}||_{\max} \sum_{i=1}^n \sum_{j=1}^n \Omega^M_{ji} \tag*{\big( $||M^{-1}||_{\max} = \underset{1 \leq i,j \leq n}{\max} |M^{-1}_{ij}|$ \big)} \\ 
&\leq \ ||M^{-1}||_{\max} \sum_{i=1}^n \sum_{j=1}^n \epsilon \tag*{\big( $\underset{1 \leq i,j \leq n}{\max} |\Omega^M_{ij}|\leq \epsilon$ \big)} \\
&\leq \ ||M^{-1}||_{\max} \ n^2 \ \epsilon \tag*{}.
\end{align}


Because the $\epsilon$ bound on $\Omega^K$ applies to $-\Omega^K$ as well, the upper bound on the second term may be found analogously as that on the first:
\begin{align}
-\text{tr}\big[ \tilde{K}^{-1}\Omega^K \big] &= \text{tr}\big[ \tilde{K}^{-1}(-\Omega^K) \big] \tag*{(Trace linear)} \\ 
&= \ \sum_{i=1}^n \sum_{j=1}^n \tilde{K}^{-1}_{ij}(-\Omega^K_{ji}) \tag*{(Definition of trace)} \\
&\leq \ ||\tilde{K}^{-1}||_{\max} \sum_{i=1}^n \sum_{j=1}^n (-\Omega^K_{ji}) \tag*{\big( $||\tilde{K}^{-1}||_{\max} = \underset{1 \leq i,j \leq n}{\max} |\tilde{K}^{-1}_{ij}|$ \big)} \\ 
&\leq \ ||\tilde{K}^{-1}||_{\max} \sum_{i=1}^n \sum_{j=1}^n \epsilon \tag*{\big( $\underset{1 \leq i,j \leq n}{\max} |-\Omega^K_{ij}| = \underset{1 \leq i,j \leq n}{\max} |\Omega^K_{ij}| \leq \epsilon$ \big)} \\
&\leq \ ||\tilde{K}^{-1}||_{\max} \ n^2 \ \epsilon \tag*{}.
\end{align}

It would be preferable to have the bound only depend on $K$ and the specific approximation $\tilde{K}$. Weyl's inequality states that for a matrix $A$ perturbed by some matrix $\Omega,$ the difference  between the eigenvalues of the original matrix $A$ and its perturbation $\tilde{A}=A+\Omega$ is bounded by the spectral norm of $\Omega.$ Specifically, 
\begin{align}\label{eq:weyls}
|\lambda^{A}_i - \lambda^{\tilde{A}}_i| \leq ||\Omega||_2 \quad \forall i=1,\ldots n.
\end{align}

Let $\sigma_{\min}(\cdot)$ and $\sigma_{\max}(\cdot)$ define the minimum and maximum eigenvalues, respectively, of a matrix. Then by equation (\ref{eq:weyls}) we can say 
\begin{align}\label{eq:sigminBnd}
| \sigma_{\min}(K) - \sigma_{\min}(\tilde{K}) | \leq ||K - \tilde{K}||_2 \leq n^2 ||K - \tilde{K}||_{\max} = n^2 \epsilon.
\end{align}

We consider the bound on $||\tilde{K}^{-1}||_{\max},$ namely
\begin{align}
||\tilde{K}^{-1}||_{\max} \ &\leq \ \sqrt{n}||\tilde{K}^{-1}||_2 \tag*{} \\
&= \ \sqrt{n \sigma_{\max}(\tilde{K}^{-1})} \tag*{\text{(Definition of spectral norm)}} \\
&= \ \sqrt{n/\sigma_{\min}(\tilde{K})} \tag*{($\sigma_{\max}(A^{-1}) = 1/\sigma_{\min}(A)$)} \\
&= \ \sqrt{n/\sigma_{\min}(\tilde{K})} \tag*{} \\
&= \ \sqrt{n/[\sigma_{\min}(K) + \delta_{K}]} \tag*{(Letting $\delta_K=\sigma_{\min}(K)-\sigma_{\min}(\tilde{K})$)}.
\end{align}

The smaller $\sigma_{\min}(K) + \delta_{K}$ is, the looser the bound on $||\tilde{K}^{-1}||_{\max}$ becomes. From equation (\ref{eq:sigminBnd}) we know that $|\delta_{K}| \leq n^2 \epsilon$. In the worst case, $\delta_{K}$ will be as negative as possible so that the term in the denominator becomes closer to 0; this ``worst'' value is $-n^2 \epsilon$. This worst case relies on the assumption that $n^2 \epsilon < \sigma_{\min}(K),$ so that $\sigma_{\min}(\tilde{K})$ remains nonzero. Any $\epsilon < \frac{\sigma_{\min}(K)}{n^2}$ satisfies this condition. Then the bound on the second term is
\begin{align}
-\text{tr}\big[ \tilde{K}^{-1}\Omega^K \big] 
&\leq \ ||\tilde{K}^{-1}||_{\max} \ n^2 \ \epsilon \tag*{} \\
&\leq \ n^{5/2} \ [\sigma_{\min}(K) - n^2 \epsilon]^{-1/2} \ \epsilon.
\end{align}


Before we move on to the third term, we state two trace inequalities. First, let $A,B$ be two matrices of compatible dimensions with $A'B$ being a square matrix, then by the Cauchy-Schwarz inequality: 
\begin{align}\label{eq:tr_ineq1}
\text{tr}\big[ A' B \big] \leq \sqrt{ \text{tr}\big[ A' A \big] \text{tr}\big[ B' B \big] }.
\end{align}

Next, let $A$ be a positive semi-definite matrix with eigenvalues $\lambda_i$, then: 
\begin{align}\label{eq:tr_ineq2}
\text{tr}\big[ A^2 \big] = \sum_i \lambda_i^2 \leq \big(\sum_i \lambda_i\big)^2 = \text{tr}\big[ A \big]^2.
\end{align}

Now consider the third and final term. To bound this term, we will rely on two trace inequalities, shown in equations (\ref{eq:tr_ineq1}) and (\ref{eq:tr_ineq2}). Treat $-\tilde{K}^{-1}\Omega^K$ as $A'$ and $M^{-1}\Omega^M$ as $B$ in equation (\ref{eq:tr_ineq1}), then
\begin{align}
-\text{tr}\big[ \tilde{K}^{-1}\Omega^K M^{-1}\Omega^M \big] \ &= \ \text{tr}\big[ -\tilde{K}^{-1}\Omega^K M^{-1}\Omega^M \big] \tag*{(Trace linear)} \\
& \leq \ \sqrtlimspace{\underbrace{\text{tr}\big[ (-\tilde{K}^{-1}\Omega^K)(-\tilde{K}^{-1}\Omega^K)' \big]}_\text{(t1.)} \underbrace{\text{tr}\big[ (M^{-1}\Omega^M)'(M^{-1}\Omega^M) \big]}_\text{(t2.)} } \tag*{(Equation (\ref{eq:tr_ineq1}))}.
\end{align}

Consider bounding only part (t1.) above:
\begin{align}
\text{tr}\big[ (-\tilde{K}^{-1}\Omega^K)(-\tilde{K}^{-1}\Omega^K)' \big] \ &= \ \text{tr}\big[ \tilde{K}^{-1}\Omega^K \Omega^K \tilde{K}^{-1} \big] \tag*{($\tilde{K}^{-1},\Omega^K$ symmetric)} \\
&= \ \text{tr}\big[ \tilde{K}^{-1} \tilde{K}^{-1}\Omega^K \Omega^K \big] \tag*{(Trace circular)} \\
&\leq \ \sqrt{ \text{tr}\big[ (\tilde{K}^{-1})^4 \big] \text{tr}\big[ (\Omega^K)^4 \big] } \tag*{(Equation (\ref{eq:tr_ineq1}))} \\
&\leq \ \sqrt{ \text{tr}\big[ \tilde{K}^{-1} \big]^4 \text{tr}\big[ (\Omega^K)^4 \big] } \tag*{(Equation (\ref{eq:tr_ineq2}) and $\tilde{K}^{-1}$ is PSD)} \\
&\leq \ \sqrt{ \text{tr}\big[ \tilde{K}^{-1} \big]^4 (n \epsilon)^4 } \tag*{($\text{tr}\big[ (\Omega^K)^r \big] \leq (n \epsilon)^r$, see Section \ref{sec:more_algebra_maxnormTr})} \\
&\leq \ \text{tr}\big[ \tilde{K}^{-1} \big]^2 (n \epsilon)^2 \tag*{}.
\end{align}

An analogous sequence of arguments bounds part (t2.):
\begin{align}
\text{tr}\big[ (M^{-1}\Omega^M)'(M^{-1}\Omega^M) \big] \ \leq \ \text{tr}\big[ M^{-1} \big]^2 (n \epsilon)^2 \tag*{}.
\end{align}

Thus, the bounds on (t1.) and (t2.) can be used to bound the whole expression and we have:
\begin{align*}
-\text{tr}\big[ \tilde{K}^{-1}\Omega^K M^{-1}\Omega^M \big] \ 
&\leq \ \sqrt{\text{tr}\big[ \tilde{K}^{-1} \big]^2 (n \epsilon)^2 \text{tr}\big[ M^{-1} \big]^2 (n \epsilon)^2} \tag*{} \\
&= \ |\text{tr}(\tilde{K}^{-1})| \ | \text{tr}(M^{-1}) | \ n^2 \epsilon^2 \tag*{}.
\end{align*}

As with the max-norm of $\tilde{K},$ we want the term $|\text{tr}(\tilde{K}^{-1})|$ to only involve $K,$ $n,$ and $\epsilon.$ Assuming as before that $\epsilon < \frac{\sigma_{\min}(K)}{n^2},$ and using the fact that the trace of a matrix is the sum of its eigenvalues, we have
\begin{align*}
|\text{tr}(\tilde{K}^{-1})| \ &\leq \ n||\tilde{K}^{-1}||_2^2 \\
&= \ n \ \sqrt{\sigma_{\max}(\tilde{K}^{-1})} \\
&= \ n/\sigma_{\min}(\tilde{K}) \\
&\leq \ n/(\sigma_{\min}(K) - n^2 \epsilon).
\end{align*}


Therefore, the bound on all of part (ii.) is
\begin{equation}\label{eq:bnd_partii}
\text{tr}\big[  (\tilde{K} \tilde{M}^{-1})^{-1} (KM^{-1}) \big] - n \leq \ ||M^{-1}||_{\max} \ n^2 \epsilon \ + \ \frac{n^{5/2}}{\sqrt{\sigma_{\min}(K) - n^2 \epsilon}} \ \epsilon \ + \ \frac{n^3 \ |\text{tr}(M^{-1})|}{\sigma_{\min}(K) - n^2 \epsilon} \ \epsilon^2.
\end{equation}

\end{proofpart} 

\begin{proofpart} 

To bound the part (iii.) term, we will rely on the following norm inequality. Let $A$ be a symmetric $n \times n$ matrix and $||A||_2$ denote the spectral norm, then it follows by Rayleigh's inequality and the definitions of the spectral and Frobenius norms that
\begin{align}\label{eq:norm_ineq1}
\bm{x}'A\bm{x} \ &\leq \ \max_{1\leq j \leq n} \lambda_j \ ||\bm{x}||^2 \ = \ ||A||_2 \ ||\bm{x}||^2 \ \leq \ ||A||_F \ ||\bm{x}||^2
\end{align}
where $\lambda_j$ are eigenvalues of matrix $A$ and $\bm{x}$ is a length-$n$ vector.

We will also rely on an inequality for the difference of matrix inverses. Let $A, B$ be invertible $n \times n$ matrices and let $|| \cdot ||$ denote some submultiplicative matrix norm. Suppose $||A^{-1}(B-A)||<1$ and $||A-B||<c$ for some $c>0.$ Then
\begin{align}\label{eq:bound_invMatDiff}
\begin{split}
||A^{-1} - B^{-1}|| &\leq \frac{||A^{-1}||^2  \  ||A-B||}{1-||A^{-1}|| \ ||A-B||} \\
&\leq \frac{c \ ||A^{-1}||^2 }{1-c \ ||A^{-1}||}.
\end{split}
\end{align}
See Section \ref{sec:more_algebra_diffInvDiff} for details of this inequality, which is based on \cite{horn2012condition}.


\begin{align}
\big[ & \tau \tilde{K}\tilde{M}^{-1}\bm{y} - \tau KM^{-1}\bm{y} \big]' \big(\tilde{K}\tilde{M}^{-1}\big)^{-1} \big[ \tau \tilde{K}\tilde{M}^{-1}\bm{y} - \tau KM^{-1}\bm{y} \big] \tag*{} \\
&= \tau^2 \big[ (\tilde{K}\tilde{M}^{-1}-KM^{-1})\bm{y} \big]' \big(\tilde{K}\tilde{M}^{-1}\big)^{-1} \big[ (\tilde{K}\tilde{M}^{-1}-KM^{-1})\bm{y} \big] \tag*{} \\ 
&= \tau^2 \bm{y}' (\tilde{K}\tilde{M}^{-1}-KM^{-1}) \big(\tilde{K}\tilde{M}^{-1}\big)^{-1} (\tilde{K}\tilde{M}^{-1}-KM^{-1}) \bm{y} \tag*{($\tilde{K}\tilde{M}^{-1}$ and $KM^{-1}$ symmetric)} \\
&= \tau^2 ||(\tilde{K}\tilde{M}^{-1}-KM^{-1}) \big(\tilde{K}\tilde{M}^{-1}\big)^{-1} (\tilde{K}\tilde{M}^{-1}-KM^{-1})||_2 ||\bm{y}||^2 \tag*{(Equation (\ref{eq:norm_ineq1})} \\
&\leq \underbrace{ \tau^2 ||\big(\tilde{K}\tilde{M}^{-1}\big)^{-1}||_2 \ ||\bm{y}||^2}_\text{(n1.)} \  \underbrace{||\tilde{K}\tilde{M}^{-1}-KM^{-1}||_2^2}_\text{(n2.)} \tag*{(Submultiplicativity of $2$-norm)}
\end{align}

We want to bound the term in (n1.) involving $\tilde{K}$ and $\tilde{M}$ by some function of $K,$ $M$, $n$, and $\epsilon$. We have, assuming as before that $\epsilon < \frac{\sigma_{\min}(K)}{n^2}$,
\begin{align*}
||\big(\tilde{K}\tilde{M}^{-1}\big)^{-1}||_2 \ &= \ ||\tilde{M}\tilde{K}^{-1}||_2 \\
&\leq \ ||\tilde{M}||_2 \ ||\tilde{K}^{-1}||_2 \\
&= \sqrt{ \sigma_{\max}(\tilde{M}) \sigma_{\max}(\tilde{K}^{-1}) } \\
&= \sqrt{ \frac{\sigma_{\max}(\tilde{M})}{\sigma_{\min}(\tilde{K})} } \\
&\leq \sqrt{ \frac{\sigma_{\max}(M) + n^2 \epsilon}{\sigma_{\min}(K) - n^2 \epsilon} }.
\end{align*}

The term (n1.) acts as a constant, and (n2.) will go toward 0 with $\epsilon$. Here, consider the square root of (n2.). Assume $||M^{-1}(\tilde{M}-M)||_2<1$, a required condition for the inequality shown in equation (\ref{eq:bound_invMatDiff}). (Note that any $\epsilon<\frac{1}{n ||M^{-1}||_2}$ will satisfy this condition -- see Section \ref{sec:more_algebra_diffInvDiff} for details.)

\begin{align*}
||\tilde{K}\tilde{M}^{-1}-KM^{-1}||_2 &= ||(K+\Omega^K)\tilde{M}^{-1}-KM^{-1}||_2 \tag*{($\tilde{K}=K+\Omega^K$ by definition)} \\
&= ||K\tilde{M}^{-1}+\Omega^K\tilde{M}^{-1}-KM^{-1}||_2 \tag*{} \\
&= ||K(\tilde{M}^{-1}-M^{-1})+\Omega^K\tilde{M}^{-1}||_2 \tag*{} \\
&\leq ||K(\tilde{M}^{-1}-M^{-1})||_2+||\Omega^K\tilde{M}^{-1}||_2 \tag*{(Triangle inequality)} \\
&\leq ||K||_2 ||\tilde{M}^{-1}-M^{-1}||_2+||\Omega^K||_2 ||\tilde{M}^{-1}||_2 \tag*{(Submultiplicativity of $2-$norm)} \\
&\leq ||K||_2 ||\tilde{M}^{-1}-M^{-1}||_2+ n \epsilon ||\tilde{M}^{-1}||_2 \tag*{($||\Omega^K||_2 \leq \sqrt{\underset{i}{\sum}\underset{j}{\sum}|\max{\Omega^K_{ij}}|^2}$)} \\
&\leq ||K||_2 ||M^{-1}-\tilde{M}^{-1}||_2+ n \epsilon ||\tilde{M}^{-1}||_2 \tag*{($||A-B|| = ||B-A||$)} \\
&\leq \frac{||K||_2 \ ||M^{-1}||_2^2 \ ||M-\tilde{M}||_2}{1-||M^{-1}||_2 \ ||M-\tilde{M}||_2}+ n \epsilon ||\tilde{M}^{-1}||_2 \tag*{(Equation (\ref{eq:bound_invMatDiff}))} \\
&= \frac{||K||_2 \ ||M^{-1}||_2^2 \ ||-\Omega^M||_2}{1-||M^{-1}||_2 \ ||-\Omega^M||_2}+ n \epsilon ||\tilde{M}^{-1}||_2 \tag*{($-\Omega^M=M-\tilde{M}$)} \\
&= \frac{||K||_2 \ ||M^{-1}||_2^2 \ ||\Omega^M||_2}{1-||M^{-1}||_2 \ ||\Omega^M||_2}+ n \epsilon ||\tilde{M}^{-1}||_2 \tag*{($||A||=||-A||$)} \\
&\leq \frac{n \epsilon \ ||K||_2 \ ||M^{-1}||_2^2}{1 - n \epsilon  ||M^{-1}||_2}+ n \epsilon ||\tilde{M}^{-1}||_2. \tag*{($||\Omega^K||_2 \leq \sqrt{\underset{i}{\sum}\underset{j}{\sum}|\max{\Omega^K_{ij}}|^2}$)}
\end{align*}

As with $||\tilde{K}^{-1}||_2,$ we can bound $||\tilde{M}^{-1}||_2$ in terms of $M$, $n$, and $\epsilon$.
Assuming that $\epsilon < \frac{\sigma_{\min}(M)}{n^2}$ (which, since $M$ is almost always better conditioned than $K$, will likely be true already because of the earlier statement that we chose $\epsilon < \frac{\sigma_{\min}(K)}{n^2}$),
\begin{align*}
||\tilde{M}^{-1}||_2 \ &= \sqrt{ \sigma_{\max}(\tilde{M}^{-1})} \\
&= 1/\sqrt{\sigma_{\min}(\tilde{M})} \\
&\leq 1/\sqrt{\sigma_{\min}(M) - n^2 \epsilon}.
\end{align*}

Note that since $n,$ $||K||_2,$ and $||M^{-1}||_2$ are fixed (although is is possible they are each quite large), we have
\begin{align*}
\lim_{\epsilon \rightarrow 0} \ \frac{n \epsilon \ ||K||_2 \ ||M^{-1}||_2^2}{1 - n \epsilon  ||M^{-1}||_2} \ = \ 0.
\end{align*}

Since we chose $\epsilon < \frac{\sigma_{\min}(M)}{n^2},$ we have $\epsilon < n \sigma_{\min}(M) < n \sqrt{\sigma_{\min}(M)}$ (assuming $\sigma_{\min}(M)<1$, if not we could simply choose $\epsilon$ satisfying $\epsilon < n \sqrt{\sigma_{\min}(M)}$). Then $1-n\epsilon ||M^{-1}||_2 > 1,$ so $\frac{n \epsilon \ ||K||_2 \ ||M^{-1}||_2^2}{1 - n \epsilon  ||M^{-1}||_2} < n \epsilon \ ||K||_2 \ ||M^{-1}||_2^2,$ and all of (n2.) is $\mathcal{O}(n^2 \epsilon^2)$.


Therefore, the bound on all of part (iii.) is
\begin{align}\label{eq:bnd_partiii}
\begin{split}
\big[ \tau \tilde{K}\tilde{M}^{-1}\bm{y} -& \tau KM^{-1}\bm{y} \big]' \big(\tilde{K}\tilde{M}^{-1}\big)^{-1} \big[ \tau \tilde{K}\tilde{M}^{-1}\bm{y} - \tau KM^{-1}\bm{y} \big] \\
&\leq \tau^2 \sqrt{ \frac{\sigma_{\max}(M) + n^2 \epsilon}{\sigma_{\min}(K) - n^2 \epsilon} } \ ||\bm{y}||^2 \bigg( \frac{n \epsilon \ ||K||_2 \ ||M^{-1}||_2^2}{1 - n \epsilon  ||M^{-1}||_2}+ \frac{n \epsilon}{\sqrt{\sigma_{\min}(M) - n^2 \epsilon}} \ \bigg)^2.
\end{split}
\end{align}

\end{proofpart} 

\end{proof}

\subsection{Additional algebraic work}\label{sec:more_algebra}

\subsubsection{Relationship between bound on absolute difference in eigenvalues and their ratio}\label{sec:more_algebra_evalDifRat}

Say you have $|\lambda^0 - \lambda^1| \ \leq c.$

Suppose $\lambda^0 \geq \lambda^1,$ then 
$$\lambda^0 - \lambda^1 \ \leq c \ \implies \ \frac{\lambda^0}{\lambda^1} \leq 1+c.$$

Suppose $\lambda^0 \leq \lambda^1,$ i.e. $\lambda^0 - \lambda^1$ is negative, then
$$-(\lambda^0 - \lambda^1) \ \leq c \ \implies \ \lambda^0 - \lambda^1 \ \geq -c \ \implies \ \frac{\lambda^0}{\lambda^1} \geq 1-c.$$

\subsubsection{Relationship between bound on max norm and that on $F-$norm}\label{sec:more_algebra_maxnormFnorm}

Let $A$ be some $m\times n$ matrix. By definition, the max norm $||A||_{\max}$ is 
$$||A||_{\max} = \underset{i,j}{\max}\{|a_{ij}|^2\}$$ 
and the Frobenius norm $||A||_F$ is 
$$||A||_F = \sqrt{\underset{i}{\sum}\underset{j}{\sum} |a_{ij}|^2}.$$ 
Thus, $||A||_{\max} \leq ||A||_F \leq \sqrt{mn} ||A||_{\max}$ and
$||A||_{\max}^2 \leq ||A||_F^2 \leq mn ||A||_{\max}^2.$

Suppose that $m=n$ and the max norm $||A||_{\max}$ of $A$ is bounded by $\epsilon>0,$ i.e. 
$||A||_{\max} < \epsilon.$ Then 
$$||A||_F^2 \leq n^2 \epsilon^2.$$

\subsubsection{Relationship between bound on max norm and that on trace}\label{sec:more_algebra_maxnormTr}

Suppose $\underset{1 \leq i,j \leq n}{\max} |\Omega_{ij}|\leq \epsilon,$ then $\text{tr}\big[ \Omega^r \big] \leq (n \epsilon)^r.$ To see this illustrated for $r=4$ (the case of interest for the inequality bounds above), consider a simple example with $n=3.$ The worst case is that every entry of $\Omega$ is equal to its upper bound, i.e.
\[
\Omega =
  \begin{bmatrix}
    \epsilon & \epsilon & \epsilon \\
    \epsilon & \epsilon & \epsilon
  \end{bmatrix}.
\]

Then 
\[
\Omega^2 =
  \begin{bmatrix}
    \epsilon & \epsilon & \epsilon \\
    \epsilon & \epsilon & \epsilon \\
    \epsilon & \epsilon & \epsilon
  \end{bmatrix}
  \begin{bmatrix}
    \epsilon & \epsilon & \epsilon \\
    \epsilon & \epsilon & \epsilon \\
    \epsilon & \epsilon & \epsilon
  \end{bmatrix} 
  = \begin{bmatrix}
    \sum_{i=1}^n \epsilon^2 & \sum_{i=1}^n \epsilon^2 & \sum_{i=1}^n \epsilon^2 \\
    \sum_{i=1}^n \epsilon^2 & \sum_{i=1}^n \epsilon^2 & \sum_{i=1}^n \epsilon^2 \\
    \sum_{i=1}^n \epsilon^2 & \sum_{i=1}^n \epsilon^2 & \sum_{i=1}^n \epsilon^2
  \end{bmatrix}
  = \begin{bmatrix}
    n \epsilon^2 & n \epsilon^2 & n \epsilon^2 \\
    n \epsilon^2 & n \epsilon^2 & n \epsilon^2 \\
    n \epsilon^2 & n \epsilon^2 & n \epsilon^2
  \end{bmatrix},
\]
\[
\Omega^3 =
  \Omega^2
  \begin{bmatrix}
    \epsilon & \epsilon & \epsilon \\
    \epsilon & \epsilon & \epsilon \\
    \epsilon & \epsilon & \epsilon
  \end{bmatrix} 
  = \begin{bmatrix}
    n^2 \epsilon^3 & n^2 \epsilon^3 & n^2 \epsilon^3 \\
    n^2 \epsilon^3 & n^2 \epsilon^3 & n^2 \epsilon^3 \\
    n^2 \epsilon^3 & n^2 \epsilon^3 & n^2 \epsilon^3
  \end{bmatrix},
\]
and 
\[
\Omega^4 =
  \Omega^3
  \begin{bmatrix}
    \epsilon & \epsilon & \epsilon \\
    \epsilon & \epsilon & \epsilon \\
    \epsilon & \epsilon & \epsilon
  \end{bmatrix} 
  = \begin{bmatrix}
    n^3 \epsilon^4 & n^3 \epsilon^4 & n^3 \epsilon^4 \\
    n^3 \epsilon^4 & n^3 \epsilon^4 & n^3 \epsilon^4 \\
    n^3 \epsilon^4 & n^3 \epsilon^4 & n^3 \epsilon^4
  \end{bmatrix}.
\]

Then $$\text{tr}(\Omega^4) = \sum_{i=1}^n n^3 \epsilon^4 = n \cdot n^3 \epsilon^4 = (n \epsilon)^4.$$

\subsubsection{Relationship between bound on matrices difference and bound on inverse matrices difference}\label{sec:more_algebra_diffInvDiff}

Let $A, B$ be invertible $n \times n$ matrices and let $|| \cdot ||$ denote some submultiplicative matrix norm. Define $\Delta = B-A$. Assume $$||A^{-1}\Delta||<1$$ to ensure that $B$ is nonsingular. For the bound on $||A^{-1} - B^{-1}||$, we will use the following result from \cite{horn2012condition}:
\begin{align}\label{eq:binv_bound}
||B^{-1}|| \leq \frac{||A^{-1}||}{1-||A^{-1} \Delta||}.
\end{align}

This bound comes from the following:
\begin{align*}
||B^{-1}|| &= ||A^{-1}-A^{-1}BB^{-1} + A^{-1}AB^{-1}|| \tag*{(Add 0 matrix and multiply by identity)} \\
&= ||A^{-1}-A^{-1} \Delta B^{-1}|| \tag*{(Factor out terms)} \\ 
&\leq ||A^{-1}|| + ||A^{-1} \Delta B^{-1}|| \tag*{(Triangle inequality)} \\
&\leq ||A^{-1}|| + ||A^{-1} \Delta|| \ ||B^{-1}|| \tag*{(Submultiplicativity of norm)}
\end{align*}
Dividing both sides by $||B^{-1}||$ in the inequality $||B^{-1}|| \leq ||A^{-1}|| + ||A^{-1} \Delta|| \ ||B^{-1}||$ gives 
$$1 \leq \frac{||A^{-1}||}{||B^{-1}||} + ||A^{-1} \Delta||,$$ giving the result in equation (\ref{eq:binv_bound}).

We will also use the result:
\begin{align}\label{eq:oneMinABinv_bound}
(1-||AB||)^{-1} \leq (1-||A|| \ ||B||)^{-1},
\end{align}
derived as follows:
\begin{align*}
& ||AB|| \leq ||A|| \ ||B|| \tag*{(Submultiplicativity of norm)}\\
\implies & -||AB|| \geq -||A|| \ ||B|| \tag*{(Multiply by -1)}\\
\implies & 1 -||AB|| \geq 1 -||A|| \ ||B|| \tag*{(Add 1 to both sides)}\\
\implies & (1 -||AB||)^{-1} \geq (1 -||A|| \ ||B||)^{-1}. \tag*{($a\geq b \implies \frac{1}{a}\leq\frac{1}{b}$)}\\
\end{align*}

Moving to the overall bound on $||A^{-1} - B^{-1}||,$ we have
\begin{align}
||A^{-1} - B^{-1}|| &= ||A^{-1}BB^{-1} - A^{-1}AB^{-1}|| \tag*{}\\
&= ||A^{-1} \Delta B^{-1}|| \tag*{(Factor out terms)}\\
&\leq ||A^{-1} \Delta|| \ ||B^{-1}|| \tag*{(Submultiplicativity of norm)} \\
&\leq \frac{||A^{-1} \Delta|| \ ||A^{-1}||}{1-||A^{-1} \Delta||} \tag*{(Equation (\ref{eq:binv_bound}))} \\
&\leq \frac{||A^{-1}||^2 \ ||\Delta||}{1-||A^{-1} \Delta||} \tag*{(Submultiplicativity of norm)} \\
&\leq \frac{||A^{-1}||^2 \ ||\Delta||}{1-||A^{-1}|| \  ||\Delta||} \tag*{(Equation (\ref{eq:oneMinABinv_bound}))}
\end{align}

Then if $||A-B||$ is bounded, i.e. $||A-B||\leq c$ for some $c>0,$ and $||A^{-1}\Delta||<1$ the bound on $||A^{-1} - B^{-1}||$ can be expressed as
$$||A^{-1} - B^{-1}|| \leq \frac{c \ ||A^{-1}||^2}{1-c \ ||A^{-1}||}.$$

Given that one can choose the bound $c$, it can be ensured that $||A^{-1}\Delta||<1$, because 
\begin{align*}
||A^{-1}\Delta|| &= ||A^{-1}(B - A)|| \\
&\leq ||A^{-1}|| \ ||B-A|| \tag*{(Submultiplicativity of norm)} \\
&\leq c ||A^{-1}|| \tag*{($||B-A|| = ||A-B||\leq c$)}
\end{align*}
so if $c<\frac{1}{||A^{-1}||},$ then $||A^{-1}\Delta|| \leq c ||A^{-1}|| < 1.$


\subsection{Heteroskedastic noise derivation}\label{sec:gp_posterior_hetero}

We can find the GP posterior covariance and expectation under heteroskedastic variance. Recall $\text{Cov}(\bm{f}|\bm{y}, X, \Omega) = K - K(K + D^{-1})^{-1}K.$ By the Sherman-Morrison-Woodbury formula and algebraic manipulation, we have
\begin{align*}
    (K + D^{-1})^{-1} &= K^{-1} - K^{-1}(D+K^{-1})^{-1}K^{-1} \\
    \implies \quad &K(K + D^{-1})^{-1}K = KK^{-1}K - KK^{-1}(D+K^{-1})^{-1}K^{-1}K \\
    \color{white} \implies \quad & \color{white}K(K + D^{-1})^{-1}K \color{black}= K - (D+K^{-1})^{-1}\\
    \implies \quad &K-K(K + D^{-1})^{-1}K = K - K + (K + D^{-1})^{-1} \\
    \color{white}\implies \quad &\color{white}K-K(K + D^{-1})^{-1}K \color{black}= (D+K^{-1})^{-1}\\
    \color{white}\implies \quad &\color{white}K-K(K + D^{-1})^{-1}K \color{black}= KK^{-1}(D+K^{-1})^{-1}\\
    \color{white}\implies \quad &\color{white}K-K(K + D^{-1})^{-1}K \color{black}= K([D+K^{-1}]K)^{-1}\\
    \color{white}\implies \quad &\color{white}K-K(K + D^{-1})^{-1}K \color{black}= K(DK+I)^{-1} \\
    \color{white}\implies \quad &\color{white}K-K(K + D^{-1})^{-1}K \color{black}= K(DK + I)^{-1}DD^{-1} \\
    \color{white}\implies \quad &\color{white}K-K(K + D^{-1})^{-1}K \color{black}= K(K + D^{-1})^{-1}D^{-1} \\
    \color{white}\implies \quad &\color{white}K-K(K + D^{-1})^{-1}K \color{black}= D^{-1}(K + D^{-1})^{-1}K.
\end{align*}

Thus, $\text{Cov}(\bm{f}|\bm{y}, X, \Omega) = K - K(K + D^{-1})^{-1}K = K(DK+I)^{-1} = D^{-1}(K + D^{-1})^{-1}K.$

We can also find the GP posterior mean under heteroskedastic variance. Recall $\mathbb{E}(\bm{f}|\bm{y}, X, \Omega) = K(K + D^{-1})^{-1}\bm{y}.$ Then, via simple algebraic manipulation, we have
\begin{align*}
    K(K + D^{-1})^{-1}\bm{y} &= K(K + D^{-1})^{-1}D^{-1}D\bm{y} \\
    &= K(D [K + D^{-1}])^{-1}D\bm{y}\\
    &= K(DK + I)^{-1}D\bm{y} \\
    &= \text{Cov}(\bm{f}|\bm{y}, X, \Omega)D\bm{y}.
\end{align*}

The bound on the KL divergence, which can be found in equation (\ref{eq:big_KL_pts1to3}), then uses $M=DK+I$ for the GP covariance rather than $M=\tau K+I.$ The bound on part (iii.) of the KL divergence, which can be found in equation (\ref{eq:bnd_partiii}), then only requires the slight modification of replacing $\tau^2||\bm{y}||^2$ with $||D\bm{y}||^2$ due to this adjustment to the GP mean under heteroskedasticity.

The sampling steps in Algorithm 1 also need adjusting to account for heteroskedastic noise. The modification is described below. We avoid working with the non-symmetric matrix $DK + I$ by using the form of the posterior covariance $D^{-1}(K + D^{-1})^{-1}K$; while it may look computationally more intensive, it actually requires fewer HODLR operations than the alternative forms and leads to more stable sampling because all HODLR matrices are symmetric.

\textbf{Algorithm 1 (heteroskedastic)}: Given noisy observations $\bm{y}$ of an underlying function $\bm{f}$ at inputs $X,$ diagonal precision matrix $D,$ GP hyperparameters $\Theta,$ and HODLR specifications tolerance $\epsilon$ and maximum block size $B,$ approximate a sample from $p(\bm{f}|\bm{y}, X, \Theta)$ using a factorization of cost $\mathcal{O}(n\log^2{}n)$ and operations of cost at most $\mathcal{O}(n\log{}n)$. Sampling proceeds as follows: 
\begin{enumerate}[\textit{Step} 1:]
\item Create $\mathcal{H}$ matrix decomposition of $K$ and $K + D^{-1}$ with tolerance $\epsilon$, call these $\tilde{K}$ and $\tilde{P},$ respectively (assembly, cost $\mathcal{O}(n\log{} n)$).
\item Get $W,$ the symmetric factor of $\tilde{K}$ s.t. $\tilde{K}=WW^T$ (factorization, cost $\mathcal{O}(n\log{}^2n)$).
\item Sample $\bm{a} \sim \text{N}(0,D)$ and $\bm{b} \sim \text{N}(0,I),$ both independent $n-$dimensional random vectors.
\item Let $Z = \tilde{K} a  + W b$ (matrix vector products, cost $\mathcal{O}(n\log{}n)$). \label{alg1h_1}
\item Let $Z^* = \tilde{P}^{-1} Z$ (solve linear system, cost $\mathcal{O}(n\log{}n)$).\label{alg1h_2}
\item Let $Z^{**} = D^{-1} Z^*$ (multiply vector by diagonal matrix, cost $\mathcal{O}(n)$).
\item Let $R = \tilde{K} \tilde{P}^{-1} \bm{y}$ (solve linear system then matrix vector product, cost $\mathcal{O}(n\log{}n)$).\label{alg1h_3}
\item Finally, let $Z^{***} = R + Z^{**},$ which is the approximated sample from $p(\bm{f}|\bm{y}, X, \Theta)$ (vector addition, cost $\mathcal{O}(n)$).
\end{enumerate}

\addvspace{\baselineskip}
\begin{lemma}\label{thm:alg1het}
From Algorithm 1.h, $Z^{***} \sim \text{N}(\tilde{\mu}_{f}, \tilde{\Sigma}_{f}),$ where $\tilde{\Sigma}_f = \tilde{K}\tilde{P}^{-1}D^{-1}$ and $\tilde{\mu}_f = \tilde{K}\tilde{P}^{-1}\bm{y}$ are defined to be the approximations of the posterior function variance $\Sigma_f$ and mean $\mu_f,$ respectively.
\end{lemma}

\begin{proof}
Analogous to the proof of Algorithm 1 for homoskedastic variance. Note that step \ref{alg1h_2} relies on the property $(A^{-1})^T = (A^T)^{-1}.$ Specifically, $\text{Cov}(Z^*) = \text{Cov}(\tilde{P}^{-1} Z) = \tilde{P}^{-1} \text{Cov}(Z) (\tilde{P}^{-1})^T,$ implying $\text{Cov}(Z^*) = \tilde{P}^{-1}(\tilde{K}D\tilde{K} + \tilde{K})\tilde{P}^{-1} = \tilde{P}^{-1}\tilde{K}(D\tilde{K} + I)\tilde{P}^{-1} = \tilde{P}^{-1}\tilde{K}D(\tilde{K} + D^{-1})\tilde{P}^{-1} = \tilde{P}^{-1}\tilde{K}D.$
\end{proof} 

\newpage
\section{Supplemental Materials}

\subsection{HODLR decomposition and factorization}\label{sec:symm_fac_alg}

The symmetric factorization of HODLR matrices was described thoroughly in Ambikasaran (2014). Rather than restating the algorithm in full here, we include an example to provide intuition into the decomposition and symmetric factorization.

For the example matrix, $n=200$ uniform data points were sampled in the range $(0,1)$ and then sorted, and the exponential covariance kernel $k(x,x')=\text{exp}(-4(x-x')^2)$ was used. Figure \ref{fig:example_1} shows the matrix and the block structure of the matrix as a 2-level HODLR matrix.

\begin{figure}[!htb]
  \centering
  \includegraphics[width=.45\linewidth]{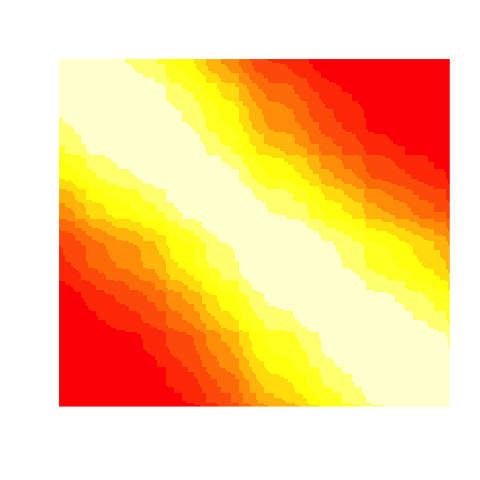}
  \includegraphics[width=.45\linewidth]{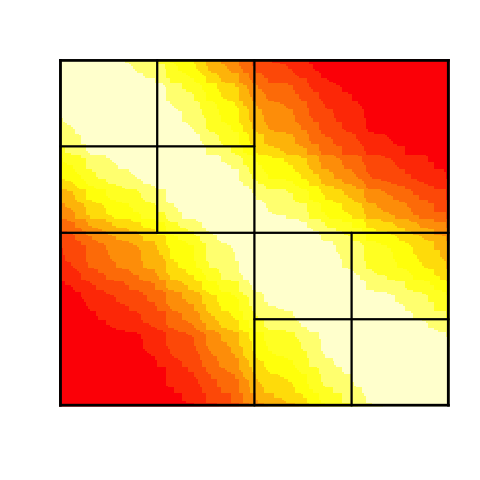}
  \caption{Left: realizations of covariance kernel for $d=1$ with sorted inputs. Right: block partition of a 2-level HODLR tree.}
  \label{fig:example_1}
\end{figure}

After the matrix has been divided into block diagonal HODLR structure, the off-diagonal blocks are approximated to the desired tolerance (i.e., the desired maximum absolute difference between the true and approximated matrix entries). In the example, we set the tolerance to $10^{-8}.$ Figure \ref{fig:example_2} shows the low-rank description of the off-diagonal blocks alongside the elementwise difference between the true and approximated matrices when the tolerance is set to $10^{-8}$. The rank of the large off-diagonal blocks with this tolerance is 7, and that of the smaller off-diagonal blocks is 5. More explicitly, $U_1^{(1)}$ and $V_1^{(1)}$ are $\frac{n}{2} \times 7$ matrices, and $U_1^{(2)}, V_1^{(2)}, U_2^{(2)}$ and $V_2^{(2)}$ are $\frac{n}{4} \times 5$ matrices. Note that the block diagonal elements of the matrix are dense, not approximated by low rank matrices.

\begin{figure}[!htb]
  \centering
  \includegraphics[width=.43\linewidth]{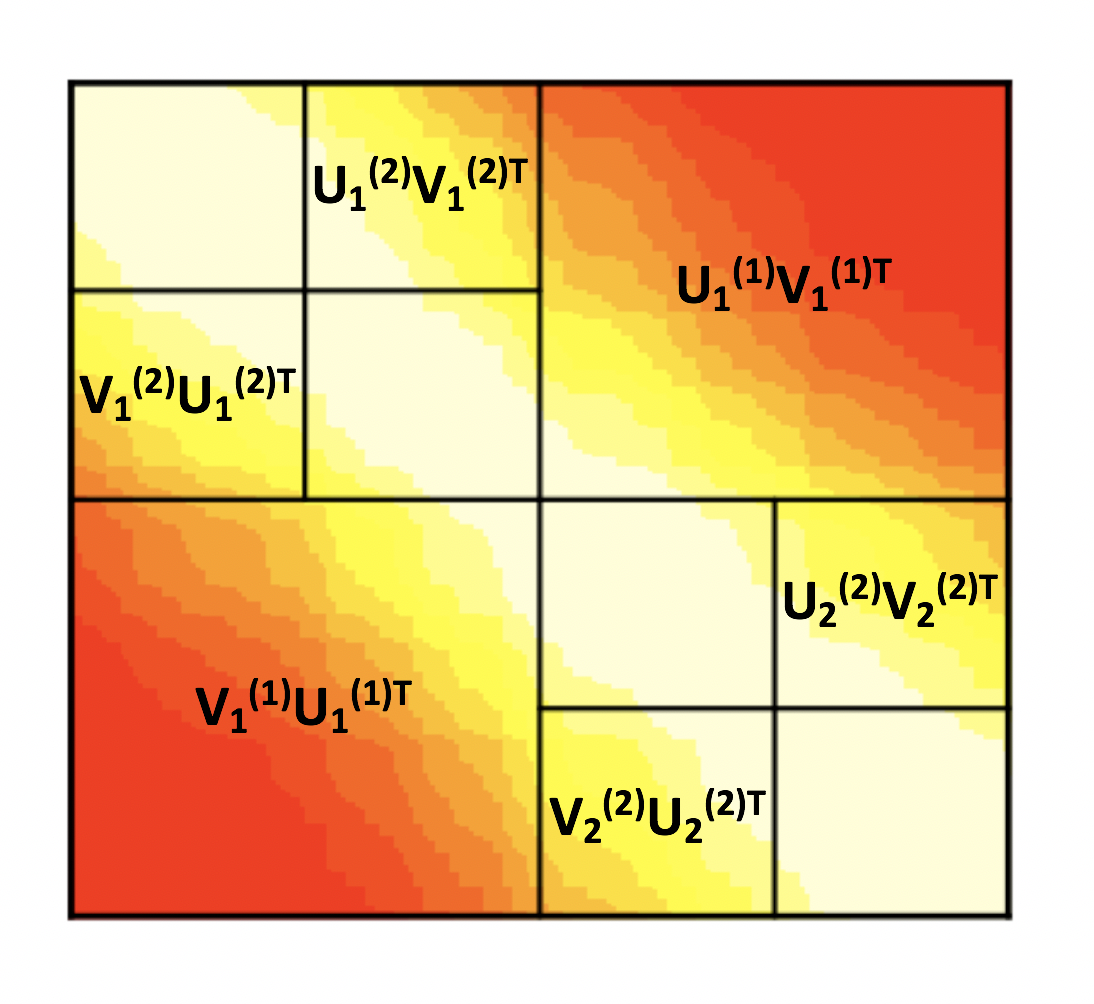}
  \includegraphics[width=.4\linewidth]{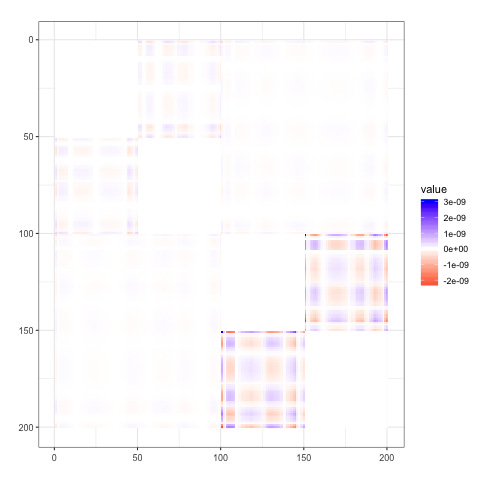}
  \caption{Left: HODLR matrix with low-rank description overlaid on off-diagonal blocks. Right: Difference between the true and approximated matrices when the tolerance for each off-diagonal block is set to $10^{-8}$.}
  \label{fig:example_2}
\end{figure}

The matrix structure shown in Figure \ref{fig:example_2} can be manipulated to provide a symmetric factorization of the original matrix $A$ as $A=WW^T,$ where $W=A_{\ell}A_{\ell-1}\ldots A_1 A_0$ is the product of a block-diagonally dense matrix and $\ell$ matrices that are low-rank updates to the identity (where $\ell$ is the level of the HODLR matrix). To quickly summarize, the first step of the algorithm is to symmetrically factor each of the diagonal blocks of the matrix. Then, the inverse of the relevant block diagonal matrix/matrices is applied to each $U_i^{j},V_i^{j}$, leading to the factorization $A=A_{\ell} \tilde{A}_{\ell} A_{\ell}^T$ where $A_{\ell}$ is a block diagonal matrix with entries being the symmetric factorization of the diagonal blocks of $A$ and diagonal blocks of $\tilde{A}_{\ell}$ are low-rank updates to the identity matrix. These diagonal blocks can be quickly factorized using the Sherman-Morrison-Woodbury formula and Theorem 3.1 in  Ambikasaran (2014). The algorithm proceeds iteratively until reaching the 0th level. Full details are provided in Ambikasaran (2014).

\subsubsection{Cost and dimension of input space}\label{SI_cost_inputdim}

Figure \ref{fig:offdiag_example} shows realizations of covariance kernels for randomly generated $\bm{x}$ having $d=1$ and $d=2$ with the large off-diagonal block emphasized. These $\frac{n}{2} \times \frac{n}{2}$ off-diagonal blocks can be approximated by rank $r$ matrices. If one requires the largest absolute difference to be less than $10^{-12},$ the $d=1$ example shown requires $r=8$ while the $d=2$ example requires $r=56$ (here we use the SVD because this is a small $n=200$ example). The higher the dimension of the input space, the more complex the structure in the covariance matrix even after sorting and the more costly storage and computation become. In the example, the storage cost of the dense $100 \times 100$ off-diagonal block is 80,000 bytes and that of the SVD approximation with $r=8$ is 12,800 bytes, with an equal number of flops required for matrix-vector multiplication. The real savings come as $n$ increases. With the same hyperparameters and range of data, setting $n=2000$ leads to a storage cost for the dense off-diagonal block matrix of 8,000,000 bytes but only requires an $r=9$ matrix for the same fidelity of approximation (i.e., 144,000 bytes of storage).

\begin{figure}[!htb]
  \centering
  \includegraphics[width=.45\linewidth]{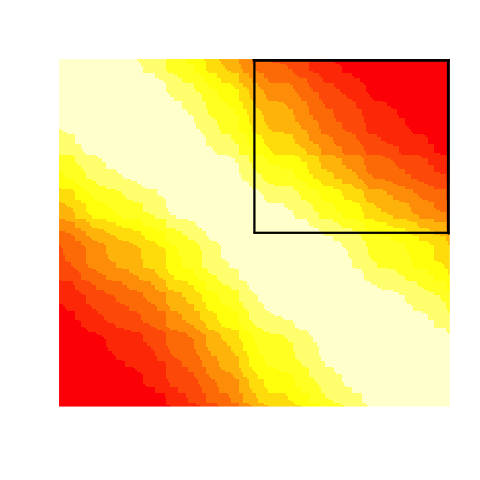}
  \includegraphics[width=.45\linewidth]{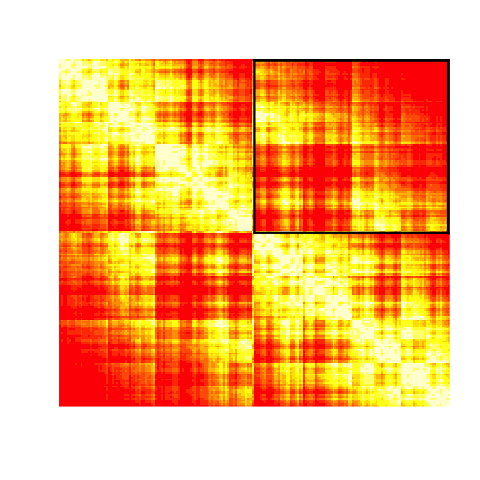}
  \caption{Left: realizations of covariance kernel for $d=1$ with sorted inputs. Right: realizations of covariance kernel for $d=2$ with inputs sorted via a $kd$-tree.}
  \label{fig:offdiag_example}
\end{figure}

\subsubsection{Determinant computation speedup}

As an example of how the HODLR factorization enables fast algebraic operations, consider the calculation of $\text{det}(A)$, the determinant of $A$. We will rely on Sylvester’s Determinant Theorem and two facts regarding the determinant: first, for square matrices the determinant of the product of the matrices is equal to the product of the determinants, and second, the determinant of a block diagonal matrix is the product of the determinants of the individual blocks. Suppose $A$ is factorized as in Figure \ref{fig:hodlrmatrix_fac} into $A_2 A_1 A_0.$ The only full-rank matrices are those on the diagonal of $A_2,$ and the number of levels in the factorization can be chosen such that these are of sufficiently low dimension that computing the determinant of any individual block is of low computational expense. The blocks of $A_1$ and $A_0$ are low rank updates to the identity and thus can be computed using Sylvester’s Determinant Theorem, which says that $\text{det}(I_m + TS) = \text{det}(I_n + ST)$ for $S \in \mathbb{R}^{n \times p}$ and $T \in \mathbb{R}^{p \times n}$.

\subsection{Additional simulation results}\label{sec:sim_replicates}

\subsubsection{Replicate simulation runs}

\begin{figure}[htp]
    \centering
    \includegraphics[width=0.95\textwidth]{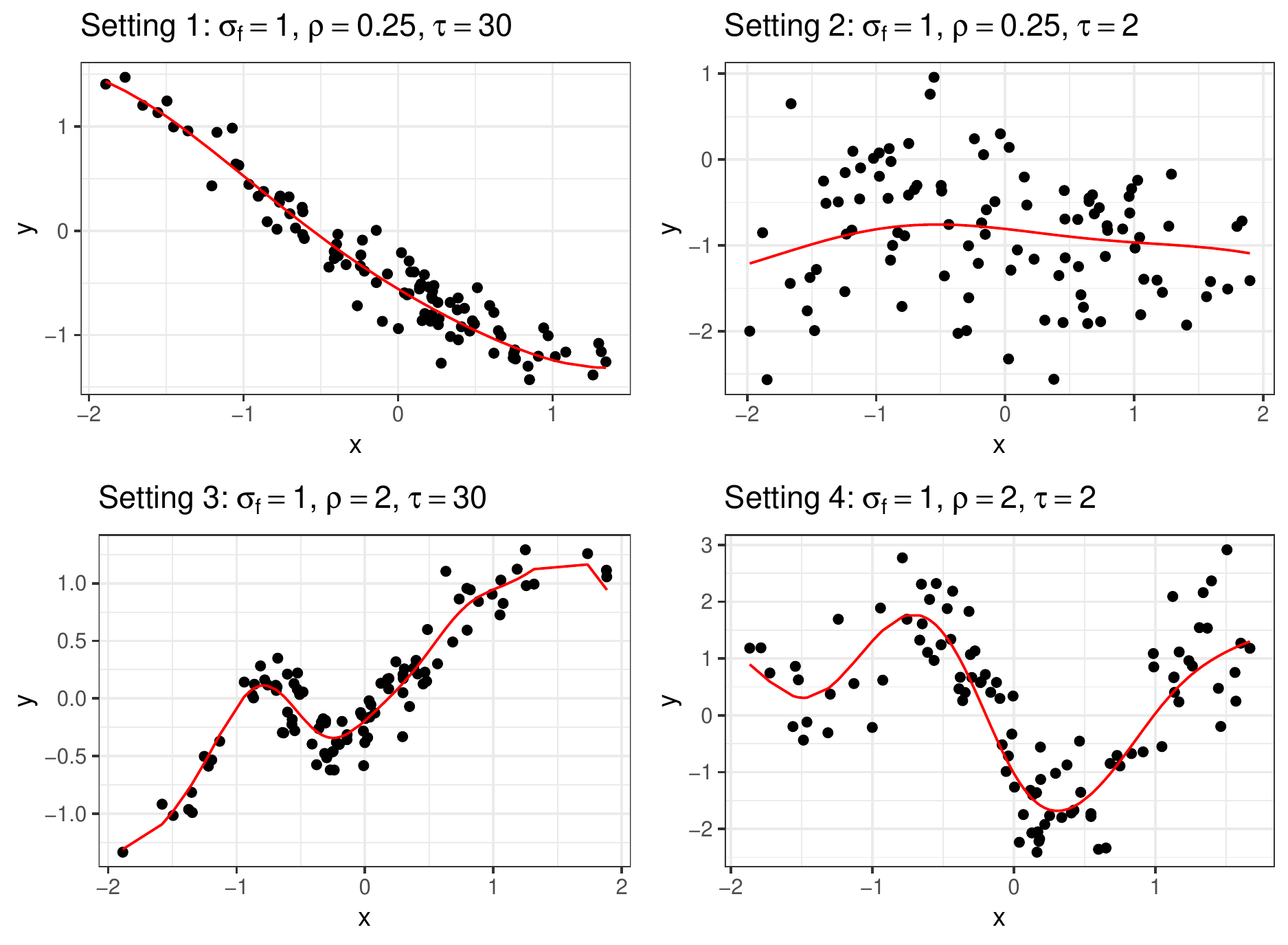}
    \caption{Data and true parameter values from each simulation setting.}
    \label{fig:simRepData}
\end{figure}

\begin{figure}[htp]
    \centering
    \includegraphics[height=0.95\textheight]{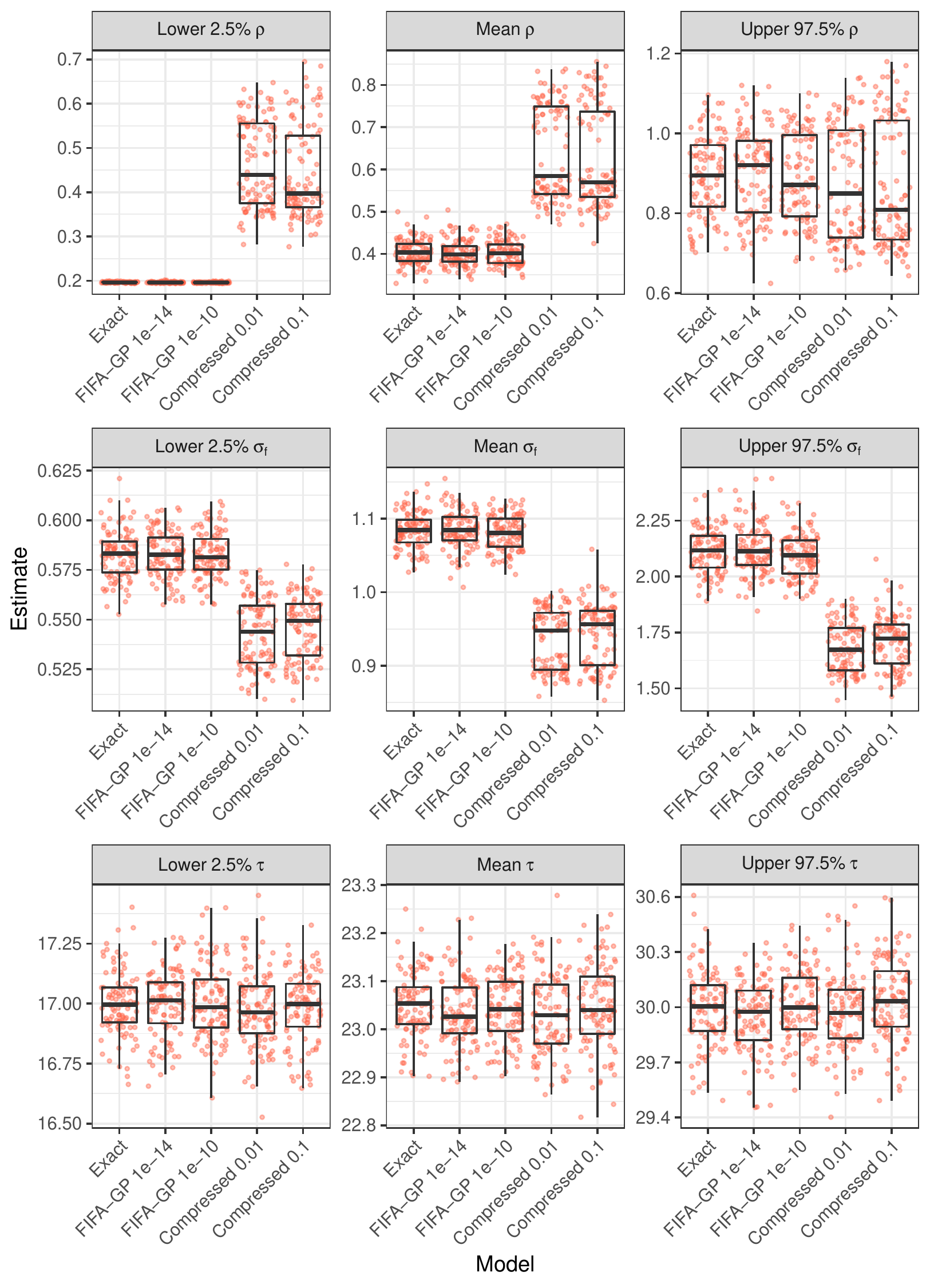}
    \caption{Parameter estimates from each sampler run 100 times under Setting 1.}
    \label{fig:simRepParams1}
\end{figure}

\begin{figure}[htp]
    \centering
    \includegraphics[height=0.95\textheight]{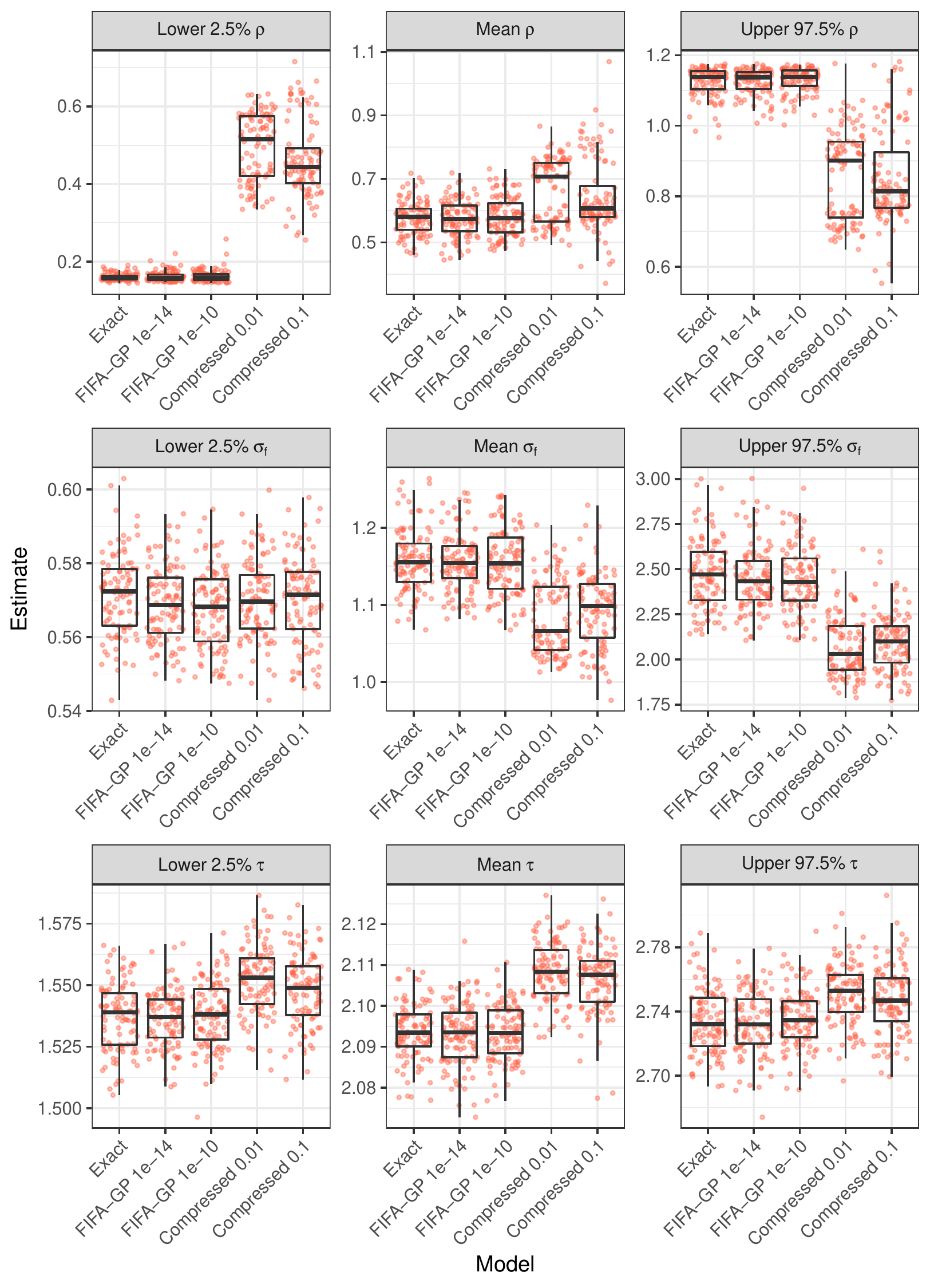}
    \caption{Parameter estimates from each sampler run 100 times under Setting 2.}
    \label{fig:simRepParams2}
\end{figure}

\begin{figure}[htp]
    \centering
    \includegraphics[height=0.95\textheight]{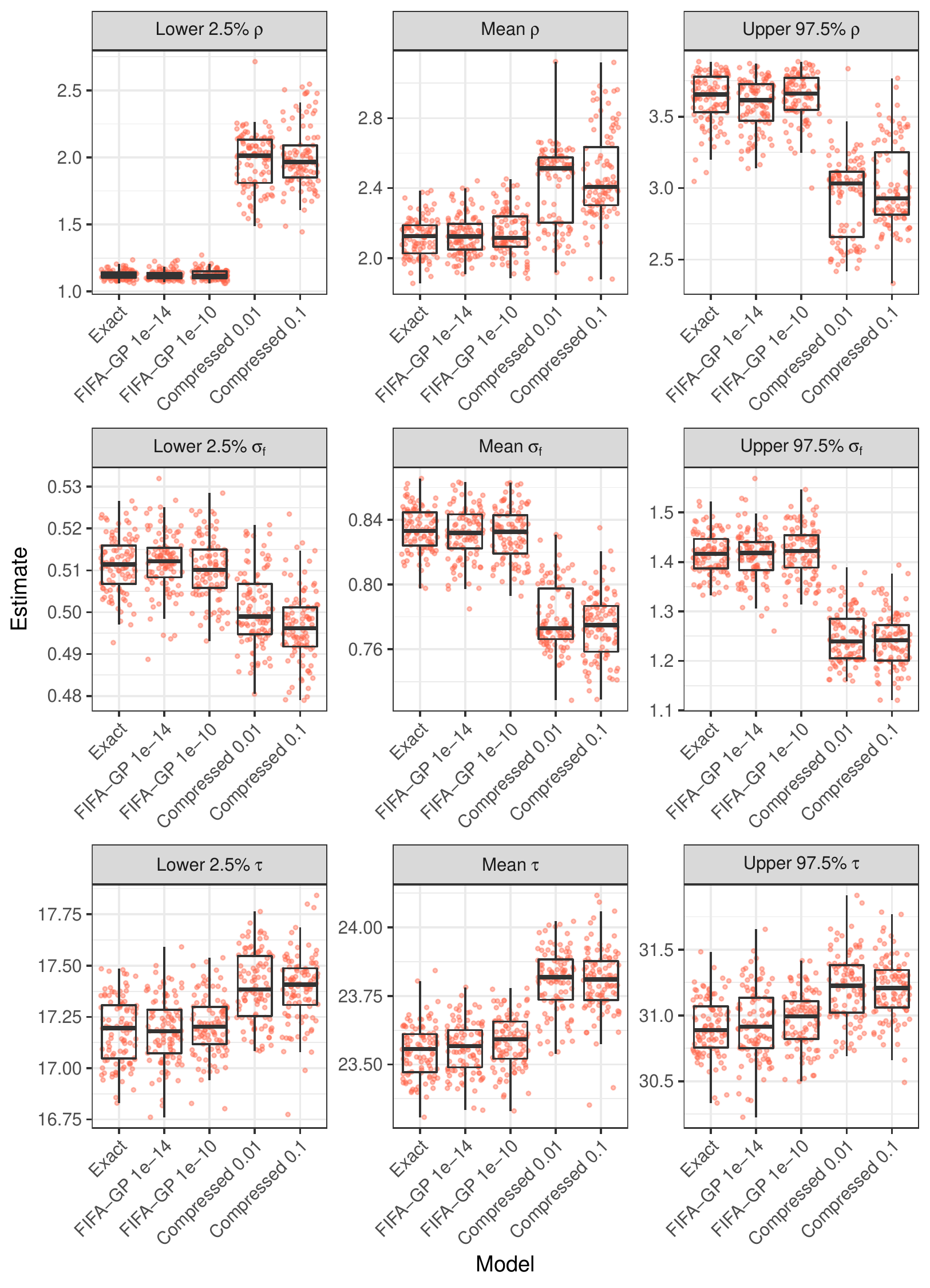}
    \caption{Parameter estimates from each sampler run 100 times under Setting 3.}
    \label{fig:simRepParams3}
\end{figure}

\begin{figure}[htp]
    \centering
    \includegraphics[height=0.95\textheight]{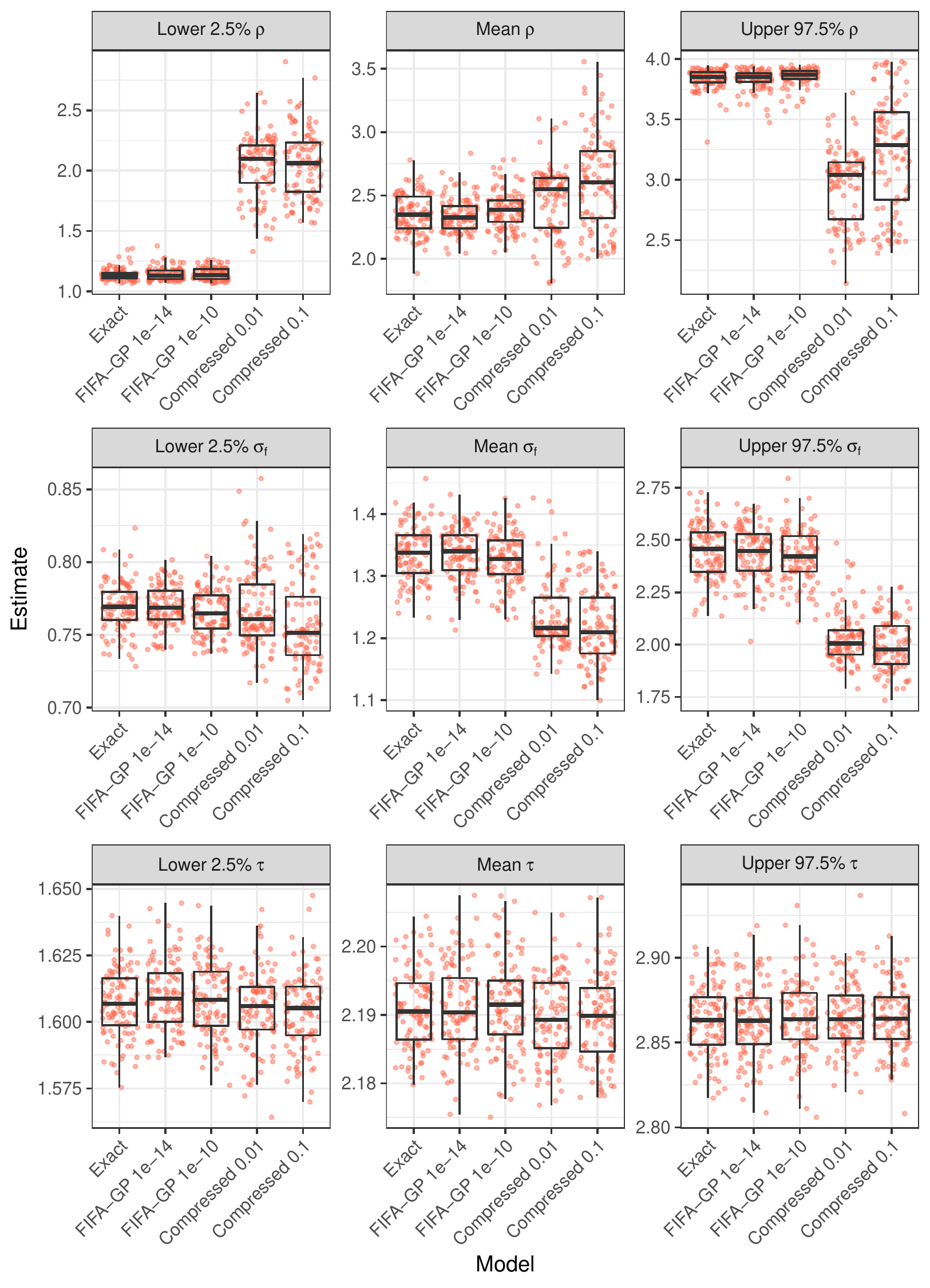}
    \caption{Parameter estimates from each sampler run 100 times under Setting 4.}
    \label{fig:simRepParams4}
\end{figure}

\begin{figure}[htp]
    \centering
    \includegraphics[height=0.75\textheight]{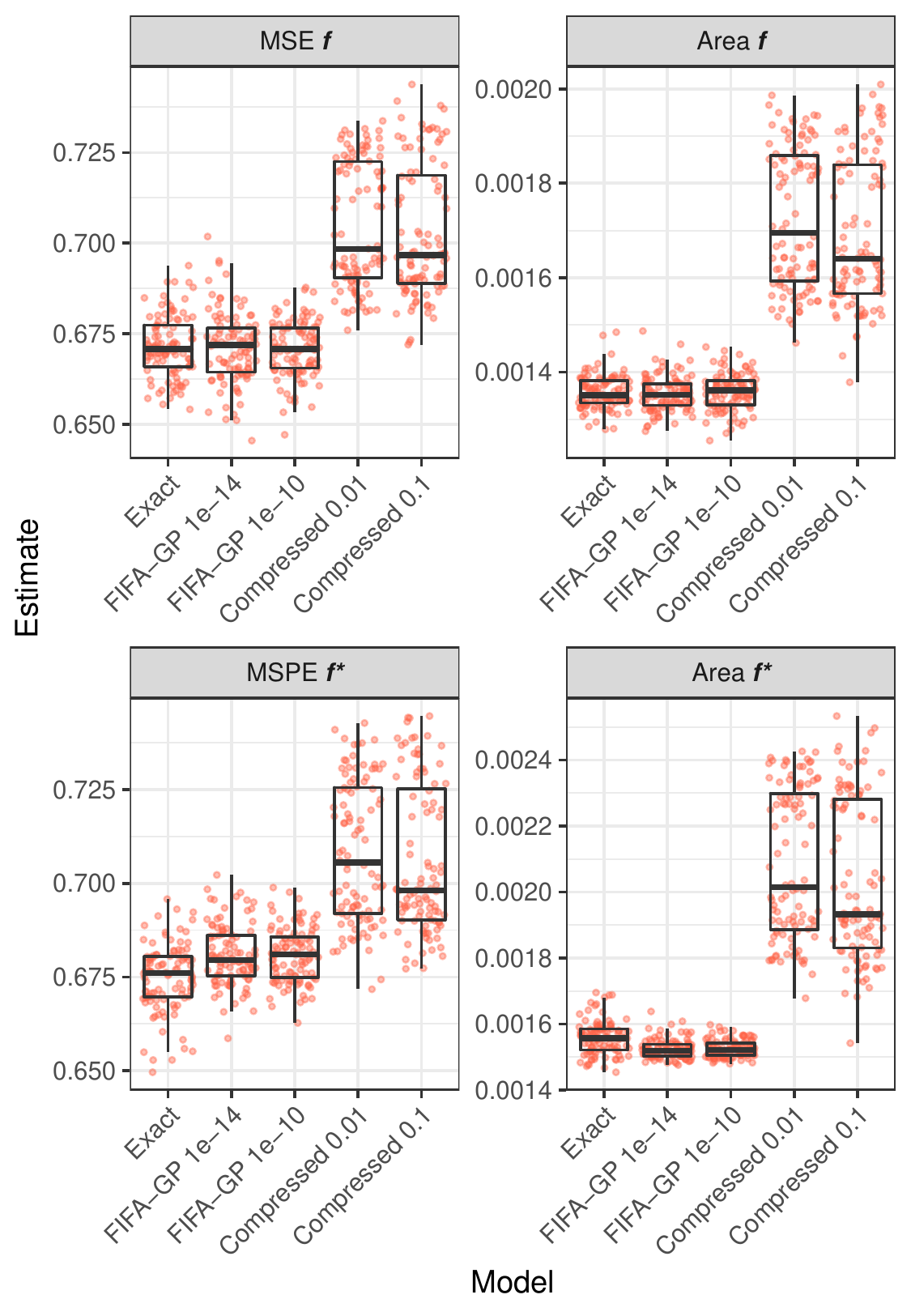}
    \caption{Performance results from each sampler run 100 times under Setting 1.}
    \label{fig:simRepPerf1}
\end{figure}

\begin{figure}[htp]
    \centering
    \includegraphics[height=0.75\textheight]{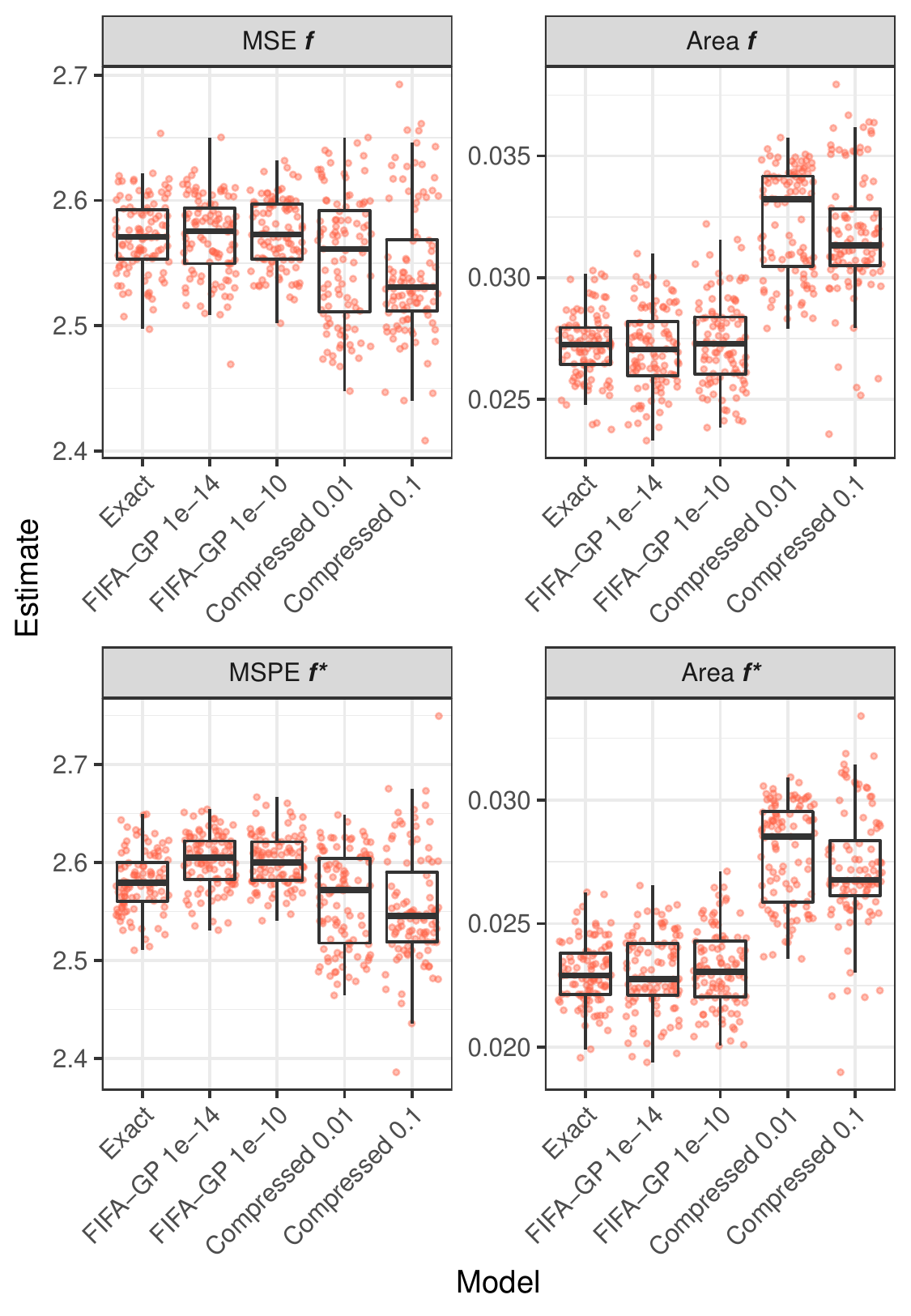}
    \caption{Performance results from each sampler run 100 times under Setting 2.}
    \label{fig:simRepPerf2}
\end{figure}

\begin{figure}[htp]
    \centering
    \includegraphics[height=0.75\textheight]{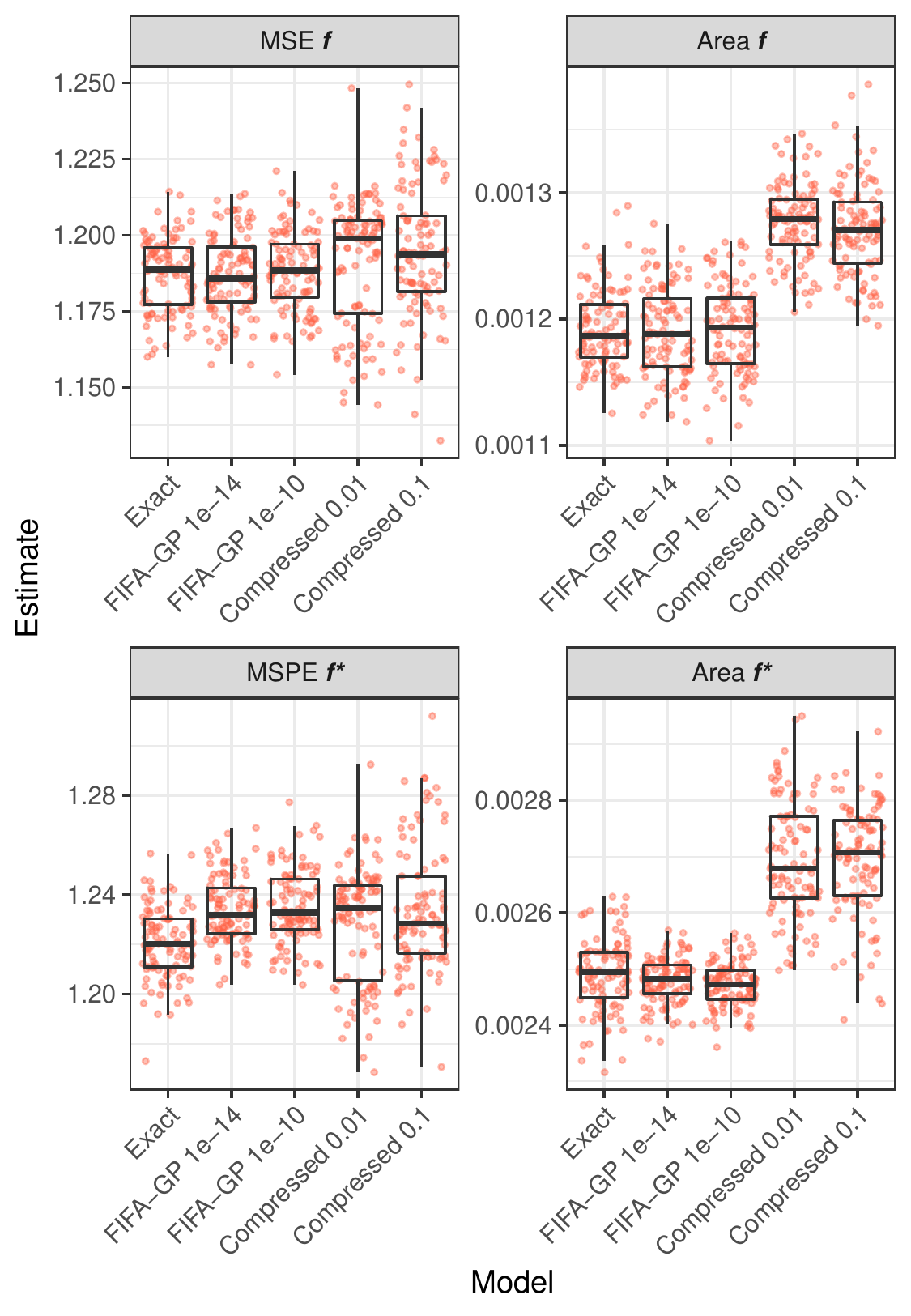}
    \caption{Performance results from each sampler run 100 times under Setting 3.}
    \label{fig:simRepPerf3}
\end{figure}

\begin{figure}[htp]
    \centering
    \includegraphics[height=0.75\textheight]{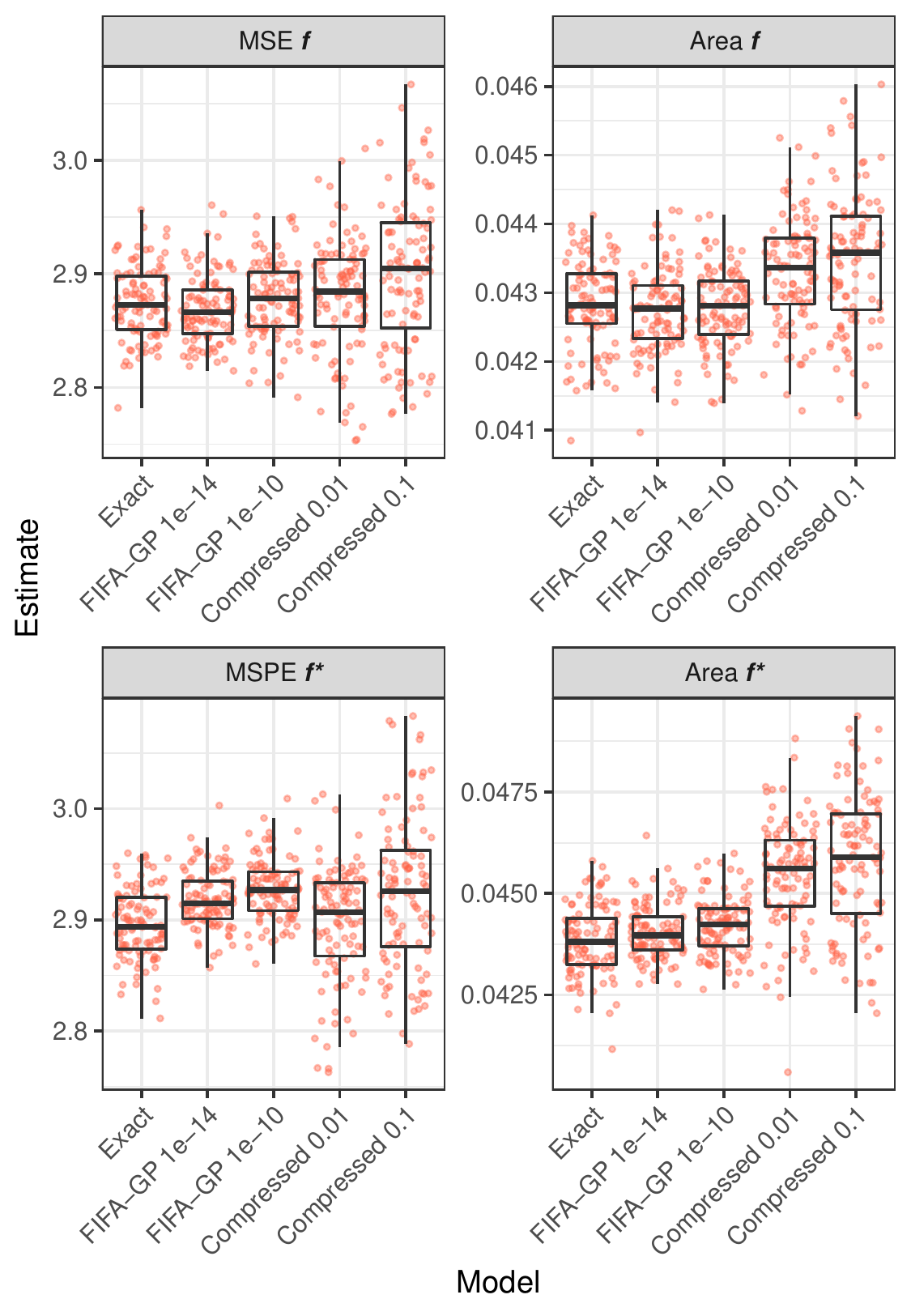}
    \caption{Performance results from each sampler run 100 times under Setting 4.}
    \label{fig:simRepPerf4}
\end{figure}

\newpage

\begin{landscape}

\subsubsection{Comparison tables for $n=100$ and $n=500$}

\begin{table}[!phtb]
    \begin{tabular}{l c c c c c c c c }\toprule
        & & Truth & Exact GP & \multicolumn{2}{c}{FIFA-GP}  & \multicolumn{2}{c}{Compressed GP} \\
        \cmidrule(r){3-3} \cmidrule(r){4-4}\cmidrule(l){5-6}\cmidrule(l){7-8}
        &&& & $\epsilon_{\text{max}}=10^{-14}$ & $\epsilon_{\text{max}}=10^{-10}$ & $\epsilon_{\text{fro}}=0.01$ & $\epsilon_{\text{fro}}=0.1$ \\\midrule
        \rule{0pt}{4ex}Smooth and & $\text{MSPE}_{f^*}$ & - & 0.001 & 0.001 & 0.001 & 0.001 & 0.001 \\
            low noise & 95\% CI ($\bm{f}^*$) Area & - & 0.35 & 0.35 & 0.35 & 0.38 & 0.39 \\
            &$\hat{\tau} \ [\tau_{\text{ll}},\tau_{\text{ul}}]$ & $\tau = 30$ & 30.4 [26.8, 34.4] & 30.4 [26.6, 34.3] & 30.4 [26.7, 34.3] & 30.4 [26.7, 34.2] & 30.4 [26.7, 34.4] \\
             &$\hat{\sigma}_f \ [\sigma_{f,\text{ll}},\sigma_{f,\text{ul}}]$ & $\sigma_f=1$ & 1.11 [0.56, 2.39] & 1.25 [0.64, 2.52] & 1.26 [0.61, 2.69] & 0.85 [0.52, 1.49] & 0.79 [0.48, 1.29] \\
             &$\hat{\rho} \ [\rho_{\text{ll}},\rho_{\text{ul}}]$ & $\rho=0.25$ & 0.36 [0.17, 0.85] & 0.31 [0.16, 0.58] & 0.32 [0.16, 0.61] & 0.55 [0.51, 0.69] & 0.70 [0.59, 0.81] \\
             & Time (min) & - & 15.3 & 4.2 & 3.9 & 1.2 & 1.2 \\
        \rule{0pt}{4ex}Smooth and & $\text{MSPE}_{f^*}$ & - & 0.003 & 0.004 & 0.003 & 0.002 & 0.003 \\
            high noise & 95\% CI ($\bm{f}^*$) Area & - & 1.15 & 1.12 & 1.16 & 1.25 & 1.19 \\
            &$\hat{\tau} \ [\tau_{\text{ll}},\tau_{\text{ul}}]$ & $\tau = 2$ & 1.82 [1.60, 2.06] & 1.83 [1.60, 2.06] & 1.82 [1.60, 2.05] & 1.82 [1.60, 2.05] & 1.82 [1.60, 2.04] \\
             &$\hat{\sigma}_f \ [\sigma_{f,\text{ll}},\sigma_{f,\text{ul}}]$ & $\sigma_f=1$ & 0.57 [0.31, 1.12] & 0.54 [0.31, 1.01] & 0.56 [0.31, 1.11] & 0.48 [0.29, 0.84] & 0.52 [0.30, 0.94] \\
             &$\hat{\rho} \ [\rho_{\text{ll}},\rho_{\text{ul}}]$ & $\rho=0.25$ & 0.39 [0.13, 1.10] & 0.33 [0.13, 0.89] & 0.41 [0.14, 1.10] & 0.68 [0.58, 0.86] & 0.52 [0.43, 0.65] \\
             & Time (min) & - & 20.0 & 4.2 & 3.9 & 1.2 & 1.1 \\
        \rule{0pt}{4ex}Wiggly and & $\text{MSPE}_{f^*}$ & - & 0.001 & 0.001 & 0.001 & 0.001 & 0.001 \\
            low noise & 95\% CI ($\bm{f}^*$) Area & - & 0.53 & 0.53 & 0.53 & 0.53 & 0.54 \\
            &$\hat{\tau} \ [\tau_{\text{ll}},\tau_{\text{ul}}]$ & $\tau = 30$ & 28.7 [25.2, 32.3] & 28.7 [25.2, 32.5] & 28.7 [25.2, 32.2] & 28.6 [25.2, 32.3] & 28.7 [25.3, 32.3] \\
             &$\hat{\sigma}_f \ [\sigma_{f,\text{ll}},\sigma_{f,\text{ul}}]$ & $\sigma_f=1$ & 1.09 [0.68, 1.82] & 1.07 [0.65, 1.82] & 1.03 [0.65, 1.66] & 0.96 [0.63, 1.48] & 0.94 [0.64, 1.43] \\
             &$\hat{\rho} \ [\rho_{\text{ll}},\rho_{\text{ul}}]$ & $\rho=2$ & 2.05 [1.41, 3.13] & 2.10 [1.43, 3.06] & 2.14 [1.50, 3.27] & 2.34 [2.09, 2.65] & 2.39 [2.17, 2.51] \\
             & Time (min) & - & 20.1 & 4.7 & 4.3 & 1.4 & 1.3 \\
         \rule{0pt}{4ex}Wiggly and & $\text{MSPE}_{f^*}$ & - & 0.015 & 0.015 & 0.015 & 0.015 & 0.015 \\
            high noise & 95\% CI ($\bm{f}^*$) Area & - & 1.83 & 1.81 & 1.82 & 1.76 & 1.75 \\
            &$\hat{\tau} \ [\tau_{\text{ll}},\tau_{\text{ul}}]$ & $\tau = 2$ & 2.05 [1.80, 2.31] & 2.05 [1.80, 2.32] & 2.05 [1.80, 2.32] & 2.05 [1.81, 2.32] & 2.05 [1.80, 2.31] \\
             &$\hat{\sigma}_f \ [\sigma_{f,\text{ll}},\sigma_{f,\text{ul}}]$ & $\sigma_f=1$ & 1.03 [0.64, 1.68] & 1.07 [0.66, 1.85] & 1.07 [0.63, 1.90] & 1.10 [0.69, 1.81] & 1.13 [0.70, 1.85] \\
             &$\hat{\rho} \ [\rho_{\text{ll}},\rho_{\text{ul}}]$ & $\rho=2$ & 2.93 [1.53, 3.92] & 2.69 [1.44, 3.89] & 2.80 [1.47, 3.93] & 2.51 [2.23, 2.74] & 2.31 [2.14, 2.56] \\
             & Time (min) & - & 20.1 & 5.0 & 4.4 & 1.4 & 1.4 \\
        \\ \bottomrule
    \end{tabular}
    \caption{Performance results, parameter estimates, and computing time for Bayesian GP regression with $n=500$ data points simulated using a squared exponential covariance kernel. Time shown is total time for setup and 27,000 iterations through Gibbs sampler, with 2,500 samples retained.}\label{tab:sim_smalln500}
\end{table}

\end{landscape}

\begin{landscape}

\begin{table}[!phtb]
    \begin{tabular}{l c c c c c c c c }\toprule
        & & Truth & Exact GP & \multicolumn{2}{c}{FIFA-GP}  & \multicolumn{2}{c}{Compressed GP} \\
        \cmidrule(r){3-3} \cmidrule(r){4-4}\cmidrule(l){5-6}\cmidrule(l){7-8}
        &&& & $\epsilon_{\text{max}}=10^{-14}$ & $\epsilon_{\text{max}}=10^{-10}$ & $\epsilon_{\text{fro}}=0.01$ & $\epsilon_{\text{fro}}=0.1$ \\\midrule
        \rule{0pt}{4ex}Smooth and & $\text{MSPE}_{f^*}$ & - & 0.005 & 0.005 & 0.005 & 0.006 & 0.006 \\
            low noise & 95\% CI ($\bm{f}^*$) Area & - & 0.94 & 0.95 & 0.94 & 1.01 & 1.01 \\
            &$\hat{\tau} \ [\tau_{\text{ll}},\tau_{\text{ul}}]$ & $\tau = 30$ & 22.5 [16.5, 29.6] & 22.6 [16.7, 29.5] & 22.5 [16.7, 29.3] & 22.4 [16.6, 29.2] & 22.4 [16.4, 29.3] \\
             &$\hat{\sigma}_f \ [\sigma_{f,\text{ll}},\sigma_{f,\text{ul}}]$ & $\sigma_f=1$ & 1.38 [0.71, 2.88] & 1.38 [0.73, 2.85] & 1.42 [0.73, 2.95] & 1.11 [0.66, 1.91] & 1.11 [0.66, 1.92] \\
             &$\hat{\rho} \ [\rho_{\text{ll}},\rho_{\text{ul}}]$ & $\rho=0.25$ & 0.37 [0.17, 0.74] & 0.37 [0.17, 0.79] & 0.35 [0.17, 0.76] & 0.60 [0.41, 0.81] & 0.59 [0.41, 0.91] \\
             & Time (min) & - & 0.5 & 0.9 & 0.8 & 0.2 & 0.2 \\
        \rule{0pt}{4ex}Smooth and & $\text{MSPE}_{f^*}$ & - & 0.035 & 0.035 & 0.034 & 0.037 & 0.037 \\
            high noise & 95\% CI ($\bm{f}^*$) Area & - & 2.19 & 2.23 & 2.20 & 2.36 & 2.28 \\
            &$\hat{\tau} \ [\tau_{\text{ll}},\tau_{\text{ul}}]$ & $\tau = 2$ & 2.07 [1.54, 2.72] & 2.08 [1.56, 2.71] & 2.08 [1.54, 2.69] & 2.06 [1.51, 2.70] & 2.07 [1.53, 2.69] \\
             &$\hat{\sigma}_f \ [\sigma_{f,\text{ll}},\sigma_{f,\text{ul}}]$ & $\sigma_f=1$ & 0.64 [0.32, 1.33] & 0.63 [0.32, 1.28] & 0.65 [0.33, 1.43] & 0.57 [0.31, 1.14] & 0.60 [0.31, 1.20] \\
             &$\hat{\rho} \ [\rho_{\text{ll}},\rho_{\text{ul}}]$ & $\rho=0.25$ & 0.49 [0.15, 1.07] & 0.52 [0.15, 1.13] & 0.47 [0.15, 1.14] & 0.81 [0.59, 1.13] & 0.62 [0.43, 0.92] \\
             & Time (min) & - & 0.6 & 0.9 & 0.8 & 0.2 & 0.3  \\
        \rule{0pt}{4ex}Wiggly and & $\text{MSPE}_{f^*}$ & - & 0.005 & 0.005 & 0.005 & 0.004 & 0.004 \\
            low noise & 95\% CI ($\bm{f}^*$) Area & - & 1.14 & 1.12 & 1.13 & 1.17 & 1.15 \\
            &$\hat{\tau} \ [\tau_{\text{ll}},\tau_{\text{ul}}]$ & $\tau = 30$ & 23.0 [16.8, 30.4] & 22.9 [16.9, 29.8] & 23.1 [16.8, 30.2] & 23.2 [16.9, 30.4] & 23.1 [16.8, 30.4] \\
             &$\hat{\sigma}_f \ [\sigma_{f,\text{ll}},\sigma_{f,\text{ul}}]$ & $\sigma_f=1$ & 0.84 [0.51, 1.45] & 0.87 [0.52, 1.52] & 0.85 [0.52, 1.44] & 0.78 [0.49, 1.29] & 0.79 [0.51, 1.26] \\
             &$\hat{\rho} \ [\rho_{\text{ll}},\rho_{\text{ul}}]$ & $\rho=2$ & 1.89 [1.02, 3.74] & 1.73 [1.02, 3.47] & 1.83 [1.03, 3.39] & 2.30 [1.95, 2.68] & 2.20 [1.53, 2.71] \\
             & Time (min) & - & 0.6 & 1.0 & 0.9 & 0.3 & 0.3 \\
         \rule{0pt}{4ex}Wiggly and & $\text{MSPE}_{f^*}$ & - & 0.054 & 0.052 & 0.053 & 0.054 & 0.053 \\
            high noise & 95\% CI ($\bm{f}^*$) Area & - & 3.10 & 3.12 & 3.12 & 3.09 & 3.05 \\
            &$\hat{\tau} \ [\tau_{\text{ll}},\tau_{\text{ul}}]$ & $\tau = 2$ & 2.32 [1.68, 3.04] & 2.31 [1.68, 3.02] & 2.31 [1.69, 3.05] & 2.32 [1.70, 3.05] & 2.31 [1.67, 3.04] \\
             &$\hat{\sigma}_f \ [\sigma_{f,\text{ll}},\sigma_{f,\text{ul}}]$ & $\sigma_f=1$ & 0.94 [0.55, 1.69] & 0.95 [0.55, 1.71] & 0.94 [0.54, 1.7] & 0.94 [0.54, 1.65] & 0.97 [0.56, 1.73] \\
             &$\hat{\rho} \ [\rho_{\text{ll}},\rho_{\text{ul}}]$ & $\rho=2$ & 2.96 [1.45, 3.96] & 2.81 [1.29, 3.92] & 2.86 [1.38, 3.93] & 2.81 [2.35, 3.33] & 2.64 [2.08, 3.16] \\
             & Time (min) & - & 0.6 & 1.0 & 1.0 & 0.3 & 0.3 \\
        \\ \bottomrule
    \end{tabular}
    \caption{Performance results, parameter estimates, and computing time for Bayesian GP regression with $n=100$ data points simulated using a squared exponential covariance kernel. Time shown is total time for setup and 27,000 iterations through Gibbs sampler, with 2,500 samples retained.}\label{tab:sim_smalln100}
\end{table} 

\end{landscape}

\begin{landscape}

\subsubsection{Comparison tables for non-GP true function}\label{SIsec:large_n_compare}

\begin{table}[!phtb]
    \begin{tabular}{l c c c c c c c }\toprule
        & & Exact GP & \multicolumn{2}{c}{FIFA-GP}  & \multicolumn{2}{c}{Compressed GP} \\ \cmidrule(r){3-3}\cmidrule(l){4-5}\cmidrule(l){6-7}
        &&& $\epsilon_{\text{max}}=10^{-14}$ & $\epsilon_{\text{max}}=10^{-10}$ & $\epsilon_{\text{fro}}=0.01$ & $\epsilon_{\text{fro}}=0.1$ \\\midrule
        \rule{0pt}{4ex}$n=100$ & $\text{MSPE}_{f^*}$ & 0.025 & 0.023 & 0.024 & 0.024 & 0.024 \\
            & 95\% CI ($\bm{f}^*$) Area & 2.75 & 2.72 & 2.69 & 2.67 & 2.68 \\
            &$\hat{\tau} \ [\tau_{\text{ll}},\tau_{\text{ul}}]$ & 1.85 [1.32, 2.38] & 1.87 [1.37, 2.37] & 1.86 [1.38, 2.40] & 1.87 [1.31, 2.49] & 1.85 [1.37, 2.47] \\
             &$\hat{\sigma}_f \ [\sigma_{f,\text{ll}},\sigma_{f,\text{ul}}]$ & 0.93 [0.47, 2.16] & 1.10 [0.49, 3.57] & 0.82 [0.47, 1.41] & 0.85 [0.48, 1.53] & 0.82 [0.48, 1.37] \\
             &$\hat{\rho} \ [\rho_{\text{ll}},\rho_{\text{ul}}]$ & 1.11 [0.28, 1.98] & 1.04 [0.21, 1.97] & 1.13 [0.47, 1.94] & 1.21 [1.06, 1.36] & 1.18 [0.89, 1.29] \\
             & Time (min) & 0.1 & 0.1 & 0.1 & 0.0 & 0.0 \\
        \rule{0pt}{4ex}$n=500$ & $\text{MSPE}_{f^*}$ & 0.014 & 0.013 & 0.012 & 0.012 & 0.015 \\
            & 95\% CI ($\bm{f}^*$) Area & 1.49 & 1.56 & 1.54 & 1.48 & 1.57 \\
            &$\hat{\tau} \ [\tau_{\text{ll}},\tau_{\text{ul}}]$ & 1.94 [1.73, 2.18] & 1.94 [1.70, 2.20] & 1.93 [1.71, 2.18] & 1.93 [1.70, 2.16] & 1.93 [1.70, 2.19] \\
             &$\hat{\sigma}_f \ [\sigma_{f,\text{ll}},\sigma_{f,\text{ul}}]$ & 0.97 [0.53, 2.02] & 1.02 [0.51, 2.48] & 0.93 [0.52, 1.88] & 0.82 [0.51, 1.32] & 0.86 [0.48, 1.85] \\
             &$\hat{\rho} \ [\rho_{\text{ll}},\rho_{\text{ul}}]$ & 1.02 [0.43, 1.81] & 1.05 [0.39, 1.78] & 1.01 [0.39, 1.92] & 1.05 [1.05, 1.05] & 1.28 [1.28, 1.28] \\
             & Time (min) & 2.0 & 0.6 & 0.5 & 0.2 & 0.2 \\
        \rule{0pt}{4ex}$n=1000$ & $\text{MSPE}_{f^*}$ & 0.007 & 0.006 & 0.008 & 0.007 & 0.007 \\
            & 95\% CI ($\bm{f}^*$) Area & 1.12 & 1.11 & 1.11 & 1.08 & 1.06 \\
            &$\hat{\tau} \ [\tau_{\text{ll}},\tau_{\text{ul}}]$ & 2.02 [1.84, 2.21] & 2.02 [1.86, 2.20] & 2.02 [1.84, 2.20] & 2.01 [1.84, 2.19] & 2.02 [1.86, 2.21] \\
             &$\hat{\sigma}_f \ [\sigma_{f,\text{ll}},\sigma_{f,\text{ul}}]$ & 0.99 [0.48, 2.16] & 0.92 [0.54, 1.68] & 0.80 [0.47, 1.59] & 0.82 [0.57, 1.16] & 0.87 [0.56, 1.29] \\
             &$\hat{\rho} \ [\rho_{\text{ll}},\rho_{\text{ul}}]$ & 1.02 [0.35, 1.86] & 0.98 [0.41, 1.86] & 1.19 [0.61, 1.91] & 1.05 [1.05, 1.05] & 1.05 [1.05, 1.05] \\
             & Time (min) & 12.7 & 1.2 & 1.1 & 0.8 & 0.8 \\
         \rule{0pt}{4ex}$n=2000$ & $\text{MSPE}_{f^*}$ & 0.002 & 0.002 & 0.002 & 0.002 & 0.002 \\
            & 95\% CI ($\bm{f}^*$) Area & 0.74 & 0.77 & 0.74 & 0.79 & 0.80 \\
            &$\hat{\tau} \ [\tau_{\text{ll}},\tau_{\text{ul}}]$ & 2.00 [1.87, 2.13] & 2.00 [1.88, 2.13] & 2.00 [1.88, 2.11] & 2.00 [1.88, 2.14] & 2.00 [1.87, 2.11] \\
             &$\hat{\sigma}_f \ [\sigma_{f,\text{ll}},\sigma_{f,\text{ul}}]$ & 1.37 [0.65, 2.19] & 0.92 [0.58, 1.66] & 1.27 [0.61, 2.28] & 0.81 [0.51, 1.40] & 0.80 [0.50, 1.52] \\
             & $\hat{\rho} \ [\rho_{\text{ll}},\rho_{\text{ul}}]$ & 0.50 [0.24, 0.95] & 0.75 [0.33, 1.56] & 0.54 [0.26, 1.16] & 1.09 [1.09, 1.09] & 1.05 [1.05, 1.05] \\
             & Time (min) & 62.9 & 2.6 & 2.2 & 4.1 & 4.1 \\
        \\ \bottomrule
    \end{tabular}
    \caption{Performance results, parameter estimates, and computing time for Bayesian GP regression with $n \in \{100,300,500,1000\}$ data points with mean $f(x) = \sin{2x} + \frac{1}{8}e^x$ and precision $\tau=2$. Time shown is total time for setup and 7,000 iterations through Gibbs sampler, with the first 2,000 samples discarded and every tenth after retained.}\label{tab:sim_largeN}
\end{table}

\end{landscape}

\subsubsection{Large-$n$ simulation results}\label{SIsec:large_n_fifa}

Figures \ref{fig:simLargeNmspe} through \ref{fig:simLargeNtimings} show performance results, $\tau$ estimate and 95\% CI limits, and computing time for FIFA-GP regression with $n$ up to 100,000. Observations have mean $f(x) = \sin{2x} + \frac{1}{8}e^x$ and precision $\tau=2$. Time shown includes setup and 7,000 iterations through Gibbs sampler, with the first 2,000 samples discarded and every tenth after retained. For all $n,$ the tolerance is set to $\epsilon_{\max}=10^{-14}.$

We see that as $n$ increases, both the accuracy and the precision of the function estimates increase (i.e., the MSPE and the area of the function 95\% credible interval approach 0). The precision estimate $\hat{\tau}$ approaches the truth, and the 95\% credible interval width narrows. Finally, even with $n=$ 100,000 the sampler only takes 200 minutes to run (on a several-years-old MacBook laptop).

\begin{figure}[htp]
    \centering
    \includegraphics[width=0.95\textwidth]{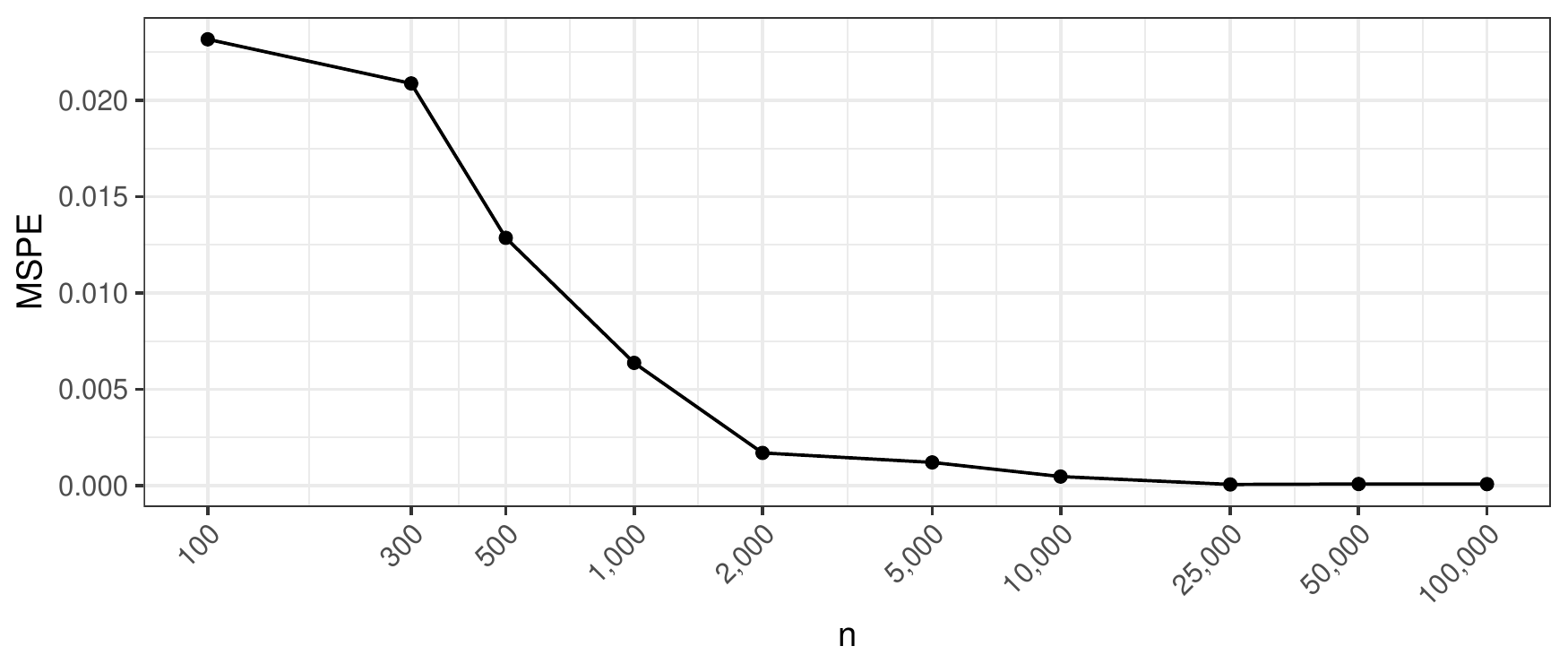}
    \caption{MSPE of hold-out data for large $n$ simulation.}
    \label{fig:simLargeNmspe}
\end{figure}

\begin{figure}[htp]
    \centering
    \includegraphics[width=0.95\textwidth]{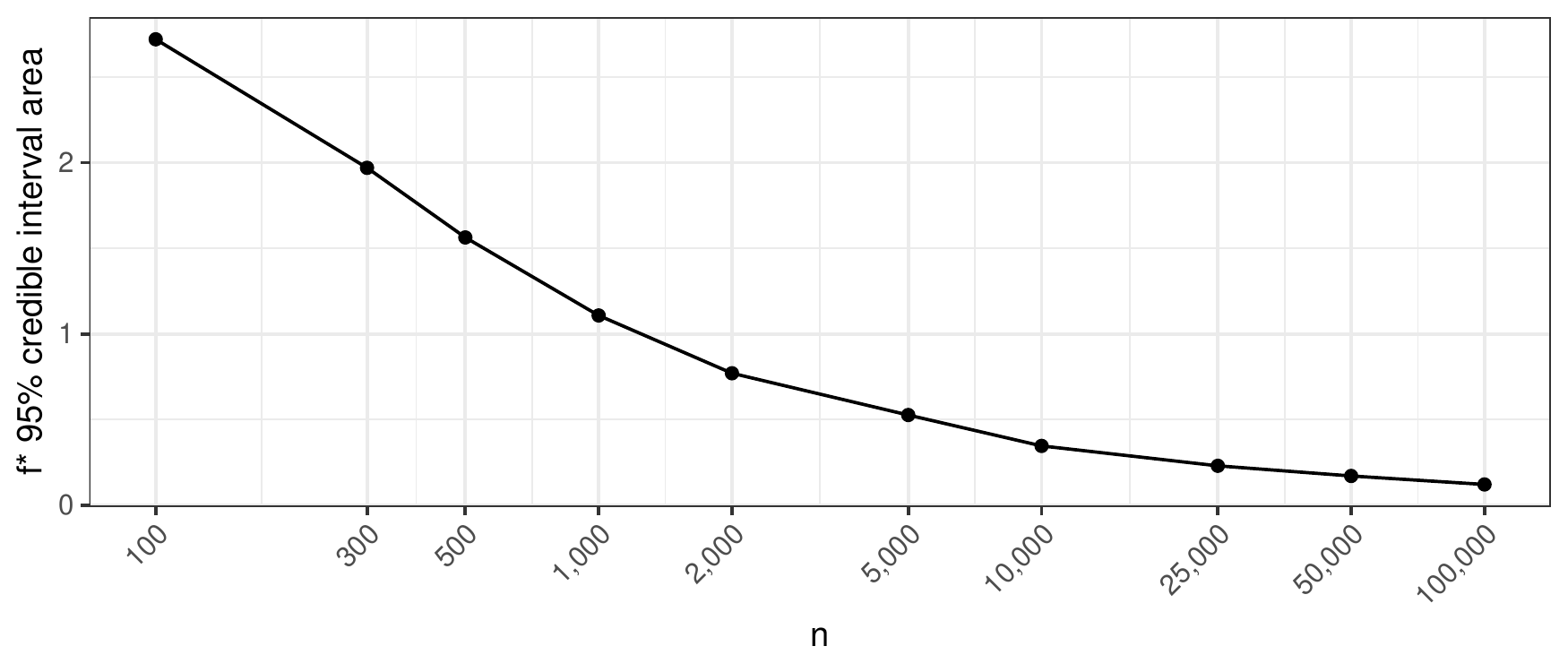}
    \caption{Area of 95\% credible interval around $\bm{f}^*$ for large $n$ simulation.}
    \label{fig:simLargeNarea}
\end{figure}

\begin{figure}[htp]
    \centering
    \includegraphics[width=0.95\textwidth]{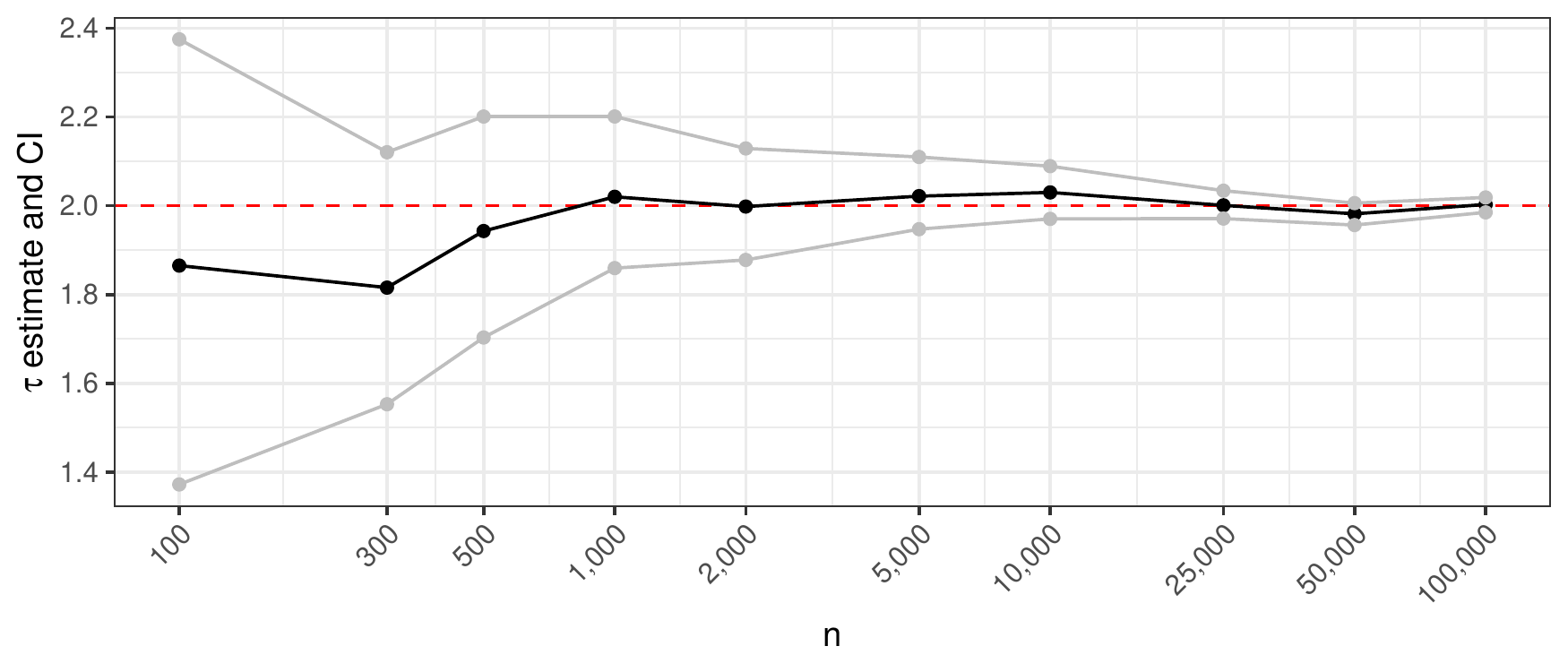}
    \caption{Estimate and credible interval of $\tau$ for large $n$ simulation.}
    \label{fig:simLargeNtau}
\end{figure}

\begin{figure}[htp]
    \centering
    \includegraphics[width=0.95\textwidth]{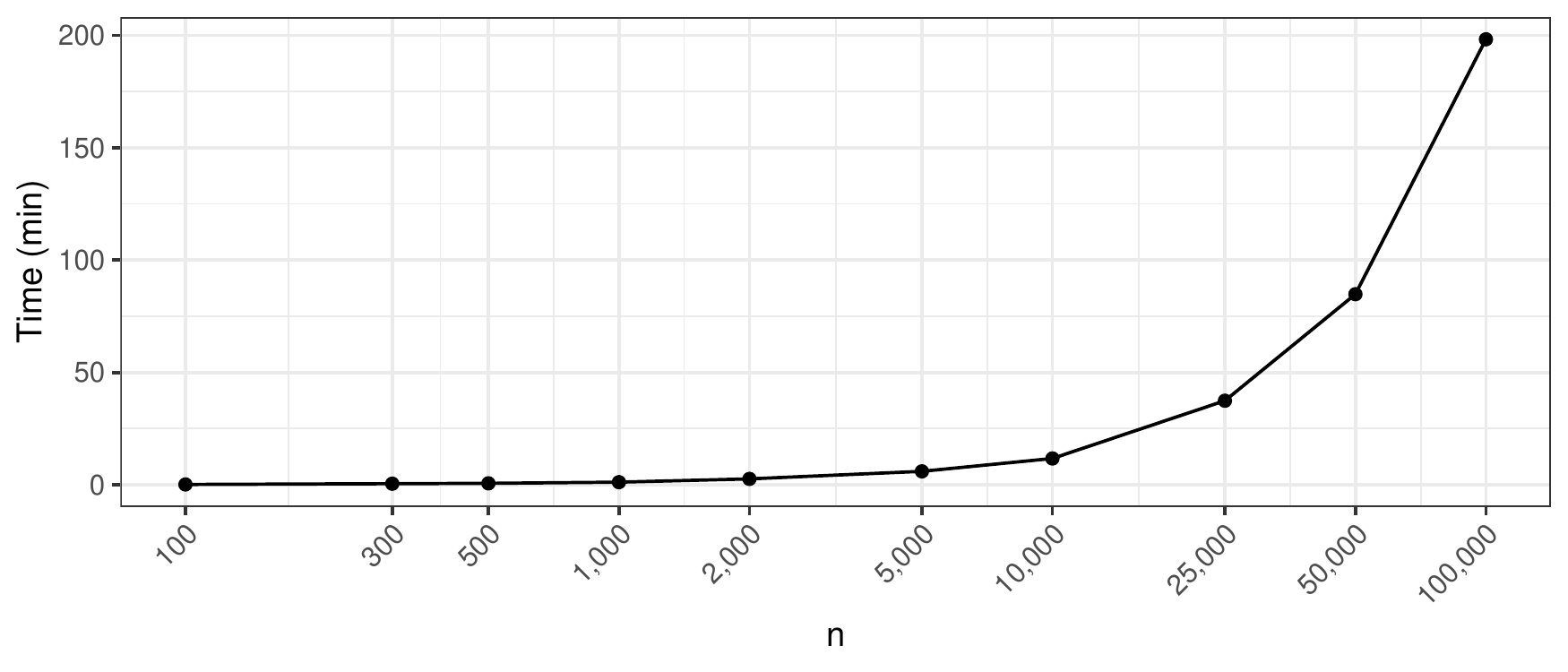}
    \caption{Timing for large $n$ simulation.}
    \label{fig:simLargeNtimings}
\end{figure}

\end{document}